\documentclass{article}
\usepackage{amsmath}
\usepackage{amsthm}
\usepackage{amssymb}
\usepackage{amsfonts}

\usepackage[english]{babel}
\usepackage{wrapfig}
\usepackage{graphicx}
\usepackage{geometry}
\usepackage[rightcaption]{sidecap}
\usepackage{tikz}
\usepackage{bbm}
\usepackage{url}
\usepackage{marvosym}
\usepackage{wasysym}
\usepackage{pifont}
\usepackage{mathrsfs}
\usepackage{caption}
\DeclareMathAlphabet{\mathpzc}{OT1}{pzc}{m}{it}

\allowdisplaybreaks[2]
\author{Jingchen Hu}

\newtheorem{theorem}{Theorem}[section]
\newtheorem{problem}[theorem]{Problem}

\newtheorem{proposition}[theorem]{Proposition}

\newtheorem{remark}[theorem]{Remark}

\newtheorem{lemma}[theorem]{Lemma}
\newtheorem{definition}[theorem]{Definition}
\numberwithin{equation}{section}
\newcommand{\ER}{\mathbb{R}}

\newcommand{\EC}{\mathbb{C}}

\newcommand{\MR}{\mathcal{R}} 

\newcommand{\MV}{\mathcal{V}} 
\newcommand{\MH}{\mathcal{H}}

\newcommand{\MG}{\mathcal{G}}
\newcommand{\ML}{\mathcal{L}}

\newcommand{\ddbar}{\partial\overline\partial}
\newcommand{\tautaubar}{{\tau\overline{\tau}}}

\newcommand{\taubar}{{\overline{\tau}}}

\newcommand{\abbar}{{\alpha\overline{\beta}}}

\newcommand{\ijbar}{{i\overline j}}
\newcommand{\dzijbar}{dz^i\otimes\overline{dz^j}}

\newcommand{\gammabar}{{\overline{\gamma}}}
\newcommand{\alphabar}{{\overline{\alpha}}}

\newcommand{\betabar}{{\overline{\beta}}}
\newcommand{\mubar}{{\overline{\mu}}}
\newcommand{\nubar}{{\overline{\nu}}}
\newcommand{\zetabar}{{\overline{\zeta}}}

\newcommand{\etabar}{{\overline{\eta}}}

\newcommand{\zbetabar}{\overline{z^{\beta}}}
\newcommand{\jbar}{{\overline{j}}}
\newcommand{\ibar}{{\overline{i}}}

\newcommand{\zzbar}{{z\overline{z}}}
\newcommand{\zbar}{{\overline z}}

\newcommand{\tr}{\text{tr}}

\newcommand{\Bbar}{{\overline{B}}}
\newcommand{\MP}{\mathcal{P}}
\newcommand{\MU}{{\mathcal{U}}}
\newcommand{\MF}{\mathcal{F}}

\newcommand{\diag}{\text{diag}}
\newcommand{\Abarinverse}{\overline{A^{-1}}}

\newcommand{\Ainverse}{{A^{-1}}}

\newcommand{\iu}{\sqrt{-1}}
\newcommand{\Ree}{\text{Re}}

\newcommand{\Imm}{\text{Im}}

\newcommand{\bb}{\mathcal{G}}
\newcommand{\spaceQ}{{\mathscr{Q}}}
\newcommand{\cQ}{{\mathcal{Q}}}
\newcommand{\ca}{{\mathfrak{a}}}
\newcommand{\cb}{{\mathfrak{b}}}
\newcommand{\cQphi}{{\mathcal{Q}_\Phi}}
\newcommand{\caphi}{{\mathfrak{a}_\Phi}}
\newcommand{\cbphi}{{\mathfrak{b}_\Phi}}

\newcommand{\Phitauzbar}{\Phi_{\tau\overline z}}
\newcommand{\Phiztaubar}{\Phi_{z\overline \tau}}
\newcommand{\Phitauz}{\Phi_{\tau z}}
\newcommand{\Phitautaubar}{\Phi_{\tau\overline \tau}}
\newcommand{\Phizzbar}{\Phi_{z\overline z}}
\newcommand{\Phizz}{\Phi_{z z}}
\newcommand{\Phitautau}{\Phi_{\tau \tau}}
\newcommand{\zetazetabar}{{\zeta\zetabar}}
\newcommand{\slopev}{\mathfrak{v}_\theta}

\newcommand{\eqe}{{\varepsilon}}
\newcommand{\ese}{{\epsilon}}

\newcommand{\qbar}{{\overline{q}}}
\newcommand{\pbar}{{\overline{p}}}
\newcommand{\vbar}{{\overline{v}}}
\newcommand{\ubar}{{\overline{u}}}

\newcommand{\sbar}{{\overline{s}}}

\newcommand{\kbar}{{\overline{k}}}
\newcommand{\lbar}{{\overline{l}}}

\newcommand{\bbform}{{\omega_{\mathcal{G}}}}
\newcommand{\MY}{\mathcal{Y}}
\newcommand{\nbar}{{\overline{n}}}
\newcommand{\nnbar}{{n\overline{n}}}

\newcommand{\zjbar}{\overline{z^j}}

\newcommand{\cConvexity}{$\EC$-Convexity}
\newcommand{\cconvex}{$\EC$-convex}

\newcommand{\scconvex}{strongly $\EC$-convex}
\newcommand{\cconvexity}{$\EC$-convexity}
\newcommand{\scconvexity}{strong $\EC$-convexity}
\newcommand{\sqcconvex}{strongly $\mathfrak{q}\EC$-convex}

\newcommand{\qcconvexity}{$\mathfrak{q}\EC$-convexity}
\newcommand{\sqcconvexity}{strong \qcconvexity}

\newcommand{\MA}{{\mathcal{A}}}
\newcommand{\MB}{{\mathcal{B}}}
\newcommand{\MK}{{\mathcal{K}}}

\newcommand{\Onep}{{\text{I}_\Phi}}
\newcommand{\Twop}{{\text{II}_\Phi}}
\newcommand{\Threep}{{\text{III}_\Phi}}
\newcommand{\po}{{\bf p}}

\newcommand{\plainA}{{\mathscr{A}}}
\newcommand{\plainB}{{\mathscr{B}}}
\newcommand{\plainL}{{\mathscr{L}}}
\newcommand{\plainK}{{\mathscr{K}}}

\newcommand{\MMB}{{\mathfrak{B}}}
\newcommand{\uglb}{{C_\MR^g}}
\newcommand{\bglb}{{C_{\partial\MR}^g}}
\newcommand{\tred}{\mathfrak{t}}
\newcommand{\sigmalamb}{{\widetilde{\sigma}}}
\newcommand{\tH}{{\widetilde{\MH}}}

\newcommand{\Omegabar}{{\overline{\Omega}}}
\newcommand{\trho}{{\widetilde\rho}}
\newcommand{\ab}{{\alpha\beta}}
\newcommand{\tct}{{\widetilde{C_2}}}
\newcommand{\MT}{{\mathcal{T}}}

\newcommand{\spaceL}{{\mathscr{L}}}

\newcommand{\scs}{\sigma}
\newcommand{\plane}{{\mathfrak{l}}}
\newcommand{\na}{{\bf a}}
\newcommand{\dist}{{\text{dist}}}
\newcommand{\sigmal}{{\underline\sigma}}
\newcommand{\MW}{\mathcal{W}}

\title{A Maximum Rank Theorem for Solutions to the Homogenous Complex Monge-Amp\`ere Equation in a $\EC$-Convex Ring}

\author{Jingchen Hu}
\begin{document}
\maketitle

\begin{abstract}
	Suppose $\Omega_0,\Omega_1$ are two bounded strongly \cconvex\ domains in $\EC^n$, with $n\geq 2$ and  $\Omega_1\supset\overline{\Omega_0}$. Let $\MR=\Omega_1\backslash\overline{\Omega_0}$. We call $\MR$ a \cconvex \ ring. We will show that for a solution  $\Phi$ to the homogenous complex Monge-Amp\`ere equation in $\MR$, with $\Phi=1$ on $\partial\Omega_1$ and $\Phi=0$ on $\partial\Omega_0$, $\iu\ddbar\Phi$ has rank $n-1$ and the level sets of $\Phi$ are strongly \cconvex.
\end{abstract}
\section{Introduction}
\subsection{Related Research and Our Main Results}
The convexity of level sets of the solution to a partial differential equation in a ring-shaped domain has been studied for a long time. In general, the following Dirichlet problem is considered:
\begin{problem}
	[The Dirichlet Problem in a Ring-Shaped Domain]
	\label{prob:General_Problem_ring}
	Suppose $\Omega_0, \Omega_1$ are bounded domains with smooth boundaries in $\ER^n$, with $\overline{\Omega_0}\subset \Omega_1$.  Let $\MR=\Omega_1\backslash\overline{\Omega_0}$ and let $\MF$ be a partial differential operator on $\MR$. Find $\Phi$ satisfying
	\begin{align}
		\MF(\Phi)=0\ \ \ \ \ \ &\text{ in }\MR,\\
   		\Phi=0\ \ \ \ \ \  &\text{ on }\partial \Omega_0,\\
		\Phi=1\ \ \ \ \ \  &\text{ on }\partial \Omega_1.
	\end{align}
	
\end{problem}
Let $\Omega_t=\Omega_0\cup\{\Phi<t\}$, for $t\in(0,1)$. We ask does $\Omega_t$ inherit some properties from $\Omega_0$ and $\Omega_1$? In many cases, the answer is affirmative. 

One particular case, which attracts many attentions, is when $\Omega_0$ and $\Omega_1$ are both convex and $\MF$ is the Laplacian (or the $p$-Laplacian).  In this case, we can show $\Omega_t$ are strongly convex for $t\in(0,1)$, and there are several different approaches. 

 First, there is a ``macroscopic" approach, where a ``macroscopic" auxiliary function is introduced to gauge the convexity. For example, extending the idea of
 \cite{Gabriel_Newzealand_ExtendedPrincipleMaxi} \cite{GabrielNewzealand_3d_harmonic} \cite{GabrielNewzealand_3d_harmonic_Third}, in \cite{Lewis_Convexity_pLaplace} the following function is considered:
\begin{align}
	\tilde \Phi(x)=\inf_{x\in \  \overline{yz}}\max\{\Phi(y), \Phi(z)\},
\end{align}
where the infimum is taken over all $y, z \in \ER^n$ such that $x$ lies on the closed line segment from $y$ to $z$. Here $\Phi$ is considered to be a function on $\ER^n$ with $\Phi=0$ in $\Omega_0$ and $\Phi=1$ in $\Omega_1^c$.  It's easy to see that the convexity of the level sets is equivalent to $\tilde \Phi=\Phi$.
In \cite{RosayRudin1989Michigan}, Rosay and Rudin considered the following function 
\begin{align}
	Q(x,y)=\frac{\Phi(x)+\Phi(y)}{2}-\Phi(\frac{x+y}{2})-c|x-y|^2,
	\label{711105_RR}
\end{align}where $c$ is a constant. In \cite{Weikove_two_point},
 the following two-point function and its variants are considered:
\begin{align}
	Q(x,y)=(\nabla \Phi(x)-\nabla\Phi(y))\cdot (x-y);
\end{align}it's obvious that when restricted to the set 
\begin{align}
	\Sigma=\{(x,y)\in \overline{\MR}| \Phi(x)=\Phi(y)\}
\end{align}   the strong convexity of the level sets is equivalent to the positivity of $Q$.

The second approach uses the constant rank theorem, and it is often referred to as a ``microscopic" approach.
In \cite{Korevaar_convex_Ring}, it was shown that the second fundamental form of the level sets of $\Phi$ has constant rank, providing the level sets are convex; then combing with a deformation argument, the strong convexity can be proved.
The method has been studied and generalized by many authors, for example  \cite{BianGuanXuMa} \cite{GuanXu_convexity}  \cite{BianGuan09_Invention} \cite{Garbo_Weinkove_ConstantRank}.

The third approach is to directly estimate the Gaussian or principle curvatures of the level sets.  It was first proved by
\cite{Longinetti} and \cite{Curvature_LevelCurvature_Canada}
 that when $n=2$, the curvature of level curves of $\Phi$ attains minimum on $\partial\MR$. In general dimension, \cite{MaOuZhang_CPAM} proved when $\MF$ is the $p$-Laplacian and  providing $p, n$ satisfy some conditions,
 \begin{align}
 	|\nabla\Phi|^{n+1-2p}K,\ \ 
 	 |\nabla \Phi|^{1-p}K\text{\ \  or } K 
 \end{align} attain their minimums on $\partial\MR$. Here, $K$ is the Gaussian curvature of the level sets of $\Phi.$ As a consequence, the convexity of $\Omega_0$ and $\Omega_1$ implies the  convexity of $\Omega_t$. A similar  estimate for the principle curvature was proved in \cite{Zhangwei_ZhangTing_PrincipleCurvature}.

Beyond the Laplacian and the real convexity in Euclidean spaces, there are many possible directions to generalize the forementioned results.

First, we may ask if $\ER^n$ is replaced by a space form, can we have similar results. This was first proved by \cite{Papadimi} in the case of $n=2$ in a Poincare disc and in the case of general dimension by  \cite{XinanMa_ZhangyongBing}.

Another question is when $\MF(\Phi)=0$ is the minimal surface equation, does the convexity of $\Omega_0$ and $\Omega_1$ imply the convexity of $\Omega_t$. The answer is affirmative. Let $S$ be a minimal surface in the 3 dimensional Euclidean space bounded by two plane
curves $r_1, r_2$, lying in parallel planes.  The following result is proved in \cite{Shiffman}:
If $r_1, r_2$ are convex curves, then the intersection of $S$ by a plane parallel
to the planes of $r_1, r_2$ is again a convex curve. In the situation of general dimension and that $S$ is a minimal graph, it was proved by \cite{Korevaar_convex_Ring}; it has also been generalized to minimal graphs over space forms by \cite{ZhangDekai_MinimalGraph}.

Using the constant rank theorems developed by \cite{BianGuanXuMa} and 
\cite{GuanXu_convexity}, we know when $\MF$ satisfies some structural conditions, the convexity of $\Omega_0$ and $\Omega_1$ implies the convexity of $\Omega_t$. This can be proved via a continuity argument. But these structure conditions exclude some important equations, for example the real and complex homogenous Monge-Amp\`ere equations. We may ask, if $\Omega_0$ and $\Omega_1$ are both strongly convex and $\MF$ is the real Monge-Amp\`ere operator, is it true that $\Omega_t$ are all strongly convex? In this case, a smooth solution can be constructed explicitly, and we can answer the question positively. Using methods from complex analysis, in \cite{Duval_deux_Convexes} and \cite{Convexity_LevelSet_PSH_extremal}  it was proved that when $\MF$ is the complex Monge-Amp\`ere operator and $\Omega_0$ and $\Omega_1$ are both convex, then $\Omega_t$ are all convex.

We may also ask, instead of the real convexity, when $\Omega_0$ and $\Omega_1$ satisfy other convexity conditions, does $\Omega_t$ satisfy the same convexity condition? For example, in complex spaces the following concepts  are natural generalizations of the convexity in real spaces:
\begin{definition}
	[$\EC$-Convexity and Linear Convexity \cite{EncycolopediaC_Convexity}\cite{Zelenski}\cite{Book_Cconvexity}\cite{JaquePhdThesis}]
	In $\EC^n$, a domain $E$ is said to be \cconvex\ if for any complex line $l\in\EC^n$, the intersection $E\cap l$ is both connected and simply connected; it is called linearly convex if $E^c$ is a union of complex hyperplanes.
\end{definition}
\begin{remark}
	When $n=1$, any connected and simply connected domain in $\EC$ is \cconvex;  in this case, a hyperplane is a point, so any domain in $\EC$ is considered linearly convex. When $n>1$ and $E$ is a bounded domain with a $C^1$ boundary, the two notions are equivalent.
\end{remark}

We can ask when $\Omega_0$ and $\Omega_1$ are both \cconvex\ with smooth boundaries and $\MF$ is some linear or non-linear operator, is $\Omega_t$ \cconvex?  Or more generally, we may consider the geometric convexity introduced by \cite{HarveyLawson_GeometricConvexity}. It seems there has been no result in this direction.

In this paper, we are concerned with the situation where $\MF$ is the  complex Monge-Amp\`ere operator and $\Omega_0$ and $\Omega_1$ are both \scconvex. The definition of \scconvexity\ is in the following:
\begin{definition}
	Suppose $E$ is a bounded connected domain in $\EC^n$ with a $C^{1,1}$ boundary. Then, at any point $\po\in\partial E$, $\partial E$ can be locally represented as the graph of a $C^{1,1}$ function:
	after a linear change of coordinates, which makes $\po=0$, we can find $\delta>0$ and a $C^{1,1}$ function $\rho_\po$, so that
	\begin{align}
		&\partial E\cap 
		\left\{
		|\Ree(z^n) | <\delta, |\Imm(z^n)|^2+\sum_{\alpha=1}^{n-1}|z^\alpha|^2<\delta^2
		\right\}\\
		=&
		\left\{
		(z^\alpha, \rho_\po(z^\alpha, s)+\iu s)\big| \sum_{\alpha=1}^{n-1}|z^\alpha|^2+s^2<\delta^2
		\right\},
		\end{align}
		with $\rho_\po(0)=0$ and $\nabla \rho_\po(0)=0$. We say $E$ is \scconvex\ if for any $\po$ we can find $\mu>0$, so that
		\begin{align}
			\rho_\po(z^\alpha,0)\geq \mu\sum_{\alpha=1}^{n-1} |z^\alpha|^2.\label{0514111}
		\end{align}
\end{definition}

\begin{remark}
	If a bounded connected domain in $\EC^n$ with a $C^{1,1}$ boundary is \scconvex, then it's \cconvex; this can be proved by Proposition 2.5.8 of \cite{Book_Cconvexity}. We also note that in (\ref{0514111}) if we require 
		\begin{align}
		\rho_\po(z^\alpha, s)\geq \mu (\sum_{\alpha=1}^{n-1}|z^\alpha|^2+s^2),
	\end{align}then this is just the usual notion of strong convexity.
\end{remark}

We study the following Dirichlet problem:
\begin{problem}
	[The Dirichlet Problem for the Homogenous Complex Monge-Amp\`ere Equation]
	\label{prob:Main_Problem_HCMA_ring}
	Suppose $\Omega_0, \Omega_1$ are two strongly \cconvex \ domains with smooth boundaries in $\EC^n$, $n\geq 2$, with $\overline{\Omega_0}\subset \Omega_1$.  Let $\MR=\Omega_1\backslash\overline{\Omega_0}$. Find a continuous $\Phi : \ \overline\MR\rightarrow \ER$, plurisubharmonic in $\MR$,  satisfying
	\begin{align}
		\left(\sqrt{-1}\ddbar\Phi\right)^n=0\ \ \ \ \ \ &\text{ in }\MR,
		\label{eq:main_HCMA_equation_form}
		\\
		\Phi=0\ \ \ \ \ \  &\text{ on }\partial \Omega_0,\\
		\Phi=1\ \ \ \ \ \  &\text{ on }\partial \Omega_1.
	\end{align}
\end{problem}

The homogenous complex Monge-Amp\`ere equation in a ring shaped domain has been studied by many authors, including \cite{MoriyonProceeding} \cite{MoriyonMadrid} \cite{GuanAnnuals}.  In particular, it was proved in \cite{GuanAnnuals} that Problem \ref{prob:Main_Problem_HCMA_ring} has a $C^{1,1}$ solution providing it has a subsolution, which essentially depends on the fact that $\Omega_0$ is holomorphically convex in $\Omega_1$. This condition is satisfied in our case since $\Omega_0$ is strongly \cconvex\ and, therefore, polynomially convex; we explain this in Appendix \ref{sec:B_boundary_Barrier}.

\begin{remark} 	Since the solution $\Phi$ to the Dirichlet problem above may only be a $C^{1,1}$ function, its level sets $\Omega_t$ may only have $C^{1,1}$ boundaries. Therefore it's necessary to introduce the \scconvexity\ for domains with $C^{1,1}$ boundaries.
\end{remark}

\begin{remark} The Dirichlet problems in a ring-shaped domain for other nonlinear equations has also been studied. For example, the real homogenous sigma-$k$ equation was studied in \cite{Ma_Zhang_Dekai_RealSigmak_Ring}; the complex homogenous sigma-$k$ equation was studied in \cite{Gao_MA_Zhang_complex_Sigmk}. Using the estimates in a ring-shaped domain, they derived the estimates for the corresponding Green's function.
\end{remark}

We  have the following result:
\begin{theorem}
	[\cConvexity\ of Subevel Sets and the Maximum Rank Property]
	\label{theorem:main_Convexity_HCMA}
	Suppose $\Phi$ is a solution to Problem \ref{prob:Main_Problem_HCMA_ring}. Then $|\nabla\Phi|$ has a positive lower bound, and $\sqrt{-1}\ddbar\Phi$ has rank $n-1$ in the weak sense. Let
	\begin{align}
		\Omega_t=\{\Phi\leq t\}\cup \Omega_0, \ \ \ \ \ \  \text{for} \ t\in(0,1).
	\end{align}
	Then $\Omega_t$ are all strongly \cconvex.
\end{theorem}
\begin{remark}
	Since we only know $\Phi$ is $C^{1,1}$,  $\text{rank}(\iu\ddbar\Phi)=n-1 $ should be understood in the weak sense: In Section \ref{sec:Approximate_HCMA},  we can show 
	\begin{align}
		\left(
		\iu\ddbar\Phi
		\right)^{n-1}\wedge \omega_0 > 0
		\label{623117}
	\end{align} in the sense of distribution, where 
	\begin{align}
		\omega_0 =\iu \delta_{\ijbar}dz^i\wedge \overline{dz^j}.
	\end{align} (\ref{623117}) says the rank of $\sqrt{-1}\partial\overline{\partial}\Phi$ is at least $n-1$ in the weak sense, and (\ref{eq:main_HCMA_equation_form}) says the  the rank of $\sqrt{-1}\partial\overline{\partial}\Phi$ is at most $n-1$, also in the weak sense. Therefore, we say the rank of $\sqrt{-1}\partial\overline{\partial}\Phi$ is  $n-1$ in the weak sense.
	
	In addition, we also prove that 
	\begin{align}
		\iu\ddbar e^{\Phi} >0,
	\end{align}i.e. $e^\Phi$ is strictly plurisubharmonic.
\end{remark}
\begin{remark}
	In the case of $n=1$, a \cconvex \ domain is a connected and simply connected domain, and there is no essential content to the name \scconvex\ domain. The Monge-Amp\`ere operator in this case is simply the Laplacian. Actually, we can prove the following result. Suppose $\Omega_0$ and $\Omega_1$ are both connected and simply connected domains in $\EC$, with smooth boundary. Then $\nabla \Phi\neq 0$ everywhere, and $\Omega_t$ are all connected and simply connected. The proof is easy; we only need to consider the function
	$z\Phi_z$, and notice that \begin{align}
		(z\Phi_z)_{\zzbar}=0\ \ \ \ \  \ \ \ \text{ on }\EC,
	\end{align}where $z$ is the coordinate on $\EC$. When $\Omega_0$ and $\Omega_1$ are concentric balls, $z\Phi_z$ is non-zero. Also, $z\Phi_z$ can not be zero on the $\partial\Omega_0$ and $\partial\Omega_1$ because of the boundary gradient estimate. Therefore, we can prove the result using a deformation argument.
\end{remark}

Another closely related topic is the convexity of level sets of Green's functions. The earliest result known to the author is Caratheodory's proof that level curves of the Green's function for a bounded convex sets in the plane are convex; the proof can be found in the book of Ahlfors \cite{AhlforsBook}. The $3$-dimensional case was proved  by Gabriel 
\cite{GabrielNewzealand_3d_harmonic}.
In the case of general $n$ dimension, suppose $G$ is the Green's function in $\Omega$ with $G<0$ in $\Omega$ and $G=0$ on $\partial\Omega$; Jia-Ma-Shi proved that $(-G)^{\frac{1}{2-n}}$ is  strictly convex in $\Omega\backslash \{0\}$ for $n\geq 3$, where $0$ is the singular point of $G$ \cite{Shishujun_Green_LevelSets}  \cite{Jia_Ma_Shi_2022}. In the case of pluricomplex Green's function, it was proved by Lempert (page 462 of \cite{Lempert81_La_metrique}, Lemma 5.3 of \cite{Lempert85_Duke_Symmetries}) that in a strongly convex (or strongly linearly convex) domain with a smooth boundary the sublevel sets of  pluricomplex Green's functions are convex (or strongly linearly convex). It's interesting  that in this case level sets of Green's functions are geodesic balls with respect to the Kobayashi metric.
 
 
There has been a large amount of research on the convexity of level sets of solutions of  partial differential equations; we refer to \cite{Xinan_Ou_Summary_English} \cite{XinanMa_Chinese} \cite{ChuanQiangMa} for better surveys.
 \subsection{Main Ideas of the Proof}
 First, we observe that on $\EC^{n-1}$, a quadratic function 
 \begin{align}
 	f(z)=A_{\alpha\betabar}z^{\alpha}z^\betabar+\Ree\left(
 	B_{\alpha\beta}z^\alpha z^\beta
 	\right),
 	\label{711121}
 \end{align}with $\alpha, \beta\in \{1,\ ...\ ,n-1\}$,
 is strongly convex if and only  if $A>0$ and 
 \begin{align}
 	\text{ the maximum eigenvalue of  } B\overline{A^{-1}}\ \overline{B}A^{-1}<1.
 \end{align}
 Here, for strongly convex we mean the real  Hessian of $f$ is strictly positive definite.
 This is the main idea of \cite{Hu22Nov} and \cite{HuC2Perturb}, where we proved that in a product  manifold, providing that the boundary value satisfies a convexity condition, the solution to the homogenous complex Monge-Amp\`ere equation has Hessian of the maximum rank. In this paper, we apply this idea to the study of the convexity of level sets of solutions to the homogenous complex Monge-Amp\`ere equation in a ring-shaped domain in $\EC^n$.
 
 Suppose $\nabla \Phi\neq 0$ at a point $\po$. We choose a set of coordinates $\left\{
 z^\alpha,\tau
 \right\}$, where $\alpha, \beta\in \{1,\ ...\ ,n-1\}$ and $\tau $ is the $n$-th coordinate, so that $\Phi_\alpha=0$ and $\Phi_\tau\neq0$  at $\po$ .  Then, at $\po$, we let
 \begin{align} \label{516120_reduced}
 	\MA_{\alpha\betabar}=\Phi_{\alpha\betabar},
 	\ \ \ \ \ \ \ 
 	 	\MB_{\alpha\beta}=\Phi_{\alpha\beta}.
 \end{align}
 When the level sets of $\Phi$ are strongly pseudoconvex, we have $\MA>0$, and we let
 \begin{align}
 	 \MK_{\alpha}^\theta=\MB_{\alpha\mu}\MA^{\mu\etabar}\ \overline{\MB_{\eta\zeta}}\MA^{\theta\zetabar}.
 \end{align}
Using matrix notation, denoting $\MA=(\MA_{\abbar}),\ \MB=(\MB_{\ab})$ and $\MK=(\MK_{\alpha}^\beta)$,  the above is equivalent to
 \begin{align}
 	\MK=\MB\overline{\MA^{-1}}\ \overline{\MB}\MA^{-1}.
 \end{align} 
 When the maximum eigenvalues of $\MK$ is smaller than 1, we let
 \begin{align}
 	\sigma_\Phi=\tr(I-\MK)^{-1}.
 \end{align}
 We note that $\MA$, $\MB$ and $\MK$ all depend on the choice of coordinate, but the eigenvalues of $\MK$ are invariant under a complex linear change of coordinates; therefore,  $\sigma_\Phi$
is invariant under a complex linear change of coordinates. So $\sigma_\Phi$ is a well-defined function in $\overline{\MR}$. The upper bound of $\sigma_\Phi$ provides an estimate on the \cconvexity\  of the level sets of $\Phi$. How to quantitatively gauge the \cconvexity\ will be discussed in Section \ref{sec:AuxiliaryFunction_Introduce}.

If we assume the solution to Problem \ref{prob:Main_Problem_HCMA_ring} is smooth and do a formal computation, which is what we did in Section \ref{sec:Formal_Computation_2dim}, we will find 
\begin{align}
	\sigma_\Phi\leq \max_{\partial\MR}\sigma_\Phi \ \ \ \ \  \ \ \ \text{ in }\MR.
\end{align} We can consider the above estimate as a weak version of the constant rank theorem: providing the level sets of $\Phi$ are \scconvex, we have an estimate for the \cconvexity. 

However, the solution $\Phi$ to  Problem \ref{prob:Main_Problem_HCMA_ring} is only $C^{1,1}$ in general. So we need to consider a perturbed Dirichlet problem, which has a smooth solution. 
Instead of considering the usually used non-degenerate Monge-Amp\`ere equation
\begin{align}
	\det(\Phi_{\ijbar})=\eqe,
\end{align} we consider the following Hessian quotient equation:
\begin{problem}
	[The Dirichlet Problem for a Hessian Quotient Equation]
	\label{prob:Perturbed_Problem_HessianQuotient}  Let $\MG_{\ijbar}dz^i\otimes \overline{dz^j}$ be a constant coefficient Hermitian metric on $\EC^n$; then, $\bbform =\sqrt{-1}\bb_{ij}dz^i\wedge \overline{dz^j}$ is the K\"ahler form.
	Find $\Phi$ satisfying
	\begin{align}
		\left(\sqrt{-1}\ddbar\Phi\right)^n&=\eqe	\left(\sqrt{-1}\ddbar\Phi\right)^{n-1}\wedge\bbform \ \ \ \ \ \ &\text{in }\MR,\ \ \label{eq:SigmaQuotient_Problem_Perturbation}\\
		\sqrt{-1}\ddbar\Phi&>0\ \ \ \ \ \ &\text{in }\MR,\ \ 
		\label{cond:positive_ddbar}\\
		\Phi&=0\ \ \ \ \ \  &\text{\ \ on }\partial \Omega_0,\\
		\Phi&=1\ \ \ \ \ \  &\text{\ \ on }\partial \Omega_1.
	\end{align}
Here $\eqe$ is a positive constant.
\end{problem}
It was proved by \cite{GuanAnnuals} that the problem above has a subsolution providing $\overline{\Omega_0}$ is holomorphically convex in $\Omega_1$; in Appendix \ref{sec:B_boundary_Barrier}, we argue that this condition is satisfied in our situation.
Then applying Theorem 1.2 of \cite{GuanLi2012}, we know this problem has a unique and smooth solution. Using the method of \cite{Guan98}, we know the solution $\Phi$ has a uniform $C^{2}$ estimate which is independent of the lower bound of $\eqe$; we include a proof of this fact in Appendix \ref{sec:Appendix_C11Estimate} since it's not clearly stated elsewhere to the knowledge of the author. So if $\Phi^\epsilon$ is the solution to the problem above, we know $\Phi^\epsilon$ converges to a solution to Problem \ref{prob:Main_Problem_HCMA_ring} as $\epsilon$ goes to $0$.  Therefore, an estimate for the solution to Problem \ref{prob:Main_Problem_HCMA_ring} can be derived from that of Problem \ref{prob:Perturbed_Problem_HessianQuotient}.

\subsection{Notation and Conventions of the Paper}
\label{sec:notation}
In this paper, we will only allow complex linear change of coordinates on $\EC^n$. We use Roman letters to denote indices run from $1$ to $n$. Sometimes, we need to specialize an $n-1$ dimensional complex subspace. For example, we may specialize the complex tangent plane to the level sets of a function. In this situation, we use the Greek letters, except $\tau$, to denote indices go from $1$ to $n-1$; the letter $\tau$ is used to denote the $n$-th coordinate.

We will introduce a convenient notation to indicate terms in equations. The term (*.*)$_k$ indicate the $k$-th term in equation (*.*), including the sign. For example, 
\begin{align}
	(\ref{711121})_2=A_{\alpha\betabar}z^{\alpha}z^\betabar,\ \ \ \ \  \ \ \ 
		(\ref{711105_RR})_4=-c|x-y|^2.
\end{align}

For matrices with real eigenvalues, we denote that
\begin{align}
	c_1>M (c_1\geq M)\ \ \text{ or } M>c_2 (M\geq c_2)
\end{align} if the maximum eigenvalues of $M<c_1(\leq c_1)$  or the minimum eigenvalues of $M>c_2(\geq c_2)$.

Suppose $\Phi$ is a smooth function with $\nabla\Phi\neq0$ around a point $\po$. Then the level set $\{\Phi=\Phi(\po)\}$ is a smooth hypersurface around $\po$. We denote the complex tangent plane of  $\{\Phi=\Phi(\po)\}$ at $\po$ by 
\begin{align}
	T^\EC_{\Phi, \po}.
\end{align}

For a ring-shaped domain $\MR=\Omega_1\backslash\overline{\Omega_0}$, we call
\begin{align}
	\min
	\left\{
	\min_{x\in\partial\Omega_0, y\in\partial \Omega_1}|x-y|,\ 
	\max
	\left\{
	r|B_r(z)\subset\overline{\MR} \text{ for some }z\in\MR 
	\right\}
	\right\}
\end{align} the thickness of $\MR$. Let $\Phi^\eqe$ be the solution to Problem \ref{prob:Perturbed_Problem_HessianQuotient}, and let $\MG_\ijbar dz^i\otimes \overline{dz^j}$ be the Hermitian metric on $\EC^n$. If a constant $C$ depends on 
\begin{itemize}
	\item the modulus of \cconvexity\ of $\MR$,
	\item the diameter of $\MR$,
	\item the thickness of $\MR$,
	\item  the upper bound of $\max_{\overline\MR}|D^2\Phi^\eqe|_\MG$, providing $\eqe$ small enough, derived in Appendix \ref{sec:Appendix_C11Estimate},
	\item  the lower bound of $\min_{\partial \MR} |\nabla \Phi^\eqe|_{\MG}$, providing $\eqe$ small enough, derived in Appendix \ref{sec:B_boundary_Barrier},
\end{itemize} we say $C$ depends on the geometry of $\MR$.
\section{The $\EC$-Convexity and The Auxiliary Functions}
\label{sec:AuxiliaryFunction_Introduce}

In this section, we will introduce some concepts which quantitatively characterize the \cconvexity\  of sublevel sets of $\Phi$ and introduce the auxiliary functions.
\subsection{Quantitative Characterizations of the $\EC$-Convexity}
First, we introduce the convexity characterization of quadratic polynomials on $\EC^{n-1}$.
\begin{definition}
	Suppose that $g_{\alpha\betabar}dz^\alpha\otimes dz^\betabar$ is a constant coefficient Hermitian metric on $\EC^{n-1}$. Then we say the {\bf modulus of convexity} of a quadratic polynomial
	\begin{align}
		P(z)=A_{{\alpha\betabar}}z^{\alpha} z^{\betabar}+\Ree \left(
		B_{\alpha\beta} z^\alpha z^{\beta}
		\right)+L
	\end{align}
	is greater than $\mu$, for $\mu\geq 0$, if 
	\begin{align}
		P(z)-L>\mu \left(
		g_{\alpha\betabar} z^\alpha z^\betabar  
		\right)\ \ \ \ \  \ \ \  \text{ on } \EC^{n-1}\backslash 
		\left\{
		0
		\right\}.
	\end{align}Here $A$ is a Hermitian matrix, $B$ is a complex valued symmetric matrix, and $L$ is a linear function. We say that $P$ is strongly convex if and only if the modulus of convexity of $P$ is greater than $0$.
	
 We say the {\bf degree of convexity} of $P$
 is greater than $\delta$, for $\delta\geq 0$, if for any
 \begin{align}
 	\MW(z)=\Ree
 	\left(
 	W_{{\alpha\beta}}z^\alpha z^\beta
 	\right),
 \end{align} 
 with 
 \begin{align}\left(
 	W_{\alpha\beta} g^{\beta\gammabar}\ \overline{W_{\gamma\nu}} g^{\theta\nubar}
 	\right)\leq \delta^2,
 \end{align}
 $P-\MW$ is strongly convex.
\end{definition}
\begin{remark}
	The degree and the modulus of convexity are equivalent; this is proved by Lemma A.2 of \cite{Hu22Nov}. In Appendix  \ref{sec:Algebra_Lemma}, we will provide an independent proof: Lemma \ref{lemma:Equivalent_Quadratic_Modulus_Degree}.
\end{remark}
 In literature,  a function with convex level sets is often called quasi-convex. Therefore, we introduce the following concept of strongly quasi-\cconvexity:
\begin{definition}
	Suppose that $\Omega$ is a bounded domain with smooth boundaries and that $\Phi$ is a smooth function defined on $\overline{\Omega}$. For a point $\po\in\overline{\Omega}$, we say $\Phi$ is {\bf strongly quasi-\cconvex} at $\po$ if 
	\begin{align}
		\nabla{\Phi}(\po)\neq 0
	\end{align}
	and 
	\begin{align}\text{the second order Taylor expansion of }
		\Phi|_{T^\EC_{\Phi, \po}} \text{ at $\po$ is strongly convex.}
	\end{align} Here, 
	$	{T^\EC_{\Phi, \po}} $ is the complex tangent plane of $\left\{
	\Phi=\Phi(\po)
	\right\}$ at $\po$. We say $\Phi$ is strongly quasi-\cconvex\ on $\Gamma\subset\overline{\Omega}$ if it's strongly quasi-\cconvex \ at any $\po\in\Gamma$.
\end{definition}
\begin{remark}The main goal of this paper is to prove the solution $\Phi$ to Problem \ref{prob:Perturbed_Problem_HessianQuotient} is strongly quasi-\cconvex.
	In the following, for convenience, we will denote quasi-\cconvexity \ by {\bf \qcconvexity}. 
\end{remark}
To gauge the \qcconvexity, we introduce the following concepts:
\begin{definition}
	[Modulus of \qcconvexity]
	Suppose $g_\ijbar dz^i\otimes \overline{dz^j}$ is a constant coefficient metric on $\EC^n$. 
	We say the {\bf modulus of \qcconvexity } of $\Phi$ is greater than $\delta\geq0$ if
	\begin{align}
		|\nabla \Phi|_g >\delta
	\end{align}
	and 
		\begin{align}
			\begin{split}
			\text{the modulus of convexity of the second order Taylor expansion}&\\
			\text{ of }
		\Phi|_{T^\EC_{\Phi, \po}} & \text{ at $\po$ is greater than $\delta$.}
			\end{split}
	\end{align} 
\end{definition}
\begin{definition}
	[Robustness of \qcconvexity]
	Suppose $g_\ijbar dz^i\otimes \overline{dz^j}$ is a constant coefficient Hermitian metric on $\EC^n$.  Let
	\begin{equation}
		\spaceQ_\epsilon=
		\left\{
		q\  | \ \text{pluriharmonic quadratic polynomials with }|\nabla q|_g\leq \epsilon
		\right\}.
	\end{equation}
	We say the {\bf robustness of \qcconvexity} of $\Phi$ is greater than $\epsilon\geq0$ if
	for any $q\in \spaceQ_\epsilon$, $\Phi-q$ is \sqcconvex. In this paper, when we say quadratic functions we include linear functions.
\end{definition}
\begin{remark}
	An easy and useful observation is 
	\begin{align}
		\spaceQ_{\eqe_1}+\spaceQ_{\eqe_2}=\spaceQ_{\eqe_1+\eqe_2}.
	\end{align}So if the robustness of \qcconvexity\ of $\Phi$ is greater than $\eqe$ and $q\in\spaceQ_\delta$ with $\delta<\eqe$, then the robustness of \qcconvexity\ of $\Phi-q$ is greater than $\eqe-\delta$.
\end{remark}
\begin{remark} The name robustness means how robust the \qcconvexity\ is under perturbations. 
	If the modulus of \qcconvexity\ of $\Phi$ is greater than $0$, than the robustness of \qcconvexity\ is also greater than $0$, and vise versa; the two measurements are actually comparable in the sense of Lemma \ref{lemma:ModulusPositive_implies_RobustnessPositive} and Lemma \ref{lemma:RobustnessPositive_implies_ModulusPositive}.
\end{remark}

For the quantitative characterization of the \cconvexity\ of a domain, we have the following concept of the modulus of \cconvexity:
\begin{definition}[Modulus of \cconvexity\ of a Domain]	Suppose $\Omega$ is a bounded domain in $\EC^n$ with a $C^{1,1}$ boundary and that $ dz^i\otimes \overline{dz^i}$ is the metric on $\EC^n$. 
	Then, at any point $\po\in\partial E$, $\partial E$ can be locally represented as the graph of a $C^{1,1}$ function:
	after a unitary change of coordinates, which makes $\po=0$, we can find $\delta>0$ and a $C^{1,1}$ function $\rho_\po$, so that
	\begin{align}
		&\partial E\cap 
		\left\{
		|\Ree(z^n) | <\delta, |\Imm(z^n)|^2+\sum_{\alpha=1}^{n-1}|z^\alpha|^2<\delta^2
		\right\}\\
		=&
		\left\{
		(z^\alpha, \rho_\po(z^\alpha, s)+\iu s)\big| \sum_{\alpha=1}^{n-1}|z^\alpha|^2+s^2<\delta^2
		\right\},
	\end{align}
	with $\rho_\po(0)=0$ and $\nabla \rho_\po(0)=0$. We say the modulus of \cconvexity \ of $E$ is greater than $\mu\geq0$ if
	\begin{align}
		\rho_\po(z^\alpha, 0) > \mu \sum_\alpha|z^\alpha|^2,\label{0607213}
	\end{align}for $(z^\alpha)\neq 0$ when $\delta$ is small enough.
\end{definition}
\begin{remark}
	A domain is \scconvex\ if and only if its modulus of \cconvexity\ is greater than $0$. If  the inequality (\ref{0607213}) is replaced by 
	\begin{align}
		\rho_\po(z^\alpha, s)> \mu(\sum_\alpha |z^\alpha|^2+s^2),
	\end{align} then $\mu$ can be any number smaller than the minimum principle curvature of $\partial \Omega$. In Section \ref{sec:Algebra_Lemma}, we will show if the modulus of \qcconvexity \ of $\Phi$ is greater than $0$, then the modulus of \cconvexity\ of level sets of $\Phi$ is greater than $0$; that's to say, the \sqcconvexity\ of $\Phi$ implies the \scconvexity\ of sublevel sets of $\Phi$.
\end{remark}
\subsection{Auxiliary Functions}
\label{sec:AuxiliaryFunctions}
Let $\Phi$ be a solution to equation (\ref{eq:SigmaQuotient_Problem_Perturbation}). We denote $\MH=(\Phi_{\ijbar})$ and $\MG=(\MG_{\ijbar})$. Then (\ref{eq:SigmaQuotient_Problem_Perturbation}) becomes
\begin{align}
	\label{eq:Matrix_Form_Main_Perturbation_Equation}
	\tr (\MH^{-1}\MG)=\frac{1}{\eqe}.
\end{align}
We apply $\partial_i$ to the equation above and get
\begin{align}
	\label{eq:Matrix_Form_Main_Perturbation_Equation_partial_i}
	\tr (\MH^{-1}\MH_i\MH^{-1}\MG)=0.
\end{align}Let
\begin{align}
	L^{p\qbar}=\eqe^2 \Phi^{p\vbar}\MG_{u\vbar} \Phi^{u\qbar};
\end{align}
then, (\ref{eq:Matrix_Form_Main_Perturbation_Equation_partial_i}) becomes
\begin{align}
	L^{p\qbar}(\Phi_{i})_{p\qbar}=0,
	\label{eq:SigmaQuotientDerivative=0}
\end{align}
 and $\ML=L^{\ijbar}\partial_{\ijbar}$ is the linearization operator of (\ref{eq:SigmaQuotient_Problem_Perturbation}). Moreover, $\ML$ has bounded coefficient, independent of  $\eqe$. Let $\MG=(\delta_{\ij})$ and $\MH=\diag(\lambda_1,\ ...\ , \lambda_n)$. Then equation (\ref{eq:Matrix_Form_Main_Perturbation_Equation}) gives
\begin{align}
	\sum_{i=1}^n\frac{1}{\lambda_i}=\frac{1}{\eqe},
\end{align} so $\lambda_i>\eqe$ for any $i$. With this coordinate, 
$L^{\ijbar}=\delta_{ij} \frac{\eqe^2}{\lambda_i^2}<1$. 
In the following, we will construct some quantities related to the estimates and apply $\ML$ to them.

For the lower bound estimate of $|\nabla\Phi|_{\MG}$, we let
\begin{align}
	S=\eqe \Phi_i\Phi_{\jbar}\Phi^{\ijbar};
\end{align}we can show $|\nabla \Phi|_{\MG}\geq S$. So a positive lower bound for $S$ implies a positive lower bound for $|\nabla\Phi|_{\MG}$.

For the convexity estimate, we construct a tensor $\MK$.
If $\Phi$ is \sqcconvex, we know $\nabla \Phi\neq 0$ anywhere. For a point $\po\in\overline{\MR}$, we choose a set of coordinates $\left\{
z^\alpha,\tau
\right\}$, where $\alpha, \beta\in \{1,\ ...\ ,n-1\}$ and $\tau $ is the $n$-th coordinate, so that $\Phi_\alpha=0$ and $\Phi_\tau\neq0$  at $\po$ .  Then, around $\po$, we let
\begin{align} \label{516120}
	\MA_{\alpha\betabar}=\Phi_{\alpha\betabar} -\frac{\Phi_{\alpha}}{\Phi_\tau}\Phi_{\tau\betabar}
	-\frac{\Phi_{\betabar}}{\Phi_\taubar}\Phi_{\alpha\taubar}
	+\frac{\Phi_{\alpha}\Phi_{\betabar}}{|\Phi_{\tau}|^2}\Phi_{\tautaubar},\\
	\label{516121}
	\MB_{\alpha\beta}=\Phi_{\alpha\beta} -\frac{\Phi_{\alpha}}{\Phi_\tau}\Phi_{\tau\beta}
	-\frac{\Phi_{\beta}}{\Phi_\tau}\Phi_{\alpha\tau}
	+\frac{\Phi_{\alpha}\Phi_{\beta}}{(\Phi_{\tau})^2}\Phi_{\tau\tau}.
\end{align} 
Obviously, at $\po$ we have
\begin{align}
	\MA_{\abbar}=\Phi_{\abbar},\ \ \ \ \  \ \ \ \MB_{\ab}=\Phi_{\ab}.
\end{align}
Let $\MA, \MB$ be the matrices
\begin{align}
	\MA=(\MA_{\abbar}),\ \ \ \ \  \ \ \ \MB=(\MB_{\ab}).
\end{align}
Since $\Phi$ is \sqcconvex, the second order Taylor expansion of $\Phi|_{{T^\EC_{\Phi,\po}}}$ at $\po$ is strongly convex; then, we have $\MA>0$ around $\po$, and we let
\begin{align}
	\MK_{\alpha}^\theta=\MB_{\alpha\mu}\MA^{\mu\etabar}\ \overline{\MB_{\eta\zeta}}\MA^{\theta\zetabar}.
\end{align}
Using matrix notation, the above is equivalent to
\begin{align}
	\MK=\MB\overline{\MA^{-1}}\ \overline{\MB}\MA^{-1},
\end{align} where $\MK$ is the matrix $(\MK_{\alpha}^\theta)$. $\MK$ can also be considered as a tensor: at $\po$, $\MK$ is a complex linear automorphism of ${T^\EC_{\Phi,\po}}$. Therefore, eigenvalues of $\MK$ are independent of the choice of coordinate. 

Using Lemma \ref{lemma:Robustness_positive_implies_K<1}, we know that the modulus of \qcconvexity\ of $\Phi$ is greater than $\delta>0$ implies $\MK<1-\frac{\delta}{C}$, where $C$ is a constant; using Lemma \ref{lemma_Metric_and_Q_convert_to_Degree}, we know $\MA>\delta$ and $\MK<1-\delta$ implies that the modulus of \qcconvexity\ of $\Phi$ is greater than $\frac{\delta^2}{C}$ for a constant $C$.  So an estimate for the minimum eigenvalue of $I-\MK$ implies an estimate for the robustness of \qcconvexity\ of $\Phi$.

In the following, we will provide an upper bound estimate for $\tr(I-\MK)^{-1}$, which implies a positive lower bound estimate for the minimum eigenvalue of $I-\MK$. Actually, we can show the estimate is still valid if we add a pluriharmonic quadratic polynomial to $\Phi$; this provides an estimate for the robustness of \qcconvexity\ of $\Phi$.
\section{The Main Computation Result}
The main goal of this section is to prove the following computation results. These results allow us to use the maximum principle to derive the corresponding estimates in Section \ref{sec:Lower_Bound_Gradient_Norm} and Section \ref{sec:Convexity_Estimate_GeneralDim_Def+Perturb}.
\begin{proposition}
	\label{prop:Main_Computation}
	Suppose that $\Phi$ is a solution to equation (\ref{eq:SigmaQuotient_Problem_Perturbation}) and that
	\begin{align}
		\text{the robustness of \qcconvexity \ of $\Phi$ is greater than $\delta$. }
		\label{Assumption_robust_ComputationProposition}
	\end{align} Then if $\eqe$ is small enough, depending on $\delta$, we have
	\begin{align}
	&	L^{p\qbar}\left(
		-\log(S)+\sqrt{\eqe}\left[W+\MG_{\ijbar}z^i\zjbar\right]
		\right)_{p\qbar}\geq 0,
		\label{eq:-logs_main_computation}\\
&		L^{p\qbar}\left(
		\tr\left[I-\MK\right]^{-1}+\sqrt{\eqe}(W+\MG_{\ijbar}z^i\zjbar)
		   													+\eqe^{\frac{3}{4}}V
		\right)_{p\qbar}\geq 0.
		\label{eq:sigma_Q_main_computation}
	\end{align}
	In above, $L,\ S$ and $\MK$ are introduced in section \ref{sec:AuxiliaryFunctions}, and 
	$W=\Phi_{\ijbar}\MG^{\ijbar}$, $V=\Phi_{\ij}\MG^{j\kbar}\ \overline{\Phi_{kl}}\MG^{i\lbar}$.
\end{proposition}
Through this whole section, we only need to do the computation at a point $\po$. To simplify the computation, we choose  coordinates so that
$\Phi_{\alpha}(\po)=0$, for $\alpha\in \{1,\ ...\ ,n-1\}$, and $(\Phi_{\ijbar})(\po)$ is diagonal. 
Actually, we can make $\Phi_{\abbar}=\delta_{\alpha\beta}$ and $\Phi_\tau=1$, where $\tau$ denotes the $n$-th coordinate. This will be explained in the next paragraph.
We note that after this choice of coordinate, $\MG$ cannot keep diagonal; we can find $C_\delta>0$ so that
	\begin{align}
		C_\delta I\geq \MG \geq \frac{1}{C_\delta} I,
		\label{51539}
	\end{align} 
where $I$ is the identity matrix and $C_\delta$ depends on $\delta$ and the $C^2$ norm of $\Phi$. 
To prove this, we need to use $(\Phi_{\alpha\betabar})>\delta (\MG_{\alpha\betabar})$ and $|\Phi_\tau|_\MG>\delta$, which follow from the apriori assumption (\ref{Assumption_robust_ComputationProposition}).

At a point $\po$, we choose coordinates $\{\widetilde{z^\alpha}, \tilde{\tau}\}$, so that
$\Phi_{\widetilde{z^\alpha}}=0$, $\Phi_{\tilde{\tau}}\neq0$, $\Phi_{{\widetilde{z^\alpha}}\overline{\widetilde{z^\beta}}}=\delta_{\alpha\beta}\lambda_\alpha$ and
	$g_{\ijbar}=\delta_{ij}$, with $\lambda_\alpha>0$. This can be done because of the apriori assumption (\ref{Assumption_robust_ComputationProposition}).
Then we find $\{z^\alpha, \tau\}$, with
\begin{align}
	&{\widetilde{z^\alpha}}=A_\mu^\alpha z^\mu +B^\alpha \tau,\\
	&\widetilde{\tau}=E_\gamma z^\gamma+D\tau.
\end{align}
 so that
 \begin{align}
 	\Phi_{\alpha}=0, \label{condition36}\\
 	\Phi_\tau=1, \label{condition37}\\
 	\Phi_{\alpha\betabar}=\delta_{\alpha\beta} \label{condition38}
 \end{align}
and
\begin{align}
	\left(
	\Phi_{\ijbar} 
	\right)\text{ is diagonal.} \label{condition_H_diagonal}
\end{align}
Condition (\ref{condition36}) requires
\begin{align}
	0=\Phi_\alpha=\Phi_{{\widetilde{z^\mu}}}
	\frac{\partial {\widetilde{z^\mu}}}{\partial {{z^\alpha}}} +\Phi_{\widetilde{\tau}} 
		\frac{\partial {\widetilde{\tau}}}{\partial {{z^\alpha}}} =\Phi_{\widetilde{\tau}} E_\alpha,
\end{align}so we choose $E_\alpha=0$. Condition (\ref{condition37}) requires
\begin{align}
	1=\Phi_\tau=\Phi_{{\widetilde{z^\mu}}}
	\frac{\partial {\widetilde{z^\mu}}}{\partial {{\tau}}} +\Phi_{\widetilde{\tau}} 
	\frac{\partial {\widetilde{\tau}}}{\partial {{\tau}}} =\Phi_{\widetilde{\tau}} D,
\end{align}so we choose $D=\frac{1}{\Phi_{\widetilde{\tau}}}$.
Condition (\ref{condition38}) requires 
\begin{align}
	\delta_{\alpha\beta}=\Phi_{\abbar}=
	\Phi_{\widetilde{z^\mu}\overline{\widetilde{z^\gamma}}} 
	A^{\mu}_{\alpha} \overline{\left(
		A^\gamma_\beta
		\right)}.
\end{align}Such a matrix $A$ can be find since 
$\left(	\Phi_{\widetilde{z^\mu}\overline{\widetilde{z^\gamma}}}\right)>\delta I $, according to (\ref{Assumption_robust_ComputationProposition}). 
Condition (\ref{condition_H_diagonal}) requires 
\begin{align}
	0=\Phi_{{{z^\alpha}}\overline{ {\tau}}}=
	\left(
	\Phi_{\widetilde{z^\mu}}\frac{\partial  {\widetilde{z^\mu}}}{ \partial z^\alpha}
	\right)_\taubar
	=
	\Phi_{ {\widetilde{z^\mu}} \overline{\widetilde{z^\gamma}} }A^\mu_\alpha \overline{B^\gamma} +\Phi_{{\widetilde{z^\mu}} \overline{ \widetilde{\tau}} }A^\mu_\alpha \overline D. 
\end{align}Since $\Phi_{ {\widetilde{z^\mu}} \overline{\widetilde{z^\gamma}} }=\delta_{\mu\gamma}\lambda_\mu$, we choose 
 \begin{align}
 	B^\gamma=-D \frac{\Phi_{ \overline{\widetilde{z^\mu}} { \widetilde{\tau}}}}{\lambda_\mu}.
 \end{align}
Since the constants $A, B, D, E$ only depend on $\delta$ and the second order derivatives of $\Phi$, the estimate (\ref{51539}) is valid.

Denoting $\Phi_{\tau\taubar}=T$,  the equation becomes
\begin{align}
	\frac{\MG_{\tautaubar}}{T}+\sum_\alpha \MG_{\alpha\alphabar}=\frac{1}{\eqe}.
\end{align}
Therefore $T$ is comparable to $\eqe$, providing $\eqe$ is small enough.
More precisely, we have
\begin{align}
	\frac{\eqe}{C_{\delta}}\leq T \leq  \eqe C_{\delta},
	\label{Phitt_comparable_epsilon}
\end{align}providing $\eqe\leq \frac{1}{ C_{\delta}}$ for a large enough $C_{\delta}$ depending on $\delta$ and the $C^2$ norm of $\Phi$.

During the computation, most importantly, we need to control third derivatives of $\Phi$. In Section \ref{sec:Computing_Gradient} and  \ref{sec:subsec:Convexity_Estimate_computation}, we will find  quadratic forms $\Onep, \Twop$ of third derivatives  and use them to control third derivatives: 
we will prove
\begin{align}
	|\Phi_{\alpha\beta\gammabar}|\eqe+	|\Phi_{\alpha\beta\taubar}|	+|\Phi_{\alpha\betabar\tau}|
	+	\frac{1}{\eqe}|\Phi_{\alpha\tau\taubar}|+\frac{1}{\eqe}|\Phi_{\tau\tau\taubar}|
	\leq C_\delta \sqrt{\Onep} \label{51534}
\end{align}
and
\begin{align}
	|\Phi_{\alpha\beta\tau}|+|\Phi_{\alpha\tau\tau}|+|\Phi_{\tau\tau\tau}|\leq C_\delta (C_\delta  \Onep+ \Twop)^{\frac{1}{2}}.
	\label{610317}
\end{align} Here, $C_\delta$ is a constant depending on the $\delta$  of (\ref{Assumption_robust_ComputationProposition}) and the $C^2$ norm of $\Phi$.
For clarity, we list the estimates for third derivatives in the following table. 

\begin{center}
	\begin{tabular}{|c|c|c|}
		\hline
		&Pure Holomorphic Derivatives& Mixed Derivatives\\
		\hline
		Containing No $\tau$ & $\Phi_{\alpha\beta\gamma} $ Not Used& $\eqe|\Phi_{\alpha\beta\gammabar}|\leq C_\delta \Onep^{\frac{1}{2}}$\\
		\hline
	Containing One $\tau$ & $|\Phi_{\alpha\beta\tau}|\leq C_\delta (C_\delta  \Onep+ \Twop)^{\frac{1}{2}} $  & $|\Phi_{\alpha\beta\taubar}|+|\Phi_{\alpha\betabar\tau}|\leq C_\delta \Onep^{\frac{1}{2}}$\\
		\hline
	Containing Two $\tau$'s & $|\Phi_{\alpha\tau\tau}|\leq C_\delta (C_\delta  \Onep+  \Twop)^{\frac{1}{2}} $ & $\frac{1}{\eqe}|\Phi_{\alpha\tau\taubar}|+|\Phi_{\alphabar\tau\tau}|\leq C_\delta \Onep^{\frac{1}{2}}$\\
			\hline
	Containing Three $\tau$'s & $|\Phi_{\tau\tau\tau}|\leq C_\delta (C_\delta \Onep+ \Twop)^{\frac{1}{2}} $ & $|\Phi_{\tau\tau\taubar}|\leq \eqe C_\delta \Onep^{\frac{1}{2}}$\\
	\hline
	\end{tabular}
	\captionof{table}{Third Derivatives}
\label{tab:Three_Group_Third_Derivatives}  
\end{center}

\subsection{Proving (\ref{eq:-logs_main_computation})}
\label{sec:Computing_Gradient}
In this section, we prove  (\ref{eq:-logs_main_computation}). First, we compute $L^{p\qbar}W_{p\qbar}$; this will produce a positive quadratic form, which can be used to control third derivatives of mixed type.  Actually, we will see the only third derivatives involved in section \ref{sec:Computing_Gradient} are of mixed type.

We apply $\partial_{\jbar}$ to  (\ref{eq:Matrix_Form_Main_Perturbation_Equation_partial_i}) and get
\begin{align}
	\tr \left(
	\MH^{-1} \MH_{\ijbar}\MH^{-1} \MG	\right)
	=\tr\left(
	\MH^{-1} \MH_i \MH^{-1} \MH_{\jbar}\MH^{-1} \MG+
	\MH^{-1} \MH_\jbar \MH^{-1} \MH_{i}\MH^{-1} \MG
	\right).
\end{align}
Using index summation and operator $\ML$, this becomes
\begin{align} 
	L^{p\qbar}(\Phi_{\ijbar})_{p\qbar}=
		L^{p\qbar} \Phi_{p\vbar i}\Phi^{u\vbar} \Phi_{\qbar u\jbar}+
			L^{p\qbar} \Phi_{p\vbar \jbar}\Phi^{u\vbar} \Phi_{\qbar u i}.
			\label{eq:SigmaQuotient_double_derivative_indexform}
\end{align}
Since $\MG^{\ijbar} $ are constants, we get
\begin{align} \label{51537}
		L^{p\qbar}W_{p\qbar}=
	L^{p\qbar} \Phi_{p\vbar i}\Phi^{u\vbar} \Phi_{\qbar u\jbar}\MG^{\ijbar}+
	L^{p\qbar} \Phi_{p\vbar \jbar}\Phi^{u\vbar} \Phi_{\qbar u i}\MG^{\ijbar}.
\end{align}
We denote
\begin{align}
	L^{p\qbar}W_{p\qbar}=\Onep.
\end{align}

Using (\ref{51539}) and the expression of $L^{\ijbar}$, we know
\begin{align}
	\eqe^2 C_\delta \MH^{-2}\geq (L^{\ijbar}) \geq \eqe^2 \frac{1}{C_\delta} \MH^{-2}.
	\label{711324}
\end{align}
We plug (\ref{711324}) and (\ref{51539}) into (\ref{51537}) and get
\begin{align}
	\Onep\geq \frac{\eqe^2}{ C_{\delta}}\delta_{ij}\delta_{uv}\delta_{pq}\Phi^{u\ubar} \Phi^{p\pbar}\Phi^{q\qbar}  \left(
	\Phi_{p\vbar i}\Phi_{\qbar u\jbar}
	+\Phi_{p\vbar \jbar} \Phi_{\qbar u i}
	\right).
\end{align} Choosing $p=i=u=\tau$, we get
\begin{align}
	|\Phi_{\tau\tau\taubar}|^2\leq  \eqe C_{\delta} \Onep;
\end{align}
choosing $p=i=\tau,\ u=\alpha $, we get
 \begin{align}
 		|\Phi_{\alpha\tau\taubar}|^2+	|\Phi_{\alphabar\tau\tau}|^2\leq   C_{\delta} \Onep;
 		\label{516313}
 \end{align}
choosing $p=\tau,\ i=\beta, \ u=\alpha $, we get
\begin{align}\label{516314}
	|\Phi_{\alpha\beta\taubar}|^2+	|\Phi_{\alpha\betabar\tau}|^2\leq   C_{\delta} \Onep;
\end{align}
choosing $p=\gamma,\ i=\beta,\ u=\alpha $, we get
\begin{align}\label{516315}
	|\Phi_{\alpha\beta\gammabar}|^2\leq \frac{ 1}{\eqe^2}   C_{\delta} \Onep.
\end{align}
However, the estimates for $\Phi_{\tau\tautaubar}$ and $\Phi_{\alpha\tautaubar}$ are weaker than we want (comparing to (\ref{51534})).  To get a better estimate, we need to use the equation of $\Phi$. 
In (\ref{eq:Matrix_Form_Main_Perturbation_Equation_partial_i}), we let $i=\nu\in \{1,\ ...\ ,n-1\}$ and get
\begin{align}
	\label{711330}
	\frac{\eqe^2}{T^2}\MG_{\tautaubar} \Phi_{\nu\tau\taubar} 
	+\frac{\eqe^2}{T}\MG_{\alpha\taubar} \Phi_{\nu\tau\alphabar}  
	+\frac{\eqe^2}{T}\MG_{\tau\alphabar} \Phi_{\nu\taubar\alpha}  
	+\eqe^2\MG_{\alpha\mubar} \Phi_{\nu\mu\alphabar}  =0.
\end{align}
According to the estimates (\ref{516314}) and (\ref{516315}), we have
\begin{align}
	|\Phi_{\nu\tau\alphabar}  |+|\Phi_{\nu\taubar\alpha}  |+\eqe|\Phi_{\nu\mu\alphabar}  |\leq 
 C_{\delta}	\sqrt{\Onep}.
\end{align}Plugging these and the estimate (\ref{Phitt_comparable_epsilon}) into (\ref{711330}), when $\eqe$ is small enough, we have
\begin{align}
	|\Phi_{\nu\tau\taubar} |\leq C_{\delta}\eqe \sqrt{\Onep}.\label{516318}
\end{align}
Similarly, in (\ref{eq:Matrix_Form_Main_Perturbation_Equation_partial_i}), we let $i=n$ and get
\begin{align}
    \frac{\eqe^2}{T^2}\MG_{\tautaubar} \Phi_{\tau\tau\taubar} 
    +\frac{\eqe^2}{T}\MG_{\alpha\taubar} \Phi_{\tau\tau\alphabar}  
      +\frac{\eqe^2}{T}\MG_{\tau\alphabar} \Phi_{\tau\taubar\alpha}  
        +\eqe^2\MG_{\alpha\mubar} \Phi_{\tau\mu\alphabar}  =0.
\end{align}
Then using estimates (\ref{516313}) (\ref{516314}) and (\ref{516318}), we get
\begin{align}
		|\Phi_{\tau\tau\taubar} |\leq C_{\delta}\eqe \sqrt{\Onep}.\label{516320}
\end{align}

The term $\MG_{\ijbar} z^i\zjbar$ in  (\ref{eq:-logs_main_computation}) is used to produce a positive constant term. We have
\begin{align}
	L^{p\qbar}(\MG_{\ijbar} z^i\zjbar )_{p\qbar}=L^{p\qbar}\MG_{p\qbar}
	=\eqe^2 \tr\left(
     \MG	\MH^{-1} \MG	\MH^{-1}
	\right).
\end{align} 
We apply the Cauchy inequality and get
\begin{align}
	 \tr\left(
	\MG	\MH^{-1} \MG	\MH^{-1}
	\right) \geq \frac{1}{n} 
	\tr(\MG \MH^{-1})^2=\frac{1}{n\eqe^2}.
\end{align}In the last equality, we used equation (\ref{eq:Matrix_Form_Main_Perturbation_Equation}). Therefore, $	L^{p\qbar}(\MG_{\ijbar} z^i\zjbar )_{p\qbar}\geq \frac{1}{n}$.

Then we compute $L^{p\qbar}\left(
-\log(S)
\right)$; we will show it's greater than $-C_{\delta}\eqe(\Onep+1)$. Directly apply $\ML$ to $-\log(S)$, and we get
\begin{align}
	L^{p\qbar}(-\log (S))_{p\qbar}=-\frac{L^{p\qbar }{S_{p\qbar}}}{S}+\frac{L^{p\qbar} S_p S_\qbar}{S^2}.\label{516323}
\end{align}
In the following, we compute the two terms on the right-hand side of (\ref{516323}) separately. First of all, we notice that according to the choice of coordinate, at the point $\po$
\begin{align}
	S= \eqe \Phi_\tau\Phi^{\tautaubar} \Phi_{\taubar}=\frac{\eqe}{T}.
\end{align}

We apply $\partial_p$ to $S$ and get
\begin{align}
	S_p=\eqe\Phi_{ip} \Phi^{\ijbar} \Phi_{\jbar} -\eqe \Phi_i \Phi^{i\vbar} \Phi_{u\vbar p} \Phi^{u\jbar} \Phi_{\jbar}+\eqe \Phi_p.
	\label{516Sp}
\end{align}
At the point $\po$, the expression is simpler:
\begin{align}
	S_p=\frac{\eqe}{T}\Phi_{\tau p} -\frac{\eqe}{T^2}  \Phi_{\tautaubar p} +\eqe \Phi_p.
\end{align}
So, at $\po$, we have
\begin{align}
	\frac{L^{p\qbar} S_pS_\qbar}{S^2}= L^{p\qbar}
	 (\Phi_{\tau p} -\frac{1}{T}  \Phi_{\tautaubar p} +T \Phi_p)
	(\Phi_{\taubar \qbar} -\frac{1}{T}  \Phi_{\tautaubar \qbar} +T \Phi_{\qbar}).
\end{align}
According to the estimate (\ref{51534}) and the $C^1, C^2$ estimate of $\Phi$, 
$\Phi_{\tau p} -\frac{1}{T}  \Phi_{\tautaubar p} +T \Phi_p$ is bounded for any $p$. We also notice that $|L^{p\qbar}|$ is smaller than  $C_{\delta}\eqe$ when $p$ or $q$ is not $n$. So, we have
\begin{align}
	\frac{L^{p\qbar} S_pS_\qbar}{S^2}\geq &
	L^{\tautaubar}
	(\Phi_{\tau \tau} -\frac{1}{T}  \Phi_{\tau\tautaubar} +T )
	(\Phi_{\taubar \taubar} -\frac{1}{T}  \Phi_{\tautaubar \taubar} +T)
	-C_{\delta}\eqe\\
	\geq & L^{\tautaubar}
	(\Phi_{\tau \tau} -\frac{1}{T}  \Phi_{\tau\tautaubar})
	(\Phi_{\taubar \taubar} -\frac{1}{T}  \Phi_{\tautaubar \taubar} )	-C_{\delta}\eqe.
	\label{516329}
\end{align}

Then we compute $\frac{L^{p\qbar} S_{p\qbar}}{S}$. We apply $\frac{1}{S}L^{p\qbar}\partial_{\qbar}$ to the three terms on the right-hand side of (\ref{516Sp}), then sum over $p$. For the first term we have
\begin{align}
\frac{1}{S}L^{p\qbar}(	\eqe\Phi_{ip} \Phi^{\ijbar} \Phi_{\jbar})_{\qbar}
=L^{p\qbar} \Phi_{\tau p\qbar} -L^{p\qbar}\Phi_{ip} \Phi^{i\vbar} \Phi_{\tau\vbar\qbar}
+T L^{p\qbar}  \Phi_{ip} \Phi^{\ijbar} \Phi_{\qbar\jbar}.
\label{516331}
\end{align}
In above, we used that 
$\Phi_{\tautaubar}=T$, $S=\frac{\eqe}{T}$,  $\Phi_\tau=1$ and  $\Phi_{\alpha}=0$. We also observe that the first term on the right-hand side of (\ref{516331}) is $0$ because $L^{\ijbar}\partial_{\ijbar}$ is the linearized operator.

Applying $\frac{1}{S}L^{p\qbar}\partial_{\qbar}$ to the second term  on the right-hand side of (\ref{516Sp}), we get
\begin{align}
\frac{1}{S}L^{p\qbar}(	 -\eqe \Phi_i \Phi^{i\vbar} \Phi_{u\vbar p} \Phi^{u\jbar} \Phi_{\jbar})_{\qbar}=&-L^{p\qbar}\Phi_{\tau p\qbar}
 +\frac{1}{T}L^{p\qbar} \Phi_{\taubar s\qbar}\Phi^{s\vbar} \Phi_{\tau\vbar p}
 -\frac{1}{T} L^{p\qbar}\Phi_{\tautaubar p \qbar} 
 \label{516331new}
 \\
 &\ \ \ \ \  \ \ \ \ \ \ \ \  \ \ \ +\frac{1}{T}L^{p\qbar}  \Phi_{u\taubar p} \Phi^{u\kbar}\Phi_{\tau\kbar \qbar}
 - L^{p\qbar}\Phi_{\taubar u p} \Phi^{u \jbar}  \Phi_{\jbar \qbar}\label{516332}.
\end{align}
Similar to (\ref{516331}), in above we used $\Phi_\alpha=0$, $\Phi_\tau=1$ and $\Phi^{\tautaubar}=\frac{1}{T}$; we also have (\ref{516331new})$_2=0$ because $L^{\ijbar}\partial_{\ijbar}$ is the linearized operator. Here we used the convention introduced in Section \ref{sec:notation}.
We let the $i,\jbar$ of (\ref{eq:SigmaQuotient_double_derivative_indexform}) be $\tau, \taubar$ and get
\begin{align}
	L^{p\qbar}\Phi_{\tautaubar p\qbar}=
	L^{p\qbar} \Phi_{p\vbar \tau}\Phi^{u\vbar} \Phi_{\qbar u\taubar}+
	L^{p\qbar} \Phi_{p\vbar \taubar}\Phi^{u\vbar} \Phi_{\qbar u \tau}.
\end{align}
So 
\begin{align}
(\ref{516331new})_4=
-\frac{1}{T}L^{p\qbar} \Phi_{p\vbar \tau}\Phi^{u\vbar} \Phi_{\qbar u\taubar}-
\frac{1}{T}L^{p\qbar} \Phi_{p\vbar \taubar}\Phi^{u\vbar} \Phi_{\qbar u \tau},
\label{516336}
\end{align}
and we see (\ref{516336})$_2$ cancels with $(\ref{516331new})_3$. 

Applying $\frac{1}{S}L^{p\qbar}\partial_{\qbar}$ to the last term  of (\ref{516Sp}), we get
\begin{align}
\frac{1}{S}L^{p\qbar}(	 \eqe \Phi_p)_{\qbar}=\frac{T}{\eqe} \cdot \eqe\cdot 
			L^{p\qbar}	 \Phi_{p\qbar} =T \cdot \eqe^2 \tr(\MG\MH^{-1}).
\end{align}
In above, we used the definition of $L$ and that $S=\frac{\eqe}{T}$ at $\po$.
Then, using equation (\ref{eq:Matrix_Form_Main_Perturbation_Equation}), we get
\begin{align}
T \cdot \eqe^2 \tr(\MG\MH^{-1})=T\eqe.	
\end{align}

We put the results above together and get
\begin{align}
	\frac{L^{p\qbar} S_{p\qbar}}{S} =T L^{p\qbar}
	\left(
	\Phi_{ip}		-\frac{1}{T} \Phi_{ip\taubar} 
	\right)\Phi^{\ijbar}
	\left(
	\Phi_{\qbar\jbar} -\frac{1}{T}  \Phi_{\tau\jbar\qbar}
	\right)-\frac{1}{T}L^{p\qbar} 
    \Phi_{p\vbar \taubar}\Phi^{u\vbar} \Phi_{\qbar u \tau}
    +T\eqe.
    \label{516339}
\end{align}
A careful but straightforward computation shows
\begin{align}
	(\ref{516339})_2 \leq L^{\tautaubar} 
	\left(
	\Phi_{\tau\tau}		-\frac{1}{T} \Phi_{\tau\tau\taubar} 
	\right)\Phi^{\ijbar}
	\left(
	\Phi_{\taubar\taubar} -\frac{1}{T}  \Phi_{\tau\taubar\taubar}
	\right)+\eqe C_{\delta}(\Onep+1).\label{516340}
\end{align} 
Then (\ref{516340})$_2$ can cancel with (\ref{516329})$_1$.

Summing up, we find
\begin{align}
	L^{p\qbar}\left(
	-\log(S)+\sqrt{\eqe}\left[W+\MG_{\ijbar}z^i\zjbar\right]
	\right)_{p\qbar}\geq -\eqe C_{\delta} (\Onep+1)+\sqrt{\eqe} (\Onep+\frac{1}{n}).
\end{align}Therefore, when we choose $\eqe$ small enough, the above is non-negative.
\begin{remark}
	When $\eqe$ approximates to zero, the $S$ we constructed here approximates to the $S$ of section \ref{sec:A_Lower Bound Estimate for Norms of Gradients}, and the  operator $\ML$ here approximates to the leaf-wise Laplacian there. The formal computation  of section \ref{sec:A_Lower Bound Estimate for Norms of Gradients} shows $S$ is subharmonic along leaves; the computation here suggests the same result.
\end{remark}

\subsection{Proving (\ref{eq:sigma_Q_main_computation})}
\label{sec:subsec:Convexity_Estimate_computation}
Suppose that $\plainA$ and $\plainB$ are matrix-valued functions defined on a domain in $\EC^n$, where $\plainA$ is positive definite and Hermitian and $\plainB$ is symmtric, and suppose 
\begin{align}
	\plainL= \plainL^{p\qbar}\partial_{p\qbar}
\end{align}
is an elliptic operator. If $\plainA$ and $\plainB$ satisfy
\begin{align}
	\plainL^{p\qbar} \partial_{p\qbar}\plainA
	=\plainL^{p\qbar} \partial_{p} \plainA \plainA^{-1}  \partial_{\overline q}\plainA
	   +\plainL^{p\qbar} \partial_{\qbar} \plainB\overline{ \plainA^{-1}}  \partial_{p}\overline{\plainB},
	   \label{RH_plain_A}
	   \\
	   \plainL^{p\qbar} \partial_{p\qbar}\plainB
	   =\plainL^{p\qbar} \partial_{p} \plainA \plainA^{-1}  \partial_{\overline q}\plainB
	   +\plainL^{p\qbar} \partial_{\qbar} \plainB \overline{\plainA^{-1}}  \partial_{p}\overline{\plainA},
	   \label{RH_plain_B}
\end{align}then for 
\begin{align}
	\plainK=\plainB\overline{\plainA^{-1}}\ \overline{\plainB}\plainA^{-1}
\end{align}
we have
\begin{align}
	\plainL^{p\qbar} 
\left[\tr	
	(\plainK^r)
	\right]_{p\qbar}\geq 0
\end{align}providing $\plainK<1-\frac{1}{2r}$. This is the main idea of the paper \cite{Hu22Nov}.

Actually, the following is also true:
\begin{align}
	\plainL^{p\qbar}\left[
	\tr (I-\plainK)^{-1}
	\right]_{p\qbar}\geq 0, \label{RH_plain_tr(I-K)}
\end{align} providing $\plainK<1$.   

In this paper, we generalize this idea to the study of the \cconvexity\ of level sets of $\Phi$. Instead of (\ref{RH_plain_A}) and (\ref{RH_plain_B}), we can prove, for $\MA$ and $\MB$ defined in Section \ref{sec:AuxiliaryFunctions},
\begin{align}
|L^{p\qbar} \partial_{p\qbar}\MA
	-L^{p\qbar} \partial_{p} \MA \MA^{-1}  \partial_{\overline q}\MA
	-L^{p\qbar} \partial_{\qbar} \MB\overline{ \MA^{-1}}  \partial_{p}\overline{\MB}|
	\leq 
	\eqe C_{\delta} (C_{\delta}(1+\Onep)+\Twop),
	\label{RH_curvy_A}
	\\
	|L^{p\qbar} \partial_{p\qbar}\MB
	-L^{p\qbar} \partial_{p} \MA\MA^{-1}  \partial_{\overline q}\MB
	-L^{p\qbar} \partial_{\qbar} \MB \overline{\MA^{-1}}  \partial_{p}\overline{\MA}|
	\leq
	\eqe C_{\delta} (C_{\delta}(1+\Onep)+\Twop);
	\label{RH_curvy_B}
\end{align}
then, similar to (\ref{RH_plain_tr(I-K)}), we can show
\begin{align}
	L^{p\qbar}\left[
	\tr (I-\MK)^{-1}
	\right]_{p\qbar}\geq -\eqe C_{\delta} (C_{\delta}(1+\Onep)+\Twop), \label{RH_curvey_tr(I-K)}
\end{align} providing $\MK<1$.   Here
\begin{align}
	\Twop=L^{p\qbar}
	\left(
	\Phi_{ij}\MG^{j\ubar}\ \overline{\Phi_{uv}} \MG^{i\vbar}
	\right)_{p\qbar}.
\end{align}
In the following, we denote
 \begin{align}
 	V=\Phi_{ij}\MG^{j\ubar}\ \overline{\Phi_{uv}} \MG^{i\vbar}.
 \end{align} If (\ref{RH_curvey_tr(I-K)}) is valid, we have
 \begin{align}
 	\begin{split}
 	L^{p\qbar} 
 &	\left(
 	\tr(I-\MK)^{-1}+\eqe^{\frac{1}{2}}(W+\MG_\ijbar z^i \overline{z^j})+\eqe^{\frac{3}{4}} V
 	\right)_{p\qbar}
 	\geq\\
 	& \ \ \ \ \  \ \ \ \ \ \ \ \  \ \ \ \ \ \ \ \  \ \ \ \ \ \ \ \  \ \ \ -\eqe C_{\delta} (C_{\delta}(1+\Onep)+\Twop)+\eqe^{\frac{1}{2}}(\Onep+\frac{1}{n})+\eqe^{\frac{3}{4}}\Twop.
 	\label{610366}
 	\end{split}
 \end{align}So, when $\eqe$ is small enough, we have that the right-hand side of (\ref{610366}) is non-negative; then, the maximum principle implies a positive lower bound estimate for the minimum eigenvalue of $I-\MK$.
 
 In the following, we first provide computations about $\Twop$ and prove (\ref{610317}); then, we prove (\ref{RH_curvy_A}) and (\ref{RH_curvy_B}); finally, we prove (\ref{RH_curvey_tr(I-K)}).
 
 Simply using Leibniz rule, we have
 \begin{align}
 	L^{p\qbar}V_{p\qbar}=&L^{p\qbar}\Phi_{ijp}\overline{\Phi_{uvq}}\MG^{i\vbar}\MG^{j\ubar}
 											+L^{p\qbar}\Phi_{ij\qbar}\overline{\Phi_{uv\pbar}}\MG^{i\vbar}\MG^{j\ubar}\\
 											&+L^{p\qbar} \Phi_{ijp\qbar}\overline{\Phi_{uv}}\MG^{i\vbar}\MG^{j\ubar}
 											+L^{p\qbar}\Phi_{ij}\Phi_{\ubar\vbar p\qbar} \MG^{i\vbar}\MG^{j\ubar}.
 \end{align}
Let
 \begin{align}
 	\Threep=L^{p\qbar}\Phi_{ijp}\overline{\Phi_{uvq}}\MG^{i\vbar}\MG^{j\ubar}.
 \end{align}Then
 \begin{align}
\Threep\geq \frac{1}{C_{\delta}}
\left(
|\Phi_{\tau\alpha\beta}|^2+|\Phi_{\tau\tau\alpha}|^2+|\Phi_{\tau\tau\tau}|^2
\right), 	
\label{type2_III}
 \end{align}for a constant $C_{\delta}$ large enough depending on $\delta.$ We will show
 \begin{align}
 	|L^{p\qbar} \Phi_{ijp\qbar}\overline{\Phi_{uv}}\MG^{i\vbar}\MG^{j\ubar}|\leq C_{\delta} \Onep.
 	\label{610371}
 \end{align}Then it follows that
 \begin{align}
 	\Twop+C_{\delta} \Onep\geq \Threep
 \end{align} and the estimate (\ref{610317}) is valid. For (\ref{610371}), we have
 \begin{align}
 	L^{p\qbar} \Phi_{ijp\qbar} =L^{p\qbar}\Phi_{iu\qbar} \Phi^{u\vbar} \Phi_{\vbar j p}
 														+L^{p\qbar}\Phi_{ju\qbar} \Phi^{u\vbar} \Phi_{i\vbar p}, 
 \end{align}so we only need to show 
 \begin{align}
 	|L^{p\qbar}\Phi_{iu\qbar} \Phi^{u\vbar} \Phi_{\vbar j p}|\leq C_{\delta}\Onep,
 	\label{610374}\\
 	|L^{p\qbar}\Phi_{ju\qbar} \Phi^{u\vbar} \Phi_{i\vbar p}|\leq C_{\delta}\Onep.
 	\label{610375}
 \end{align}
 We will only prove (\ref{610374}); then, by simply switching $i$ with $j$ we can get (\ref{610375}). Since $(\Phi_{\ijbar})$ is diagonal, we need to show
 \begin{align}
 	|L^{p\qbar}\Phi_{i\alpha\qbar}  \Phi_{\alphabar j p}|\leq C_{\delta}\Onep
 	\label{610376}
 \end{align}and 
 \begin{align}
 	|L^{p\qbar}\Phi_{i\tau\qbar}  \Phi_{\taubar j p}|\leq \eqe C_{\delta}\Onep.
 	\label{610377}
 \end{align}
 According to the definition of operator $L$, the left-hand side of (\ref{610376}) has the following decomposition:
 \begin{align}
 	L^{p\qbar}\Phi_{i\alpha\qbar}\Phi_{\alphabar j p}=
 	\left(
 	\eqe \Phi^{p\vbar}{ \Phi_{\alphabar j p}}
 	\right) 
 	\MG_{u\vbar}
 		\left(
 	\eqe \Phi^{u\qbar}{ \Phi_{i \alpha \qbar}}
 	\right) .
 \end{align}
 We can infer from (\ref{51534}) and (\ref{Phitt_comparable_epsilon}) that both $|	\eqe \Phi^{p\vbar}{ \Phi_{\alphabar j p}}|$ and $|	\eqe \Phi^{u\qbar}{ \Phi_{i \alpha \qbar}}|$ are smaller than $C_{\delta} \Onep^{\frac{1}{2}}$, so (\ref{610376}) is valid.  Similarly the left-hand side of (\ref{610377}) has the following decomposition
  \begin{align}
 	L^{p\qbar}\Phi_{i\tau\qbar}\Phi_{\taubar j p}=
 	\left(
 	\eqe \Phi^{p\vbar}{ \Phi_{\taubar j p}}
 	\right) 
 	\MG_{u\vbar}
 	\left(
 	\eqe \Phi^{u\qbar}{ \Phi_{i \tau \qbar}}
 	\right) .
 \end{align} We have that both
  $|\eqe \Phi^{p\vbar}{ \Phi_{\taubar j p}}|$ 
  and $|\eqe \Phi^{u\qbar}{ \Phi_{i \tau \qbar}}|$ are smaller than $\eqe C_{\delta} \Onep^{\frac{1}{2}}$, so (\ref{610377}) is valid. Therefore, (\ref{type2_III}) is valid and (\ref{610317}) is proved.
  
  To prove (\ref{RH_curvy_A}) and (\ref{RH_curvy_B}), we introduce two intermediate quantities $A$ and $B$ in a neighborhood of $\po$:  let
  \begin{align}
  	A=\Phi_{\alpha\betabar},\ \ \ \ \  \ \ \ B= \Phi_{\alpha\beta}.
  \end{align} Because of the assumption (\ref{Assumption_robust_ComputationProposition}) and that $\Phi\in C^{\infty}(\overline\MR)$, $A$ is invertible in a small neighborhood of $\po$. Here we only need to compute at $\po$ and don't need to  worry about the size of the neighborhood.
  Moreover, we point out that $A=\MA$ and $B=\MB$ at $\po$.
  
  We can show $A, B$  satisfy the estimates similar to (\ref{RH_curvy_A}) and (\ref{RH_curvy_B})
  \begin{align}
  	|L^{p\qbar} \partial_{p\qbar}A
  	-L^{p\qbar} \partial_{p} A A^{-1}  \partial_{\overline q}A
  	-L^{p\qbar} \partial_{\qbar} B\overline{ A^{-1}}  \partial_{p}\overline{B}|
  	\leq 
  	\eqe C_{\delta}\Onep,
  	\label{RH_intermediate_straight_A}
  	\\
  	|L^{p\qbar} \partial_{p\qbar}B
  	-L^{p\qbar} \partial_{p} AA^{-1}  \partial_{\overline q}B
  	-L^{p\qbar} \partial_{\qbar} B \overline{A^{-1}}  \partial_{p}\overline{A}|
  	\leq
  	\eqe C_{\delta}\Onep,
  	\label{RH_intermediate_straight_B}
  \end{align}
  and $A, B$ are close to  $\MA, \MB$ in the sense that  the following relations are valid at $\po$:
\begin{align}
	|L^{i\jbar}\partial_{i\jbar}(A-\MA)|&\leq  \eqe C_{\delta}(1+\Onep),
	\label{D2A-A}\\
		|L^{i\jbar}\partial_{i\jbar}(B-\MB)|&\leq \eqe C_{\delta}( C_{\delta}(1+\Onep)+\Twop),
		\label{D2B-B}\\
		\partial_i A&=\partial_i \MA,
		\label{DA-A}\\
		|\eqe \Phi^{p\qbar} \partial_\qbar \MB-\eqe \Phi^{p\qbar} \partial_\qbar B|&\leq \eqe C_{\delta},
		\label{DB-B}\\
		|\eqe\Phi^{p\vbar} \partial_p\overline{B}|+	|\eqe\Phi^{p\vbar} \partial_p A|&\leq C_\delta \sqrt{\Onep}.
		\label{Third_DerivativeControl_Decomposition}
\end{align}
With the relations  above,  in (\ref{RH_intermediate_straight_A}) and (\ref{RH_intermediate_straight_B}) we can replace $A, B$ by $\MA, \MB$ and get (\ref{RH_curvy_A}) and (\ref{RH_curvy_B}).
For this, we only need to explain how to handle terms containing first derivatives of $A$ and $B$ in (\ref{RH_intermediate_straight_A}) and (\ref{RH_intermediate_straight_B}); indeed, we only explain how to handle $L^{p\qbar}\partial_\qbar B\overline{A^{-1}}\partial_p \overline{B}$, the treatment of other terms is simpler. We have the following decomposition:
\begin{align}
	L^{p\qbar}\partial_\qbar B\overline{A^{-1}}\partial_p \overline{B}
	=\left(
	\eqe \Phi^{u\qbar} \partial_{\qbar}B
	\right) \overline{A^{-1}}
	\left(
	\eqe \Phi^{p\vbar}\partial_p\overline{B}
	\right)
	\MG_{u\vbar}.
\end{align}
Then using (\ref{DB-B}) and (\ref{Third_DerivativeControl_Decomposition}), we find
\begin{align}
	|\left(
	\eqe \Phi^{u\qbar} \partial_{\qbar}B
	\right) \overline{A^{-1}}
	\left(
	\eqe \Phi^{p\vbar}\partial_p\overline{B}
	\right)
	\MG_{u\vbar}
	-\left(
	\eqe \Phi^{u\qbar} \partial_{\qbar}\MB
	\right) \overline{\MA^{-1}}
	\left(
	\eqe \Phi^{p\vbar}\partial_p\overline{\MB}
	\right)
	\MG_{u\vbar}|\leq  \eqe C_\delta (1+\sqrt{\Onep});
\end{align}
therefore
\begin{align}
	|	L^{p\qbar}\partial_\qbar B\overline{A^{-1}}\partial_p \overline{B}
	-	L^{p\qbar}\partial_\qbar \MB\overline{\MA^{-1}}\partial_p \overline{\MB}|
	\leq 
	\eqe C_\delta (1+\sqrt{\Onep}).
\end{align}

In the following, we prove (\ref{RH_intermediate_straight_A})-(\ref{Third_DerivativeControl_Decomposition}) one by one.
First, we prove (\ref{RH_intermediate_straight_A}). We let $i,j$ be $\alpha,\beta$ in (\ref{eq:SigmaQuotient_double_derivative_indexform}) and get
\begin{align} 
	L^{p\qbar}(\Phi_{\alpha\betabar})_{p\qbar}=
	L^{p\qbar} \Phi_{p\vbar \alpha}\Phi^{u\vbar} \Phi_{\qbar u\betabar}+
	L^{p\qbar} \Phi_{p\vbar \betabar}\Phi^{u\vbar} \Phi_{\qbar u \alpha}.
	\label{eq:SigmaQuotient_alphabetabar_derivative_indexform}
\end{align}
Since $\Phi_{\ijbar}$ is diagonal, we have
\begin{align} 
	\begin{split}
	|L^{p\qbar}(\Phi_{\alpha\betabar})_{p\qbar}-
	L^{p\qbar} \Phi_{p\nubar \alpha}\Phi^{\mu\nubar} \Phi_{\qbar \mu\betabar}&-
	L^{p\qbar} \Phi_{p\nubar \betabar}\Phi^{\mu\nubar} \Phi_{\qbar \mu \alpha}|\\
	\leq\frac{1}{T}&
	\left(	L^{p\qbar} \Phi_{p\taubar \alpha} \Phi_{\qbar \tau\betabar}+
	L^{p\qbar} \Phi_{p\taubar \betabar} \Phi_{\qbar \tau \alpha}\right).
	\label{610390}
		\end{split}
\end{align}
So we need to show the right-hand side of (\ref{610390}) can be controlled by $\eqe C_{\delta} \Onep$.  Similar to the proof of (\ref{610377}), we have 
  \begin{align}
	L^{p\qbar}\Phi_{\betabar\tau\qbar}\Phi_{\taubar \alpha p}=
	\left(
	\eqe \Phi^{p\vbar}{ \Phi_{\taubar \alpha p}}
	\right) 
	\MG_{u\vbar}
	\left(
	\eqe \Phi^{u\qbar}{ \Phi_{\betabar \tau \qbar}}
	\right) .
	\label{610391}
\end{align}
Using (\ref{51534}), we know  $|\eqe \Phi^{p\vbar}{ \Phi_{\taubar \alpha p}}|$ and $|	\eqe \Phi^{u\qbar}{ \Phi_{\betabar \tau \qbar}}|$ are both smaller than  $\eqe C_{\delta}\Onep^{\frac{1}{2}}$, so (\ref{610391}) can be controlled by $\eqe^2 C_{\delta}\Onep$. $	L^{p\qbar} \Phi_{p\taubar \betabar} \Phi_{\qbar \tau \alpha}$ can be estimated similarly, and we know (\ref{RH_intermediate_straight_A}) is valid.  The proof of (\ref{RH_intermediate_straight_B}) is almost parallel; we only need to replace $\betabar$ by $\beta$ in the argument above.

For the proof of (\ref{D2A-A}), we have
\begin{align}
	L^{\ijbar} \partial_{\ijbar}(A-\MA)=&L^{\ijbar} \Phi_{\alpha\taubar i }\overline{\Phi_{\beta j}}
    +L^{\ijbar} \Phi_{\alpha\taubar \jbar } {\Phi_{\betabar i }}
    +L^{\ijbar} \Phi_{\tau\betabar i}\Phi_{\alpha \jbar}
    +L^{\ijbar} \Phi_{\tau\betabar \jbar} \Phi_{\alpha i}
    \label{610392}
    \\
    &\ \ \ \ \  \ \ \ +L^{\ijbar}\Phi_{\tautaubar}\Phi_{\alpha i}\overline{\Phi_{\beta j}}
    +L^{\ijbar}\Phi_{\tautaubar}\Phi_{\alpha\jbar} \Phi_{\betabar i}.
    \label{610393}
\end{align}
Terms in (\ref{610393}) can be controlled by $\eqe C_{\delta}$ since second derivatives of $\Phi$ are all bounded by $C_{\delta}$ and $\Phi_{\tautaubar}$ is comparable to $\eqe$.
For the terms in the right-hand side of  (\ref{610392}), we only estimate the first one; 
we have
\begin{align}
	L^{\ijbar}\Phi_{\alpha\taubar i}\overline{\Phi_{\beta j}}
	=(\eqe \Phi^{i \qbar}\Phi_{i\alpha \taubar})\MG_{p\qbar}(\eqe \Phi^{p\jbar }\ \overline{\Phi_{\beta j}}).
	\label{611394}
\end{align}
In above, we notice that
\begin{align}
	|\Phi^{i\qbar} \Phi_{\alpha\taubar i}|\leq C_{\delta} \Onep^{\frac{1}{2}},
\end{align}which follows from (\ref{51534}), and 
\begin{align}
	|\eqe\Phi^{p\jbar}\overline{\Phi_{\beta j}}|\leq C_\delta,
\end{align} which follows from (\ref{Phitt_comparable_epsilon}). So we have that  $|(\ref{611394})|$ is smaller than $\eqe C_{\delta} \Onep^{\frac{1}{2}}$.
Estimates for $ L^{\ijbar} \Phi_{\alpha\taubar \jbar } {\Phi_{\betabar i }}$ $L^{\ijbar} \Phi_{\tau\betabar i}\Phi_{\alpha \jbar}$ and $L^{\ijbar} \Phi_{\tau\betabar \jbar} \Phi_{\alpha i}$ are similar; thus, we proved (\ref{D2A-A}).

For the proof of (\ref{D2B-B}), we have
\begin{align}
	&L^{p\qbar}\partial_{p\qbar}
	\left(
	B_{\alpha\beta}-\MB_{\alpha\beta}
	\right)\\
	=& L^{p\qbar}\Phi_{\alpha\tau p}\Phi_{\beta\qbar}
		+L^{p\qbar}\Phi_{\alpha\tau\qbar}\Phi_{\beta p}
		-L^{p\qbar}\Phi_{\alpha\tau}\Phi_{\beta\qbar} \Phi_{\tau p} 
		-L^{p\qbar}\Phi_{\alpha\tau}\Phi_{\beta p}\Phi_{\tau \qbar}
		\label{611A}
		\\
		&+ L^{p\qbar}\Phi_{\beta\tau p}\Phi_{\alpha\qbar}
		+L^{p\qbar}\Phi_{\beta\tau\qbar}\Phi_{\alpha p}
		-L^{p\qbar}\Phi_{\beta\tau}\Phi_{\alpha\qbar} \Phi_{\tau p} 
		-L^{p\qbar}\Phi_{\beta\tau}\Phi_{\alpha p}\Phi_{\tau \qbar}
		\label{611B}
		\\
		&-L^{p\qbar}\Phi_{\tau\tau }\Phi_{\alpha p}\Phi_{\beta \qbar}
		   -L^{p\qbar}\Phi_{\tau\tau }\Phi_{\beta p}\Phi_{\alpha \qbar}.
		   \label{611C}
\end{align}
We notice that $(\ref{611A})_1$, $(\ref{611A})_3$, $(\ref{611A})_4$, $(\ref{611B})_1$, $(\ref{611B})_3$, $(\ref{611B})_4$ and terms of (\ref{611C}) all contain $ L^{p\qbar}\Phi_{k\qbar}$; using
\begin{align}
	|L^{p\qbar} \Phi_{k \qbar}|=|(\eqe \Phi^{p\vbar} \Phi_{k\qbar})
	(\MG_{u\vbar} \Phi^{u\qbar} \eqe)|\leq \eqe,  
\end{align}and 
(\ref{610317}) we know they can all be controlled by $\eqe C_{\delta}(C_\delta(\Onep+1)+\Twop)^{\frac{1}{2}}$. For the estimates of 
$(\ref{611A})_2$, we have
\begin{align}
|L^{p\qbar}\Phi_{\alpha\tau\qbar}\Phi_{\beta p}|=	|(\eqe \Phi^{k\qbar} \Phi_{\alpha\tau\qbar}) \MG_{k\sbar} (\eqe \Phi^{p\sbar}\Phi_{\beta p})|\leq \eqe  C_{\delta} \Onep^{\frac{1}{2}}.
\end{align} The estimate for  $(\ref{611B})_2$ follows by switching $\alpha$ with $\beta$. Therefore, we proved (\ref{D2B-B}).

The proof of (\ref{DA-A}) is simple: we only need to notice that $\MA-A$ is a second order small quantity around $\po$. This is because $\Phi_{\alpha\taubar}=0$ and $\Phi_{\beta}=0$ at $\po$.

For (\ref{DB-B}), we have
\begin{align}
	\eqe \Phi^{p\qbar} \partial_\qbar(B_{\alpha\beta}-\MB_{\alpha\beta})=
	\eqe \Phi^{p\qbar} (\Phi_{\alpha\tau}\Phi_{\beta\qbar}-\Phi_{\beta\tau}\Phi_{\alpha\qbar}).
	\label{610395_1}
\end{align}
When $q=n$, $\Phi_{\beta\qbar}=\Phi_{\alpha\qbar}=0$; when $q\neq n$, $\Phi^{p\qbar}$ is bounded. So (\ref{610395_1}) can be controlled by $\eqe C_{\delta}$.

For (\ref{Third_DerivativeControl_Decomposition}), we only need to notice that $\eqe \Phi^{\ijbar}$ is bound and use (\ref{51534}).

With all the estimates above, we know (\ref{RH_curvy_A}) and (\ref{RH_curvy_B}) are valid; with them, we  will prove (\ref{RH_curvey_tr(I-K)}) in the following. 
Before the proof, we do some preparations.

We note that 
\begin{align}
	\MK^\dag=(\MB\overline{\MA^{-1}}\ \overline{\MB}\MA^{-1})^\dag
				=\MA^{-1}\MB\overline{\MA^{-1}}\ \overline{\MB}=\MA^{-1} \MK\MA
\end{align} and that $\MK^\dag=\MK$ at $\po$ since $\MA=I$ at $\po$.

To simplify the computation, we let
\begin{align}
	\MMB_k=\MB_k-\MA_k\MA^{-1}\MB-\MB\overline{\MA^{-1}}\ \overline\MA_k;
	\label{6113105}
\end{align} we note that $\MMB_k$ is  symmetric. Then
\begin{align}
	\MA^{-1}\MMB_k\overline{\MA^{-1}}=\partial_k(\MA^{-1}\MB\overline{\MA^{-1}}).
\end{align}
We need to compute  $L^{p\qbar}\partial_\qbar{\MMB_p}$: directly applying $L^{p\qbar}\partial_\qbar$ to (\ref{6113105}) we have
\begin{align}
	L^{p\qbar}\partial_\qbar{\MMB_p}
	&=L^{p\qbar}\partial_{p\qbar}\MB
	 -L^{p\qbar}\MA_p \MA^{-1}\MB_{\qbar}-L^{p\qbar} \MB_{\qbar} \overline{\MA^{-1}}\ \overline\MA_p
	 \\
	 -&L^{p\qbar} \MA_{p\qbar}\MA^{-1}\MB
	 -L^{p\qbar}\MB \overline{\MA^{-1}}\ \overline{\MA}_{p\qbar}
	 +L^{p\qbar}\MA_p\MA^{-1} \MA_{\qbar} \MA^{-1}\MB
	 +L^{p\qbar}\MB\ \overline{\MA^{-1}}\ \overline{\MA_{q}}\ \overline{\MA^{-1}}
	    \ \overline{\MA}_p.
\end{align}
Using (\ref{RH_curvy_A}) and (\ref{RH_curvy_B}), we get
\begin{align}
	|L^{p\qbar}\partial_\qbar{\MMB_p}
	-L^{p\qbar}\MB_\qbar\overline{\MA^{-1} }\ \overline{\MB}_{p} \MA^{-1}\MB
	-L^{p\qbar}\MB\ \overline{\MA^{-1}}\ \overline{\MB}_{p}{\MA^{-1}}
	\ {\MB}_\qbar|\leq C_{\delta}(C_\delta(\Onep+1)+\Twop)  .
	\label{6113109}
\end{align}

For $\plainA$ and $\plainB$ satisfying (\ref{RH_plain_B}), we have
\begin{align}
	L^{p\qbar}\partial_p\left(\plainA^{-1}\plainB_{\qbar}\overline{\plainA^{-1}}\right)=0;
\end{align}similarly, for $\MA, \MB$ satisfying (\ref{RH_curvy_B}), we have
\begin{align}
 |	L^{p\qbar}\partial_p\left(\MA^{-1}\MB_{\qbar}\overline{\MA^{-1}}\right)|\leq C_{\delta}(C_\delta(\Onep+1)+\Twop).
 \label{6113111}
\end{align}

When applying $\partial_\qbar$ and $\partial_p$ to $\MK$, we get
\begin{align}
	\MK_\qbar=\MB_\qbar \overline{\MA^{-1}}\ \MB\MA^{-1}
						+\MB\overline{\MA^{-1}}\ \overline{\MMB_q}\MA^{-1};
						\label{6113112}
\end{align}
its complex conjugation is
\begin{align}
	(\MK^\dag)_p=\MA^{-1} \MMB_p \overline{\MA^{-1}}\ \overline\MB
	   										+\MA^{-1}\MB\overline{\MA^{-1}}\ \overline{\MB}_p.
	   												\label{6113113}
\end{align}

With the preparation above, we can compute $L^{p\qbar}\partial_{p\qbar}\left[
\tr (I-\MK)^{-1}
\right]$. Using (\ref{6113112}), we have
\begin{align}
	&L^{p\qbar}\partial_{p}\partial_{\qbar}\left[
	\tr (I-\MK)^{-1}
	\right]
	\label{6113114}
	\\
	=&L^{p\qbar} \partial_p
	\left[
	\tr
	{(I-\MK)}^{-1} 
	\left(
	\MB_\qbar \overline{\MA^{-1}}\ \overline\MB\MA^{-1}
	+\MB\ \overline{\MA^{-1}}\ \overline{\MMB_q}\MA^{-1}
	\right)
	{(I-\MK)}^{-1} 
	\right].
	\label{6113115}
\end{align}
We insert $\MA$ and $\MA^{-1}$ into the equation above to make $\MK$ into $\MK^\dag$ so that  the computation is simpler when we apply $\partial_p$; we have
\begin{align}
	&\tr
	{(I-\MK)}^{-1} 
	\left(
	\MB_\qbar \overline{\MA^{-1}}\ \overline\MB\MA^{-1}
	+\MB\ \overline{\MA^{-1}}\ \overline{\MMB_q}\MA^{-1}
	\right)
	{(I-\MK)}^{-1} \\
	=&\tr
	\left(\MA^{-1}{(I-\MK)}^{-1}\MA\right) 
	\left(
	 \MA^{-1}\MB_\qbar \overline{\MA^{-1}}\ \overline\MB
	+\MA^{-1}\MB\ \overline{\MA^{-1}}\ \overline{\MMB_q}
	\right)
   \left(\MA^{-1}{(I-\MK)}^{-1}\MA\right) \\
   =&\tr
  {(I-\MK^\dag)}^{-1}
   \left(
   (\MA^{-1}\MB_\qbar \overline{\MA^{-1}})\ \overline\MB
   +(\MA^{-1}\MB\ \overline{\MA^{-1}})\ \overline{\MMB_q}
   \right)
   {(I-\MK^\dag)}^{-1}.
   \label{6113118}
\end{align}
Then we apply $L^{p\qbar}\partial_p$ to (\ref{6113118}) and get that $(\ref{6113115})$ equals to
\begin{align}
	&L^{p\qbar}\left[\tr
	{(I-\MK^\dag)}^{-1}
		\left(
	{\MMB_p}\overline\MB
	+  \MB\overline{\MB}_p
	\right)
	{(I-\MK^\dag)}^{-1}
	\left(
	\MB_\qbar \overline\MB
	+\MB \overline{\MMB_q}
	\right)
	{(I-\MK^\dag)}^{-1}\right]
	\label{draft136}\\
+	&L^{p\qbar}\left[\tr
	{(I-\MK^\dag)}^{-1}
	\left(
	\MB_\qbar \overline\MB
	+\MB \overline{\MMB_q}
	\right)
	{(I-\MK^\dag)}^{-1}
	\left(
	\MMB_p \overline\MB
	+\MB \overline{\MB}_p
	\right)
	{(I-\MK^\dag)}^{-1}\right]
		\label{draft137}
		\\
		+&\tr \left(L^{p\qbar}\partial_p(\MA^{-1}\MB_\qbar\ \overline{\MA^{-1}})
		\overline\MB(I-\MK)^{-2}
		\right)+
		\tr
		\left(\MB (L^{p\qbar}
		\partial_p\overline{\MMB_q})
		(I-\MK)^{-2}\right)
			\label{draft138}
		\\
		+&L^{p\qbar}\tr\left(\MB_\qbar\overline{\MB}_p(I-\MK)^{-2}\right)
		+L^{p\qbar}\tr\left(\MMB_p\overline{\MMB_q}(I-\MK)^{-2}\right).
		\label{draft139}
\end{align}
To get (\ref{draft136}) and (\ref{draft137}), we apply $L^{p\qbar}\partial_p$ to the $(I-\MK)^{-1}$'s of (\ref{6113118}) and  use equation (\ref{6113113}); to get (\ref{draft138}), we apply $L^{p\qbar}\partial_p$  to $ \MA^{-1}\MB_\qbar \overline{\MA^{-1}}
$ and $\overline{\MMB_q}$ of (\ref{6113118}); to get (\ref{draft139}), we apply $L^{p\qbar}\partial_p$  to $\overline\MB$ and $\MA^{-1}\MB\ \overline{\MA^{-1}}$. In above, we also used $\MA=I$ and $\MK=\MK^\dagger$ at $\po$.  Because of the apriori assumption (\ref{Assumption_robust_ComputationProposition}), we have $(I-\MK)^{-1}\leq C_{\delta}$. So we can use 
(\ref{6113109}) and (\ref{6113111}) to simplify (\ref{draft138}) and get
\begin{align}
	|(\ref{draft138})-L^{p\qbar}\tr
	\left(
	\MB(\overline{\MB}_p\MB_\qbar\overline{\MB}+\overline\MB\MB_\qbar\overline\MB_p)
	(I-\MK)^{-2}
	\right)
	|\leq C_{\delta}(C_\delta(\Onep+1)+\Twop).
	\label{6113123}
\end{align}

In the following, we use index summation  to simplify (\ref{draft136})-(\ref{draft139}). First, with a unitary change of coordinate in the direction of $<\partial_{z^1}, \ ...\ , \partial_{z^{n-1}}>_\EC$, we make 
\begin{align}
	\MB=\diag(\lambda_1,\ ...\ ,\lambda_{n-1});
\end{align}
this can be done with Autonne-Takagi factorization (Corollary 4.4.4(c) of \cite{HornMatrix}). Then at $\po$ 
\begin{align}
	I-\MK=I-\MB\overline{\MB} =\diag(1-\lambda_1^2, \ ...\ , 1-\lambda_{n-1}^2).
\end{align}We denote
\begin{align}
	\overline\MB_p=(\overline\MB_{p;\alpha\beta}), \ \ \ \ \  \ \ \ \MMB_p=(\MMB_{p;\alpha\beta});
\end{align} since $\MB_\qbar$ and $\MMB_p$ are both symmetric, we have
$\overline\MB_{p;\alpha\beta}=\overline\MB_{p;\beta\alpha}$ and $\MMB_{p;\alpha\beta}=\MMB_{p;\beta\alpha}$.
Then 
\begin{align}
	(\ref{draft136})=\sum_{p, q , \alpha,\beta}
	   L^{p\qbar}
	   \left(
	   \MMB_{p;\alpha\beta}, \overline{\MB}_{p;\alpha\beta}
	   \right) M^1_{\alpha\beta}  
	   \left(
	   \MMB_{q;\alpha\beta}, \overline{\MB}_{q;\alpha\beta}
	   \right)^\dag,
\end{align}
where 
\begin{align}
	M^1_{\alpha\beta}=
	\frac{1}{(1-\lambda_\alpha^2)^2(1-\lambda_\beta^2)}
	\left(
	\begin{array}{cc}
	\lambda_\beta^2	&     \lambda_\alpha\lambda_\beta\\
	 \lambda_\alpha\lambda_\beta	&\lambda_\alpha^2	
	\end{array}
	\right).
\end{align}
In (\ref{draft137}), we transpose $	{(I-\MK^\dag)}^{-1}
\left(
\MB_\qbar \overline\MB
+\MB \overline{\MMB_q}
\right)
{(I-\MK^\dag)}^{-1}
\left(
\MMB_p \overline\MB
+\MB \overline{\MB}_p
\right)
{(I-\MK^\dag)}^{-1}$ and get
\begin{align}
	(\ref{draft137})=\sum_{p, q , \alpha,\beta}
	L^{p\qbar}
	\left(
	\MMB_{p;\alpha\beta}, \overline{\MB}_{p;\alpha\beta}
	\right) M^2_{\alpha\beta}  
	\left(
	\MMB_{q;\alpha\beta}, \overline{\MB}_{q;\alpha\beta}
	\right)^\dag,
\end{align}
where 
\begin{align}
	M^2_{\alpha\beta}=
	\frac{1}{(1-\lambda_\alpha^2)^2(1-\lambda_\beta^2)}
	\left(
	\begin{array}{cc}
		\lambda_\alpha^2	&     \lambda_\alpha\lambda_\beta\\
		\lambda_\alpha\lambda_\beta	&\lambda_\beta^2	
	\end{array}
	\right).
\end{align}
Finally, (\ref{6113123}) gives us
\begin{align}
	(\ref{draft138})+(\ref{draft139})=\sum_{p, q , \alpha,\beta}
	L^{p\qbar}
	\left(
	\MMB_{p;\alpha\beta}, \overline{\MB}_{p;\alpha\beta}
	\right) M^3_{\alpha\beta}  
	\left(
	\MMB_{q;\alpha\beta}, \overline{\MB}_{q;\alpha\beta}
	\right)^\dag+F,
\end{align}
where 
\begin{align}
	M^3_{\alpha\beta}=
	\frac{1}{(1-\lambda_\alpha^2)^2}
	\left(
	\begin{array}{cc}
		1	&      \\
&1-2\lambda_\alpha^2	
	\end{array}
	\right)
\end{align}  and $F$ is a term which can be controlled by $C_{\delta}(C_\delta(\Onep+1)+\Twop)$.

Summing up, we have
\begin{align}
	L^{p\qbar}\partial_{p\qbar} (\tr (I-\MK)^{-1})=\sum_{p, q , \alpha,\beta}
	L^{p\qbar}
	\left(
	\MMB_{p;\alpha\beta}, \overline{\MB}_{p;\alpha\beta}
	\right) M_{\alpha\beta}  
	\left(
	\MMB_{q;\alpha\beta}, \overline{\MB}_{q;\alpha\beta}
	\right)^\dag+F,
	\label{7143114}
\end{align}
where 
\begin{align}
	M_{\alpha\beta}=	M_{\alpha\beta}^1+	M_{\alpha\beta}^2+	M_{\alpha\beta}^3=
	\frac{1}{(1-\lambda_\alpha^2)^2(1-\lambda_\beta^2)}
	\left(
	\begin{array}{cc}
		1+\lambda_\alpha^2	& 2\lambda_{\alpha}\lambda_\beta     \\
		 2\lambda_{\alpha}\lambda_\beta  &1-\lambda_\alpha^2	+2\lambda_\alpha^2\lambda_\beta^2
	\end{array}
	\right).
\end{align}  We can easily check that when $\lambda_\alpha$ and $\lambda_\beta$ are both positive and smaller than $1$, $M_{\alpha\beta}$ is positive definite. So $(\ref{7143114})_2$ is positive, and (\ref{RH_curvey_tr(I-K)}) follows. With the estimates for $W$ and $V$ we know (\ref{eq:sigma_Q_main_computation}) follows from (\ref{RH_curvey_tr(I-K)}).
\section{Lower Bound Estimates for Norms of Gradients}
\label{sec:Lower_Bound_Gradient_Norm}
In this section, using inequality  (\ref{eq:-logs_main_computation}) of  Proposition 
\ref{prop:Main_Computation}. We derive an apriori estiamte for the lower bound of the norm of the gradient of $\Phi$.  More precisely, we have the following apriori estimate
\begin{proposition}\label{prop:Gradient_Estiamte_Apriori}
	[Apriori Lower Bound Estimate for the Norm of the Gradient] Suppose $\Phi$ is a solution to Problem \ref{prob:Perturbed_Problem_HessianQuotient} with the robustness of \qcconvexity\ greater than $\delta$, and suppose that $\eqe$ is small enough. 
	Then, there is a positive constant $C_\MR^g$ dependent of the geometry of $\MR$ and independent of $\eqe$ and $\delta$  so that	
	\begin{align}
		|\nabla \Phi|_{\MG} \geq {C_{\MR}^g}.
	\end{align}
\end{proposition}

There are three steps in the proof.
  First, we show 
\begin{align}
	|\nabla\Phi|_{\MG}^2\geq S\ \ \ \ \  \ \ \ \text{ in }\MR;
	\label{15642}
\end{align}
then, using  (\ref{eq:-logs_main_computation}),  we show 
\begin{align}
	S\geq \frac{1}{C} \min_{\partial \MR} S\ \ \ \ \  \ \ \ \text{ in }\MR
\end{align}for a positive constant $C$; finally, using the estimate for $|\nabla\Phi|_{\MG}$ on $\partial \MR$, which we derive in Section \ref{sec:B_boundary_Barrier}, we estimate the lower bound for $S$ on $\partial\MR$.
\begin{proof}
	[Proof of Proposition \ref{prop:Gradient_Estiamte_Apriori}]
	
{\bf Step 1.}	We choose a coordinate $(z^\alpha, \tau)$, $\alpha=1,\ ...\ , n-1$, so that $\MG_{\ijbar}=\delta_{ ij}$ and $\Phi_{\ijbar}=\delta_{ij}\lambda_i$
	 for a point $\po\in\MR$. Then 
	 \begin{align}
	 	|\nabla\Phi|_{\MG}^2=\sum_i|\Phi_i|^2,\\
	 	S=\eqe \sum_i \frac{|\Phi_i|^2}{\lambda_i}.
	 \end{align}
	At the point $\po$, equation  (\ref{eq:Matrix_Form_Main_Perturbation_Equation}) is
	\begin{align}
		\sum_i\frac{\eqe}{\lambda_i}=1,
	\end{align}
	so $\frac{\eqe}{\lambda_i}<1$.Therefore, $|\nabla\Phi|_{\MG}^2\geq S$ at $\po$.  Since $\po$ can be any point in $\MR$, and $|\nabla\Phi|_{\MG}, \ S$ are both independent of coordinate, we know (\ref{15642}) is valid.

	{\bf Step 2.} When $\eqe$ is small enough, we can use (\ref{eq:-logs_main_computation}) and the maximum principle and get 
	\begin{align}
		\min_{\partial \MR}\log(S)-\max_{\partial\MR}\sqrt{\eqe}\left[W+\MG_{\ijbar}z^i\zjbar\right]\leq
		\log(S)
		\ \ \ \ \  \ \ \ \text{ in }\MR.
	\end{align}
	The exponential of the inequality above is
	\begin{align}
		S\geq 
		e^{-\sqrt{\eqe}
			\max_{\partial\MR}\left( W+ \MG_{\ijbar} z^i z^\jbar
			\right)}\min_{\partial\MR} S, \ \ \ \ \  \ \ \ \text{ in }\MR.
	\end{align}

	{\bf Step 3.}
To estimate $\min_{\partial \MR} S$, we choose a coordinate so that $\MG_\ijbar=\delta_{ij}$ and 
	$\Phi_\alpha=0$ at a boundary point $\po$. Then 
	\begin{align}
		S=\eqe \Phi_\tau\Phi^{\tautaubar} \Phi_{\taubar},\ \ \ 
		|\nabla\Phi|_{\MG}^2=|\Phi_\tau|^2\ \ \ \ \  \ \ \ \text{ at }\po. 
		\label{51648}
	\end{align}To find the relation between $|\nabla\Phi|_{\MG}^2$ and $S$, we need to estimate
	$\eqe \Phi^{\tautaubar}$.  We further require that $\Phi_{\alpha\betabar}=\delta_{\alpha\beta}\lambda_\alpha$ at $\po$. We have an estimate on the lower bound of $\lambda_\alpha$, depending  on the the \cconvexity\ of the boundary and the lower bound estimate of $|\nabla\Phi|_{\MG}$ at $\po$.  A basic linear algebra computation gives
	\begin{align}
	\Phi^{\tautaubar}=	(\MH^{-1})_{\tautaubar}=\frac{1}{\Phi_{\tautaubar}-\sum_\mu\frac{|\Phi_{\tau\mubar}|^2}{\lambda_\mu}}
	\end{align}  
	and
	\begin{align}
		\Phi^{\alpha\alphabar}=(\MH^{-1})_{\alpha\alphabar}= \frac{1}{\lambda_\alpha}+
		\frac{\frac{|\Phi_{\tau\alphabar}|^2}{\lambda_\alpha^2}}{\Phi_{\tautaubar}-\sum_\mu\frac{|\Phi_{\tau\mubar}|^2}{\lambda_\mu}}.
	\end{align}
	Therefore, at $\po$ the equation (\ref{eq:Matrix_Form_Main_Perturbation_Equation}) is 
\begin{align}
		\left(
	1+
	\sum_\alpha \frac{|\Phi_{\tau\alphabar}|^2}{\lambda_\alpha^2}
	\right)\Phi^{\tautaubar}
	+
	\sum_\mu \frac{1}{\lambda_\mu}
	=\frac{1}{\eqe}.
\end{align} From this, we get
\begin{align}
	\eqe\Phi^{\tautaubar}=
	\frac{1-\eqe \sum_\mu \frac{1}{\lambda_\mu}}
	{	1+
		\sum_\alpha \frac{|\Phi_{\tau\alphabar}|^2}{\lambda_\alpha^2}}.
\end{align} When $\eqe$ is small enough, it's greater than a positive constant.  Then using (\ref{51648}), we know $S$ has a positive lower bound on $\partial\MR$.
\end{proof} 

\section{Convexity Estimate}
\label{sec:Convexity_Estimate_GeneralDim_Def+Perturb}

The main goal of this section is to derive the estimate for  the robustness of \qcconvexity\ of the solution ${\Phi^\eqe}$ to  Problem \ref{prob:Perturbed_Problem_HessianQuotient}. In section \ref{sec:sub:Convexity_Deformation}, by deforming $\Omega_0$ and $\Omega_1$ to concentric balls, we derive an estimate which depends on this deformation procedure; in section \ref{sec:sub:Improving_Convexity}, we improve the estimate so that it only depends on the geometry of $\MR$; in section \ref{sec:sub:Form_Estimate}, we estimate $(\iu\ddbar {\Phi^\eqe})^{n-1}\wedge \bbform$ and $\iu\ddbar \left(
e^{\Phi^\eqe}
\right)$.
\subsection{Convexity Estimate Using Deformation}
\label{sec:sub:Convexity_Deformation}
Using Lemma \ref{lem:Deform_a_strongly_C-convex_domain_to_Ball}, we can construct two families of domains $\left\{
\Omega_0^\lambda
\right\}_{\lambda\in[0,1]}$ and $\left\{
\Omega_1^\lambda
\right\}_{\lambda\in[0,1]}$ connecting $\Omega_0$ and $\Omega_1$ with two concentric balls $B_r$ and $B_R$ both contained in $\Omega_0$. In addition, by properly choosing the parameter $\lambda$,
\begin{align}
	\overline{\Omega^\lambda_0}\subset\Omega_1^\lambda \ \ \ \ \  \ \ \ \text{ for  any }\lambda\in [0,1],
\end{align} can be satisfied. Then we let $\MR^\lambda=\Omega_1^\lambda\backslash\overline{\Omega_0^\lambda}$ and let $\Phi^\lambda_\eqe$ be the solution to the following problem:
\begin{problem}
	\label{prob:Perturbed_Problem_HessianQuotient_VaryingDomain}  
	Suppose that $\bbform =\sqrt{-1}\bb_{ij}dz^i\wedge \overline{dz^j}$ be a constant coefficient non-degenerate K\"ahler form on $\EC^n$ and that $\eqe$ is a positive constant.
	Find ${\Phi_\eqe^\lambda}$ satisfying
	\begin{align}
		\left(\sqrt{-1}\ddbar {\Phi_\eqe^\lambda}\right)^n&=\eqe	\left(\sqrt{-1}\ddbar{\Phi_\eqe^\lambda}\right)^{n-1}\wedge\bbform \ \ \ \ \ \ &\text{in }\MR^\lambda,\ \ \label{eq:SigmaQuotient_Problem_Perturbation_VaryingDomain}
		\\
		\sqrt{-1}\ddbar{\Phi_\eqe^\lambda}&>0\ \ \ \ \ \ &\text{in }\MR^\lambda,\ \ 
		\label{cond:positive_ddbar_in_deformation_Process}\\
	{\Phi_\eqe^\lambda}&=0\ \ \ \ \ \  &\text{\ \ on }\partial \Omega_0^\lambda,\\
	{\Phi_\eqe^\lambda}&=1\ \ \ \ \ \  &\text{\ \ on }\partial \Omega_1^\lambda.
	\end{align}
	
\end{problem} 

According to Lemma \ref{lemma:GradientNorm_Lower_Boundary_DomainFamily}, there is an $\ese_0$ and $\sigmal$ so that for any $\eqe\in(0,\ese_0]$ and any $\lambda\in[0,1]$ $|\nabla {\Phi_\eqe^\lambda}|_\MG$ is greater than $\sigmal$ on $\partial \MR^\lambda$. Since $\Omega_0^\lambda$ and $\Omega_1^\lambda$ are all \scconvex, we can apply Lemma \ref{lemma:Boundary_Convex_implies_Modulus} and get the modulus of \qcconvexity\ of ${\Phi_\eqe^\lambda}$ is greater than ${C}{\sigmal}$ on $\partial \MR^\lambda$ for a small constant $C$. We then apply Lemma \ref{lemma:ModulusPositive_implies_RobustnessPositive} and get that the robustness of \qcconvexity\ of $\Phi^\lambda_\eqe$ is greater than $C^2\sigmal^2$. In above, $C$ depends on the geometry of $\Omega_1^\lambda\backslash\overline{\Omega_0^\lambda}$ for all $\lambda\in[0,1]$.

When $\lambda=0$ and $\eqe=0$, the solution can be written down explicitly:
\begin{align}
	\Phi^0_0=\frac{\log|z|-\log r}{\log R-\log r}.
\end{align}
It's easy to see that $\Phi^0_0$ is \sqcconvex\ since the level sets of $\Phi^0_0$ are all Euclidean balls, and its gradient does not vanish anywhere. This two properties hold for $\Phi^0_\eqe$ with $\eqe$ small enough. Therefore, we can find small constants $\ese_1$ and $\delta_0$ so that the robustness of \qcconvexity\ of $\Phi^0_\eqe$ is greater than $\delta_0$, providing $\eqe\in(0,\ese_1]$.

Let
\begin{align}
	\sigma=\frac{1}{2}\min\left\{C^2\sigmal^2, \ \delta_0,\ C_\MR^g\right\}.
\end{align} Here $C_\MR^g$ is from Proposition \ref{prop:Gradient_Estiamte_Apriori}; we can choose it small enough, so that it does not depend on $\lambda$.
Then for any $\eqe\in(0, \frac{\min\{\ese_0, \ese_1\}}{2})$ we let
\begin{align}
	\begin{split}
	L^\eqe=
&\left\{
	\lambda\in[0,1] \ | \text{ the robustness of \qcconvexity\ of  } \Phi_\eqe^\lambda \text{ is greater than $\sigma$}\right
\},
		\end{split}
\end{align}
and we will show $L^\eqe=[0,1]$, providing $\eqe$ is small enough.
It's easy to see that $0\in L^\eqe$ because   $\sigma< \delta_0$; it's also easy  to see that $L^\eqe$ is open since $\Omega_0$ and $\Omega_1$ change smoothly as $\lambda$ changes. It's crucial to show that $L^\eqe$ is right-closed. Suppose that $[0, l)\subset L^\eqe$. We will show $l\in L^\eqe$; we need to show the robustness of \qcconvexity\ of ${\Phi_\eqe^l}$ is greater than $\sigma$. To do this, we choose $\delta$ to be a very small constant, which is independent of $\eqe$, and apply Proposition \ref{prop:Main_Computation}.  As illustrated in Figure \ref{fig:The_Continuity_Method_General_dim}, there are several steps:

\begin{figure}[h]	\centering  	\includegraphics[height=3.5cm]	{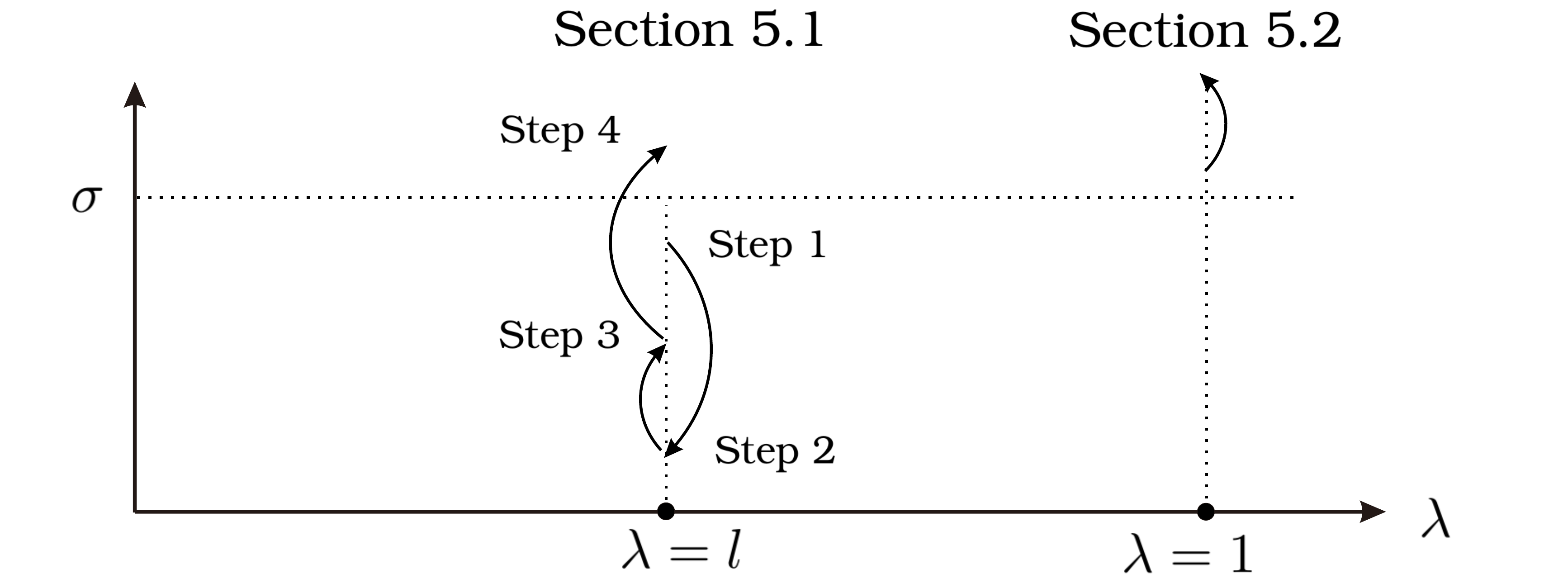}	\caption{Convexity Estimate via a Continuity Argument}	\label{fig:The_Continuity_Method_General_dim}    \end{figure}

In step 1, using Lemma \ref{lemma:Modulus_Estimate_with_Limit} we show the robustness \qcconvexity\ of ${\Phi_\eqe^l}$ is greater than $\sigma-\delta$; in step 2, we show for any $q\in \spaceQ_{\sigma-2\delta}$ the robustness of \qcconvexity\ of ${\Phi_\eqe^l}-q$ is greater than $\delta$; in step 3, we show the  robustness of \qcconvexity\ of ${\Phi_\eqe^l}-q$ is actually larger; in step 4, we show the robustness of \qcconvexity\ of ${\Phi_\eqe^l}$ is greater than $\sigma$.

{\bf Step 1.} According to the assumption, for any $\lambda< l$, the robustness of \qcconvexity\ of ${\Phi_\eqe^\lambda}$ is greater than $\sigma$. So, using Lemma \ref{lemma:Modulus_Estimate_with_Limit}, we know the robustness of \qcconvexity\ of ${\Phi_\eqe^l}$ is greater than $\sigma-\delta$, by letting $\lambda\rightarrow l$. Here we choose $0<\delta<\sigma$.

{\bf Step 2.} Since the robustness of \qcconvexity\ of ${\Phi_\eqe^l}$ is greater than $\sigma-\delta$, 
${\Phi_\eqe^l}-q-q_1$ is \sqcconvex, for any $q\in \spaceQ_{\sigma-2\delta}$ and $q_1\in \spaceQ_{\delta}$. Therefore, the robustness of \qcconvexity\ of ${\Phi_\eqe^l}-q$ is greater than $\delta$. This allows us to apply Proposition \ref{prop:Main_Computation} to ${\Phi_\eqe^l}-q$ and its variants.

{\bf Step 3.} In this step, we fix $q$ and denote ${\Phi_\eqe^l}-q$ by $\Psi$. We will show the robustness of \qcconvexity\ of $\Psi$ is greater than $\frac{\sigma^2}{C}$, for a constant $C$, which is independent of $\sigma,\ \delta$ and $\eqe$.
This  is the most crucial step, and we split it into step 3.1 and step 3.2.

{\bf Step 3.1.} In this step, we fix a point $\po\in\MR$ and choose coordinate so that $\Psi_{\alpha}(\po)=0$, for $\alpha\in\{1,\ ...\ ,n-1\}$, and estimate $\Psi_{\alpha\betabar}(\po)$. Let $h\in \spaceQ_{\frac{\sigma}{2}}$. Here we require that $4\delta<\sigma$, so 
\begin{align}
	\frac{\sigma}{2}<\sigma-2\delta.
\end{align} Therefore, $h\in\spaceQ_{\sigma-2\delta}$.  Let
\begin{align}
	\Psi_m=\Psi+m(q-h)={\Phi_\eqe^l}-(1-m)q-mh,\\
	\Psi^m=\Psi+m(q+h)={\Phi_\eqe^l}-(1-m)q+mh.
\end{align} Since $q$ and $h$ are both elements of $\spaceQ_{\sigma-2\delta}$, the robustness of \qcconvexity\ of $\Psi_m$ and $\Psi^m$ are both greater than $\delta$. Therefore, we can apply Proposition \ref{prop:Main_Computation} and Proposition \ref{prop:Gradient_Estiamte_Apriori} to $\Psi^m$ and $\Psi_m$.

We apply Proposition \ref{prop:Gradient_Estiamte_Apriori} to ${\Phi_\eqe^l}$ and get
\begin{align}
	|\nabla{\Phi_\eqe^l}|\geq \uglb.
\end{align}Then  $|\nabla\Psi|>\frac{\uglb}{2}$ since $\sigma<\frac{\uglb}{2} $.
Since $\Psi_\alpha(\po)=0$ for $\alpha\in\{1,\ ...\ ,n-1\}$, we have
$|\Psi_\tau(\po)|>\frac{\uglb}{2}$. In the following, all the computation for $\Psi^m$ and $\Psi_m$ will be  at $\po$. For $\partial_\tau\Psi_m$ we have
\begin{align}
	|\partial_\tau\Psi_m|=	|\partial_\tau(\Psi+m(q-h))|.
\end{align} We choose $h$, so that $\nabla h=0$ at $\po$.  Then 
\begin{align}
	|\partial_\tau\Psi_m|\geq \frac{\uglb}{2} -m |\nabla q(\po)|\geq \frac{\uglb}{2}-m\sigma.
\end{align} 
The above is positive providing $m<1$ since we already assumed that $\sigma\leq \frac{\uglb}{2}$.
For  $\partial_\tau\Psi^m$,  we have the same result.

Let 
\begin{align}
	v_\alpha^m=\frac{\partial_\alpha(\Psi_m)}{\partial_\tau(\Psi_m)}=\frac{\partial_\alpha(\Psi^m)}{\partial_\tau(\Psi^m)}=\frac{\Psi_\alpha+m q_\alpha}{\Psi_\tau+m q_\tau}.
\end{align}Since $\Phi_\alpha=0$, we have the estimate that
\begin{align}
	|v_\alpha^m|\leq \frac{m \sigma}{\frac{\uglb}{2}-m\sigma}.
	\label{628516}
\end{align}
Also, we know $\partial_\alpha-v_\alpha^m\partial_\tau$ is a tangential direction for the level sets of $\Psi_m$ and $\Psi^m$ at $\po$.
Therefore, we have
\begin{align}
	\frac{
|	[\Psi+m(q-h)]_{\alpha\alpha}-2v_\alpha^m [\Psi+m(q-h)]_{\alpha\tau}+(v_\alpha^m)^2[\Psi+m(q-h)]_{\tau\tau}
|	}{\Psi_{\alpha\alphabar}-v_\alpha^m\Psi_{\tau\alphabar}-\overline{ v_\alpha^m}\Psi_{\alpha\taubar}+|v_\alpha^m|^2\Psi_{\tau\taubar}}<1,
	\label{618515}\\
		\frac{
|		[\Psi+m(q+h)]_{\alpha\alpha}-2v_\alpha^m [\Psi+m(q+h)]_{\alpha\tau}+(v_\alpha^m)^2[\Psi+m(q+h)]_{\tau\tau}
|	}{\Psi_{\alpha\alphabar}-v_\alpha^m\Psi_{\tau\alphabar}-\overline{ v_\alpha^m}\Psi_{\alpha\taubar}+|v_\alpha^m|^2\Psi_{\tau\taubar}}<1.
	\label{618516}
\end{align}
We apply the triangle inequality to (\ref{618515}) and (\ref{618516}) and get
\begin{align}
		\frac{m
	|	h_{\alpha\alpha}-2v_\alpha^m h_{\alpha\tau}+(v_\alpha^m)^2h_{\tau\tau}
	|}{\Psi_{\alpha\alphabar}-v_\alpha^m\Psi_{\tau\alphabar}-\overline{ v_\alpha^m}\Psi_{\alpha\taubar}+|v_\alpha^m|^2\Psi_{\tau\taubar}}<1.
	\label{618517}
\end{align}

Because of the assumption (\ref{cond:positive_ddbar_in_deformation_Process}) and that $q$ is pluriharmonic, we have
\begin{align}
	\left(
	\begin{array}{cc}
	\Psi_{\alpha\alphabar}	&\Psi_{\alpha\taubar}     \\
	\Psi_{\tau\alphabar}	   &\Psi_{\tautaubar} 
	\end{array}
	\right)\geq 0.
\end{align}Therefore, 
\begin{align}
	|\Psi_{\alpha\taubar}|\leq  
	\left(
	\Psi_{\alpha\alphabar} \Psi_{\tautaubar}
	\right)^{\frac{1}{2}}.
\end{align}
Since $q$ is pluriharmonic, $\Psi_{\tautaubar}=\partial_{\tautaubar}{\Phi_\eqe^l}$. Then, using the $C^2$ estimate of Appendix \ref{sec:Appendix_C11Estimate}, we have 
\begin{align}
	0\leq \Psi_\tautaubar\leq C_2,
\end{align} for a constant $C_2$.
So we have
\begin{align}
	|\Psi_{\alpha\taubar}|\leq C_2^{\frac{1}{2}} (\Psi_{\alpha\alphabar})^{\frac{1}{2}}.
\end{align}
We choose $h$, so that, at $\po$, 
\begin{align}
	h_{\alpha\alpha}=\frac{\sigma}{2 \ \text{ Diameter}(\MR)}
	\label{haa}
\end{align}
and 
\begin{align}
	h_{\alpha\tau}=0,\ \ h_{\tau\tau}=0.
\end{align} We remind that we have already required that $\nabla h(\po)=0$. These conditions guarantee $|\nabla h|\leq \frac{\sigma}{2}$ and, therefore, $h\in \spaceQ_{\sigma-2\delta}$.

In the following, we denote $x=\Psi_{\alpha\alphabar}$ and derive a positive lower bound estimate for $x$. From (\ref{618517}), we get
\begin{align}
	\frac{m\sigma}{2\text{ Diameter}(\MR)}\leq 
	\left(
	x^{\frac{1}{2}}+|v_\alpha^m| C_2^{\frac{1}{2}}
	\right)^2.
\end{align}Then a straightforward computation gives
\begin{align}
	\sqrt{m}
	\left(
	  \left(\frac{\sigma}{2\text{ Diameter}(\MR)}\right)^{\frac{1}{2}}
	-\frac{\sqrt{m} \sigma C^{\frac{1}{2}}_2}{\uglb-m\sigma}
	\right)\leq x^{\frac{1}{2}}.
\end{align}
We choose 
\begin{align}
	m=\frac{1}{2}\max\left\{\frac{\uglb}{2\sigma}, \frac{(\uglb)^2}{32 \sigma C_2\text{ Diameter}(\MR)}\right\};
\end{align}then, 
\begin{align}
x\geq \frac{m\sigma}{8 \text{ Diameter}(\MR)}	.
\end{align}
Summing up, we can find a constant $C$ depending on $\text{ Diameter}(\MR)$ and $\uglb$ so that
\begin{align}
	\Psi_{\alpha\alphabar}\geq \frac{\sigma}{C}, 
\end{align} and, therefore,
\begin{align}
	\left(
	\Phi_{\abbar}
	\right)\geq \frac{\sigma}{C}.
	\label{619529}
\end{align}

{\bf Step 3.2.} According to the definition of $\sigma$, the robustness of \qcconvexity\  of ${\Phi_\eqe^l}$ is greater than $2\sigma$ on $\partial \MR$. So the robustness of \qcconvexity\ of $\Psi$ is greater than $\sigma$ on $\partial \MR$ since $q\in \spaceQ_{\frac{\sigma}{2}}$.

Then, at any point ${\bf z}\in \MR$, we can choose coordinate $(z^\alpha, \tau)$, so that 
$\Psi_\tau\neq 0$, and, similar to $\MA, \MB$ we define
 \begin{align} \label{619530}
	\MA^\Psi_{\alpha\betabar}=\Psi_{\alpha\betabar} -\frac{\Psi_{\alpha}}{\Psi_\tau}\Psi_{\tau\betabar}
	-\frac{\Psi_{\betabar}}{\Psi_\taubar}\Psi_{\alpha\taubar}
	+\frac{\Psi_{\alpha}\Psi_{\betabar}}{|\Psi_{\tau}|^2}\Psi_{\tautaubar},\\
	\label{619531}
	\MB^\Psi_{\alpha\beta}=\Psi_{\alpha\beta} -\frac{\Psi_{\alpha}}{\Psi_\tau}\Psi_{\tau\beta}
	-\frac{\Psi_{\beta}}{\Psi_\tau}\Psi_{\alpha\tau}
	+\frac{\Psi_{\alpha}\Psi_{\beta}}{(\Psi_{\tau})^2}\Psi_{\tau\tau},
\end{align}
and 
\begin{align}
	\MK^{\Psi}=\MB^\Psi\overline{(\MA^\Psi)^{-1}}\ \overline{\MB^\Psi}(\MA^\Psi)^{-1}.
\end{align}
Then Lemma \ref{lemma:Robustness_positive_implies_K<1} implies 
\begin{align}
	\MK^\Psi\leq 1-\frac{\sigma}{C}\ \ \ \ \  \ \ \  \text{ on }\partial\MR,
\end{align}for a constant $C>0$. 
Therefore 
\begin{align}
	\tr \left(
	I-\MK^\Psi
	\right)^{-1}\leq \frac{n C}{\sigma}\ \ \ \ \  \ \ \  \text{ on }\partial\MR.
\end{align}
Then, providing $\eqe$ is small enough, we can use Proposition \ref{prop:Main_Computation} and get
\begin{align}
	\tr \left(
	I-\MK^\Psi
	\right)^{-1}\leq \frac{n C {\widetilde{C}}}{\sigma}\ \ \ \ \  \ \ \  \text{ in }\MR.
\end{align}Here ${\widetilde{C}}$ is a constant depending on the $C^2$ norm of $\Psi$.
So we get
\begin{align}
	\MK^\Psi\leq 1-\frac{\sigma}{n C {\widetilde{C}}}.
	\label{619536}
\end{align}
With estimates (\ref{619529}) and (\ref{619536}), we can apply Lemma \ref{lemma_Metric_and_Q_convert_to_Degree} and get
 the robustness of \qcconvexity\ of $\Psi$ is greater than $\frac{\sigma^2}{C_0}$.
 
 {\bf Step 4.} Now, we know for any $q\in\spaceQ_{\sigma-2\delta}$, the robustness of \qcconvexity\ of ${\Phi_\eqe^l}-q$ is greater than $\frac{\sigma^2}{C_0}$. Therefore, the robustness of \qcconvexity\ of ${\Phi_\eqe^l}$ is greater than $\sigma-2\delta+\frac{\sigma^2}{C_0}$. So we choose 
 \begin{align}
 	\delta = \frac{\sigma^2}{4 C_0}.
 \end{align}Then the robustness of \qcconvexity\ of ${\Phi_\eqe^l}$ is greater than $\sigma$. Therefore $L$ is right-closed, and, as a result $L=[0,1]$, providing $\eqe$ is small enough so that the conditions of Proposition \ref{prop:Main_Computation} and \ref{prop:Gradient_Estiamte_Apriori} are satisfied.
 
 We let $\lambda=1$ and get the robustness of \qcconvexity of solution $\Phi^\eqe$ to Problem \ref{prob:Perturbed_Problem_HessianQuotient} is greater than $\sigma$.
 
 \subsection{Improving the Convexity Estimate}
 \label{sec:sub:Improving_Convexity}
 The estimate of section \ref{sec:sub:Convexity_Deformation} depends on the geometry of a family of \cconvex\ rings. In this section, we show that the estimate can be improved so that it only depends on the geometry of $\MR$.
 
 According to Lemma \ref{lemma:GradientNorm_Lower_Boundary}, for a constant $ \bglb$, 
 \begin{align}
 	|\nabla {\Phi^\eqe}| > \bglb\ \ \ \ \  \ \ \ \text{ on }\partial\MR;
 \end{align} then, using Lemma \ref{lemma:Boundary_Convex_implies_Modulus}, the modulus of \qcconvexity\ of ${\Phi^\eqe}$ is greater than a constant $\mu_g$ on $\partial\MR$; then, we apply Lemma \ref{lemma:ModulusPositive_implies_RobustnessPositive} and get the robustness of \qcconvexity\ of ${\Phi^\eqe}$ is greater than $\rho_{\partial\MR}$ on ${\partial\MR}$, for a constant $\rho_{\partial\MR}$ depending on the geometry of $\MR$.
 
 Let 
 \begin{align}
 	\tilde{\sigma}=\frac{1}{2}\min\left\{{\rho_{\partial\MR}},\ \uglb \right\},
 \end{align} and we will show the robustness of \qcconvexity\ of ${\Phi^\eqe}$ is greater than $\tilde{\sigma}$. Let $\delta$ be a small constant in $(0, \min\left\{\frac{\sigma}{4}, \frac{\tilde\sigma}{4}\right\}]$ 
 to be determined, and let
 \begin{align}
 	\begin{split}
 	I^\eqe=\left\{
 	\theta\in[0, \tilde{\sigma}]\ \big| \right.& \text{ the robustness of \qcconvexity\ of ${\Phi^\eqe}-q$ is greater than $\delta$,} \\
&\ \ \ \ \  \ \ \  \ \ \ \ \  \ \ \ \ \ \ \ \  \ \ \ \ \ \ \ \  \ \ \ \ \ \ \ \  \ \ \ \ \ \ \ \  \ \ \ \ \ \ \ \  \ \ \ \ \ \ \ \  \ \ \ \left.	  \text{with $q\in \spaceQ_\theta$ } 
 	\right\};
 	\end{split}
 \end{align}
 we will show $I^\eqe=[0, \tilde{\sigma}]$.  
 
 First, we notice that $I^\eqe\supset [0, \frac{\sigma}{4}]$. It's also easy to see that $I^\eqe$ is open. Most importantly, we can show it's right-closed; then $I^\eqe=[0, \tilde{\sigma}]$ follows.
 
For the right-closeness, we assume $I^\eqe\supset[0, l)$; then we show it contains $l$. This part is completely parallel to the proof of section \ref{sec:sub:Convexity_Deformation}. 
 
 \subsection{Estimates Related to the Rank Estimate }
 \label{sec:sub:Form_Estimate}
 In this section, we show for a positive constant $c$, which depends on the geometry of $\MR$,  
 \begin{align}
 	(\iu\ddbar\Phi^\eqe)^{n-1}\wedge\bbform\geq c\ \bbform^n,
 	\label{628538}
 \end{align} 
 and
 \begin{align}
 	\iu\ddbar \left(
 	e^{\Phi^\eqe}
 	\right)\geq c \ \bbform.
 	\label{715544}
 \end{align}
 When $\eqe$ converges to $0$, these estimates converge to the estimates for the solution $\Phi$ of the homogenous equation; the estimates then indicate the rank of $\iu\ddbar\Phi$ is $n-1$ in the weak sense.
 
We first prove (\ref{628538}). We choose a set of coordinates $(z^{\alpha}, \tau)$ so that at a point $\po$, 
 $\Phi^\eqe_{\alpha}=0$, $\Phi^\eqe_{\abbar}=\delta_{\alpha\beta}\lambda_\alpha$ and $\MG_{\ijbar}=\delta_{\ij}$. Let $\MH=(\Phi^\eqe_{\ijbar})$; then,
 \begin{align}
 	\MH=\left(
 	\begin{array}{cccc}
 		\lambda_1&  &  &\Phi^\eqe_{1\taubar}  \\
 		&  \lambda_2 &   &\Phi^\eqe_{2\taubar} \\
 		& 						  & \ddots  &  \vdots\\
 	\Phi^\eqe_{\tau\overline{1}}	&\Phi^\eqe_{\tau\overline{2}}  &\cdots  &\Phi^\eqe_{\tautaubar}
 	\end{array}
 	\right).
 \end{align}
 Direct computation gives 
 \begin{align}
 	\tr (\MH^{-1})=\frac{1}{\tred} 
 	\left(
 	1+\sum_{\mu=1}^{n-1}\frac{|\Phi^\eqe_{\mu\taubar}|^2}{\lambda_\mu^2}
 	\right) +\sum_{\mu=1}^{n-1}\frac{1}{\lambda_\mu}
 \end{align}
 and
 \begin{align}
 	\det(\MH)=\tred \prod_{\mu=1}^{n-1}\lambda_\mu,
 \end{align}where
 \begin{align}
 	\tred=\Phi^\eqe_{\tautaubar}-\sum_{\mu=1}^{n-1}\frac{|\Phi^\eqe_{\mu\taubar}|^{2}}{\lambda_k}.
 \end{align}
 Then, using equation (\ref{eq:Matrix_Form_Main_Perturbation_Equation}), we get
 \begin{align}
  \frac{1}{\tred} 
 \left(
 1+\sum_{\alpha=1}^{n-1}\frac{|\Phi^\eqe_{\alpha\taubar}|^2}{\lambda_k^2}
 \right)\leq \frac{1}{\eqe}.
 \end{align} Therefore, $\eqe\leq \tred$, and 
 \begin{align}
 	\det(\MH) \geq \eqe  \prod_{\mu=1}^{n-1}\lambda_\mu.
 \end{align}
 
 We have shown the robustness of \qcconvexity\ of ${\Phi^\eqe}$ has a uniform positive lower bound, so $\lambda_\alpha\geq\sigmalamb$, for $\alpha\in\{1,\ ...\ , n-1\}$,  for a constant $\sigmalamb$, which depends on the geometry of $\MR$.  Then, it follows 
 \begin{align}
 	\det(\MH) \geq \eqe  (\sigmalamb)^{n-1}.
 \end{align}Therefore,
 \begin{align}
	(\iu\ddbar\Phi^\eqe)^{n}\geq \eqe  (\sigmalamb)^{n-1}\ \bbform^n.
\end{align} 
Using equation \ref{eq:Matrix_Form_Main_Perturbation_Equation}, we know at the point $\po$
 \begin{align}
	(\iu\ddbar\Phi^\eqe)^{n-1}\wedge\bbform \geq  (\sigmalamb)^{n-1}\ \bbform^n.
	\label{estimate:sigma_n-1}
\end{align} Since the above inequality is independent of the choice of coordinate, we know it's valid anywhere in $\MR$.

We can also show $e^{ {\Phi^\eqe}}$ is a strictly pluri-subharmonic function and derive an estimate for the minimum eigenvalue of the matrix 
$[(e^{\Phi^\eqe})_{\ijbar}]$. 

Direct computation gives
\begin{align}
	(e^{ {\Phi^\eqe}})_{\ijbar}=e^{ {\Phi^\eqe}}
	\left(
	\Phi^\eqe_{\ijbar}+\Phi^\eqe_i \Phi^\eqe_\jbar
	\right).
\end{align} We adopt the coordinates and the notations for the previous estimate.  Then
\begin{align}
	\det [(e^{ {\Phi^\eqe}})_{\ijbar}]=e^{ n{\Phi^\eqe}} \det
	\left(
	\tH
	\right),
	\label{623549}
\end{align}
where
\begin{align}
	\tH=\MH+\diag (0,\ ...\ , |{\Phi^\eqe_\tau}|^2).
\end{align}
Again, by direction computation, we have
\begin{align}
	\det(\tH)\geq |{\Phi^\eqe_\tau}|\prod_{\mu=1}^{n-1}\lambda_\mu. 
\end{align}
According to the construction of the lower solution in Appendix \ref{sec:B_boundary_Barrier}, we have ${\Phi^\eqe}\geq 0$, so the right-hand side of $ (\ref{623549})\geq \det (\tH)$. Therefore, we have
\begin{align}
	\det[(e^{ {\Phi^\eqe}})_{\ijbar}]\geq |{\Phi^\eqe_\tau}|^2 (\sigmalamb)^{n-1}.
	\label{lempert_2_detestimate_epsilon}
\end{align}
Combining with the $C^0$ and $C^2$ estimates for ${\Phi^\eqe}$, we have  the lower bound estimate for the minimum eigenvalue of  $[(e^{ {\Phi^\eqe}})_{\ijbar}]$.

\section{Approximating to the Homogenous Equation}
\label{sec:Approximate_HCMA}
In this section, we let $\eqe$ go to zero and get an estimate for the modulus of \cconvexity \ of level sets of solutions to Problem \ref{prob:Main_Problem_HCMA_ring}.

Previously, we have shown that  when $\eqe$ is small enough,  the robustness of \qcconvexity\ of ${\Phi^\eqe}$ is greater than $\sigmalamb$, for a constant $\sigmalamb$, depending on the geometry of $\MR$. Then using Lemma \ref{lemma:Modulus_implies_levelsets_Convex}, we know the modulus of \cconvexity\ of level sets of $\Phi^\eqe$ is greater than a positive constant ${\sigma}$, which depends on the geometry of $\MR$.  

According to the equation (\ref{eq:SigmaQuotient_Problem_Perturbation}) and the $C^2$ estimate for ${\Phi^\eqe}$ in Appendix \ref{sec:Appendix_C11Estimate}, we have
\begin{align}
	\det({\Phi^\eqe}_{\ijbar})\rightarrow 0\ \ \ \ \  \ \ \ \text{ as } \eqe\rightarrow 0.
\end{align}So $\Phi^\eqe\rightarrow \Phi^0$ at least in $C^0$ norm. Then using interpolation, we know ${\Phi^\eqe}\rightarrow \Phi^0$ in $C^1$ norm. 
Therefore the uniform positive lower bound estimate for $|\nabla\Phi^\eqe|_\MG$ implies a positive lower bound for $|\nabla\Phi^0|_\MG$.

We also know $\Phi^0$ is a $C^{1,1}$ function, because of the uniform $C^{1,1}$ estimate of Appendix \ref{sec:Appendix_C11Estimate}. This together with the lower bound estimate for $|\nabla\Phi^0|_\MG$ implies level sets of $\Phi^0$ are $C^{1,1}$ hypersurfaces. 

At a point $\po\in\MR$, we choose coordinates $(z^\alpha, \tau)$, with $\tau=t+\iu s,$ so that $\po=0$, the metric is $\sum_i dz^i\otimes \overline{dz^i}$ and
\begin{align}
	\partial_{\alpha}\Phi^0=0, \ 	\partial_{s}\Phi^0=0, \ 	\partial_{t}\Phi^0>0,
	\label{lempert_1_sigmaestimate_epsilon}
\end{align} at $\po$; therefore, when $\eqe$ is small enough, 
\begin{align}
	\partial_t\Phi^\eqe>0,
\end{align} around $\po$. Let $\MT_\delta=\left\{\sum_{\alpha}|z^\alpha|^2+s^2 <\delta^2, \ |t|<\delta \right\}$.
For $\eqe$ and $\delta$ small enough, the level set $	\left\{
{\Phi^\eqe}={\Phi^\eqe}(\po)
\right\}\cap {\MT_\delta} $ is the graph of a $C^{1,1}$ function: 
\begin{align}
	\begin{split}
	&\left\{
	{\Phi^\eqe}={\Phi^\eqe}(\po)
	\right\}\cap 
{\MT_\delta} =
	\left\{
	(z^\alpha, \rho^\eqe(z^\alpha, s)+\iu s)\ \big|\   \sum_{\alpha}|z^\alpha|^2 +s^2< \delta^2
	\right\}.
	\end{split}
\end{align} In addition, we know 
\begin{align}
	|\nabla {\Phi^\eqe}| \leq C_2(\eqe+\delta), \ \ \ \ \  \ \ \ \text{ in }T_\delta.
\end{align}

In the following, we will show when $\eqe$  and $\delta$ are small enough,  
\begin{align}
	\rho^\eqe(z^\alpha, 0)\geq \frac{\sigma}{C} \sum_{\alpha}|z^\alpha|^2
\end{align}for a constant $C$ when $\sum_{\alpha}|z^\alpha|^2\leq \delta$. Most importantly, we need to show that $\delta$ is independent of $\eqe$, so when $\eqe$ converges to zero, 
we get the convexity estimate for $\rho^0$.

We choose $q=2\Ree(z^\alpha\partial_\alpha {\Phi^\eqe} )$; then,
\begin{align}
	\partial_\alpha
	\left(
	{\Phi^\eqe}-q
	\right)({\bf z})=0\ \ \ \ \  \ \ \ \text{ for }\alpha\in \{ 1,\ ...\ , n-1 \}.
\end{align}
It follows that
\begin{align}
	{T^\EC_{{\Phi^\eqe}-q, {\bf z}}}=\{\zeta\in\EC^n| \zeta^n-z^n=0\}.
\end{align}Because the robustness of \qcconvexity\ of ${\Phi^\eqe}$ is greater than $\sigma$, ${\Phi^\eqe}-q$ is \sqcconvex\ of robustness greater than $\frac{\sigma}{2}$ providing $\delta$ is small enough. Therefore, the second order Taylor expansion of ${\Phi^\eqe}|_{\{z^n=a\}}$ at any point in $\{|z^\alpha|^2\leq \delta^2\}$, for $|a|<\delta, a\in\EC$,  is convex of modulus greater than $\frac{\sigma}{C_1}$, for  a constant $C_1$ depending on the diameter of $\MR$, when $\delta$ is small enough. The convexity estimate for $\rho^\eqe(z^\alpha, 0)$ then follows since
\begin{align}
	\Phi^\eqe(z^\alpha, \rho^\eqe(z^\alpha, 0))=0.
\end{align}

For the estimate (\ref{estimate:sigma_n-1}), we let $\eqe$ converge to zero. Since $\Phi^\eqe\rightarrow \Phi^0$ in $C^1$ norm and $\Phi^\eqe$ are all plurisubharmonic functions, we have
 \begin{align}
	(\iu\ddbar\Phi^0)^{n-1}\wedge\bbform \geq  \frac{1}{C}\ \bbform^n,
	\label{estimate_sigmaN-1_zero}
\end{align} in the weak sense for a constant $C$ depending on the geometry of $\MR$.

For the estimate (\ref{715544}), since $e^{{\Phi^\eqe}}$ is a plurisubharmonic function and $\Phi^\eqe\rightarrow \Phi^0$ in $C^1$ norm,  we have
\begin{align}
	\iu\ddbar(e^{ {\Phi^0}})\geq \frac{1}{C}\bbform
\end{align}in weak sense for a constant $C$ depending on the geometry of $\MR$.
\appendix

\section{Linear Algebra Lemmas}
\label{sec:Algebra_Lemma}
In this appendix, we prove some linear algebra lemmas. The goal is to establish relations  between different descriptions of the convexity and \qcconvexity. 

 Lemma \ref{lemma:Equivalent_Quadratic_Modulus_Degree} shows that for quadratic polynomials defined on $\EC^{n-1}$ the degree and the modulus of convexity are equivalent. 
\begin{lemma}	\label{lemma:Equivalent_Quadratic_Modulus_Degree}
	The modulus of convexity of a  quadratic polynomial on $\EC^{n-1}$ is greater than $\mu\geq 0$ if and only if its degree of convexity is greater than $\mu$.
\end{lemma}
\begin{proof} We note that any quadratic polynomial $P$ on $\EC^{n-1}$ has the following form
		\begin{align}
		P(z)=A_{\abbar} z^\alpha z^\betabar+\Ree(B_{\alpha\beta} z^\alpha z^\beta)+L,
	\end{align} where
	$A$ is hermitian, $B$ is symmetric and $L$ is a linear function.
	Suppose $G_{\alpha\betabar} dz^\alpha\otimes \overline{dz^\beta}$ is the metric on $\EC^{n-1}$.  We only need to show
	\begin{align}
		\sup_{B\overline{G^{-1}}\ \overline{B}G^{-1}\leq 1}\Ree(B_{\alpha\beta}{ z^\alpha z^\beta}) =G_{\alpha\betabar} z^\alpha z^\betabar.
	\end{align}
	To simplify the computation, we change coordinates, so that $G_{\abbar}=\delta_{\ab}$. Then we need to show 
	\begin{align}
			\sup_{B\overline{B}\leq 1}\Ree(B_{\alpha\beta}{ z^\alpha z^\beta}) = |z|^2.
	\end{align}
	
	First, we fix $(z^\alpha)$ and choose $B_{\alpha\beta}=
	\frac{\overline{z^\alpha z^\beta}}{|z|^2}$, then
	\begin{align}
		\Ree (B_{\alpha\beta} z^\alpha z^\beta)= |z|^2.
	\end{align} So 
		\begin{align}
		\sup_{B\overline{B}\leq 1}\Ree(B_{\alpha\beta }{z^\alpha z^\beta}) \geq |z|^2.
			\label{0607A4}
	\end{align}
	
	Then for any symmetric $B$,  with $B\overline{B}\leq 1$, after a unitary change of coordinates, we can make $B=\diag(\sigma_1, \ ...\ , \sigma_{n-1})$, with $\sigma_\gamma\in\ER^{\geq0}$ and $\sigma_\gamma\leq 1$. This is the Autonne-Takagi factorization, which is Corollary 4.4.4(c) of \cite{HornMatrix} (or see Lemma A.3 of \cite{Hu22Nov}).
	Therefore
	\begin{align}
		\Ree (B_{\alpha\beta} z^\alpha z^\beta)= \sum \sigma_\gamma \Ree ({z^\gamma})^2\leq |z|^2;
	\end{align}as a result,
		\begin{align}
			\sup_{B\overline{B}\leq 1}\Ree(B_{\alpha\beta }{z^\alpha z^\beta}) \leq |z|^2.
			\label{0607A6}
	\end{align}
	(\ref{0607A4}) and (\ref{0607A6}) together gives the equality.
\end{proof}
We would also need the following lemma to estimate the robustness of \qcconvexity. 
\begin{lemma}
	\label{lemma:Algebra_Robustness_Half_perturbation}
	Suppose that $G_{\alpha\betabar} dz^\alpha\otimes \overline{dz^\beta}$ is a constant coefficient metric on $\EC^{n-1}$ and that the modulus of convexity of 
	\begin{align}
		\MP(z^\alpha)=A_{\abbar} z^\alpha \zbetabar+\Ree \left(
		B_{\alpha\beta} z^\alpha z^\beta
		\right)
	\end{align}is greater than $\delta$.  Let $G=(G_{\alpha\betabar} )$. Then, for any  Hermitian matrix $V$ with
	\begin{align}
	V\leq \frac{\delta}{2} G 
	\label{0703A9}
	\end{align} and any symmetric matrix $W$ with
	\begin{align}
	 W\overline{G^{-1}}\ \overline{W}G^{-1} \leq \frac{\delta^2}{4},
	\end{align} we have that
	\begin{align}
		\MP -V_{\abbar} z^\alpha \zbetabar-\Ree \left(
		W_{\alpha\beta} z^\alpha z^\beta
		\right)
	\end{align} is strongly convex.
\end{lemma}
The idea of this lemma is very intuitive in the case of $n=2$. In this case 
$
 A\in \ER^+ \text{ and } B \in \EC.
$ $\MP$ is convex of modulus $>\delta$ if and only if $(B, A)$ stays in the cone $\left\{
A-\delta>|B|
\right\}$, and $\MP$ is convex of degree  $>\delta$ if and only if for any $\beta\in[-\delta, \delta]$, $(B, A)$ stays in the cone $\left\{
A>|B-\beta|
\right\}$.  Lemma \ref{lemma:Equivalent_Quadratic_Modulus_Degree} basically says
\begin{align}
	\left\{
	|B|<A-\delta
	\right\}
	=
	\bigcap_{|\beta|\leq \delta } \left\{
	|B-\beta|< A
	\right\}.
\end{align} For Lemma \ref{lemma:Algebra_Robustness_Half_perturbation}, we need to show when $(V, W)\in \{A\geq 0, |B|\leq \delta -A\}$
\begin{align}
		\left\{
	|B|<A-\delta
	\right\} \subset 	\left\{
	|B-W|<A-V
	\right\}. 
\end{align}As illustrated by Figure \ref{fig:Half_Degree_Perturbation}, the blue cone covers the cone $	\left\{
|B|<A-\delta
\right\}$ providing the corner of the blue cone stays in the cone $\{A\geq 0, |B|\leq \delta -A\}$, the yellow area.

	 \begin{figure}[h]
	\centering  
	\includegraphics[height=3.5cm]	{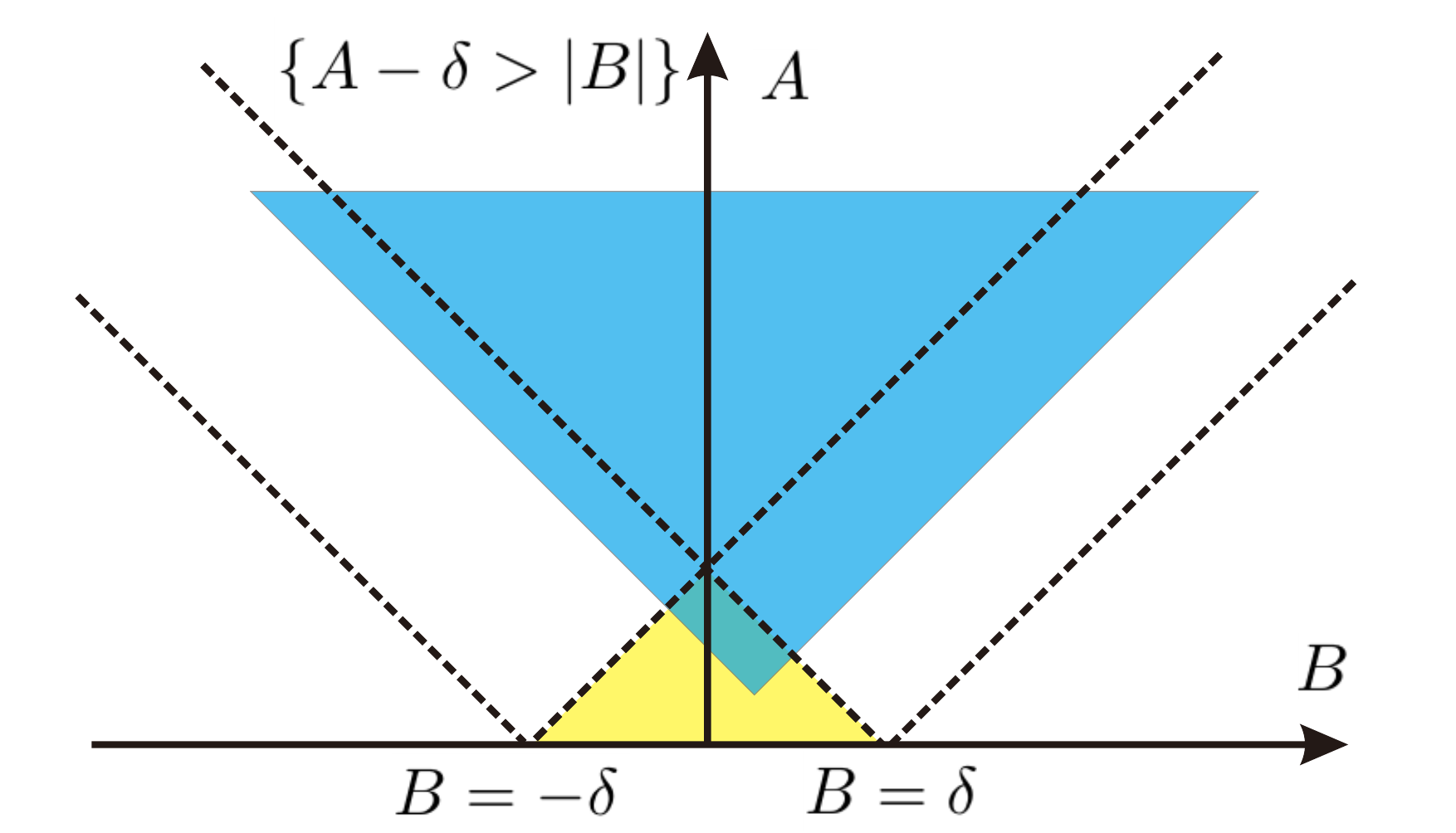}
	\caption{Visualizing the Convexity Estimate: Lemma \ref{lemma:Algebra_Robustness_Half_perturbation}}
	\label{fig:Half_Degree_Perturbation}  
\end{figure}

\begin{proof}
	[Proof of Lemma \ref{lemma:Algebra_Robustness_Half_perturbation}]
	Actually, the proof is quite simple. We denote
	\begin{align}
		\tilde\MP=	
			\MP -V_{\abbar} z^\alpha \zbetabar-\Ree \left(
			W_{\alpha\beta} z^\alpha z^\beta
			\right).
	\end{align}
	Since $\tilde\MP$ is a quadratic polynomial, it's strongly convex if and only if 
$\tilde \MP(z)>0$ for $z\neq 0$. So, we need to show 
\begin{align}
	V_{\abbar} z^\alpha \zbetabar+\Ree \left(
	W_{\alpha\beta} z^\alpha z^\beta
	\right)\leq \delta g_{\alpha\betabar} z^\alpha \overline{z^\beta}. 
	\label{0608A14}
\end{align}
 According to the proof of Lemma \ref{lemma:Equivalent_Quadratic_Modulus_Degree}, we have
 \begin{align}
 	\Ree \left(
 	W_{\alpha\beta} z^\alpha z^\beta
 	\right)\leq  \frac{\delta}{2}  g_{\alpha\betabar} z^\alpha \overline{z^\beta}. 
 \end{align} The following inequality follows from the assumption (\ref{0703A9}):
 \begin{align}
 V_{\abbar} z^\alpha \zbetabar\leq \frac{\delta}{2}	g_{\alpha\betabar} z^\alpha \overline{z^\beta}.
 \end{align} The two inequalities above implies (\ref{0608A14}) is valid. 
\end{proof}
With the preparation above, we can establish the relation between the robustness  and  modulus of \qcconvexity. We have the following Lemma \ref{lemma:ModulusPositive_implies_RobustnessPositive}, which proves the robustness of \qcconvexity\ is positive providing the modulus of \qcconvexity\ is positive, and Lemma \ref{lemma:RobustnessPositive_implies_ModulusPositive}, which proves an estimate in the opposite direction.
\begin{lemma}
	\label{lemma:ModulusPositive_implies_RobustnessPositive}
	 Let $\MR$ be the ring-shaped domain of Problem \ref{prob:Perturbed_Problem_HessianQuotient}. Suppose the modulus of \qcconvexity\ of $\Phi\in C^\infty(\overline{\MR})$ is greater than $m$ at a point $\po\in \overline{\MR}$, for $m>0$. Then the robustness of \qcconvexity\ of $\Phi$ is greater than $\frac{m^2}{C}$ for a constant $C$ depending on the $C^2$ norm of $\Phi$ and the thickness of $\MR$.
\end{lemma}
 \begin{proof}[Proof of Lemma \ref{lemma:ModulusPositive_implies_RobustnessPositive}]
 	Let the metric on $\EC^n$ be $\sum_idz^i\otimes \overline{dz^i}$.
 	
 	 We need to show for a small positive constant $\rho$, $\Phi-q$ is \sqcconvex\ providing $q\in \spaceQ_\rho$. 
 	 
 	 First, according to the definition of modulus of \qcconvexity, we have $|\nabla\Phi(\po)|>m$, so
 	 \begin{align}
 	 	|\nabla(\Phi-q)(\po)| >\frac{m}{2} 
 	 	\label{626A18}
 	 \end{align} for any $q\in \spaceQ_\rho$, providing $\rho\leq \frac{m}{8}$.
In the following, we show when $\rho$ is small enough, the second order Taylor expansions of $\Phi|_{T^\EC_{\Phi,\po}}$ and  $(\Phi-q)|_{\Phi_{T^\EC_{\Phi-q,\po}}}$ at $\po$ are close to each other. Then the \qcconvexity\ of $\Phi$ implies the \qcconvexity\ of $\Phi-q$.
 
 We choose a coordinate $(z^\alpha, \tau)$, so that $\partial_\alpha\Phi(\po)=0$ and the metric is $\sum_i dz^i\otimes \overline{dz^i}$. Let $\Psi=\Phi-q$; then,
 \begin{align}
 	|\Psi_\tau(\po)|\geq |\Phi_\tau(\po)|-|q_\tau(\po)|\geq|\nabla\Phi(\po)|-|\nabla q(\po)|\geq  m-\frac{m}{8}\geq \frac{m}{2}.
 	\label{0703A19}
 \end{align}
Let
 	\begin{align}
 	\MA_{\abbar}=\Psi_{\abbar}(\po)-\Psi_{\alpha\taubar}(\po)\overline{
 		\left(
 		\frac{ \Psi_\beta(\po)}{ \Psi_\tau(\po)}
 		\right)}
 	-\Psi_{\tau\betabar}(\po)
 	\left(
 	\frac{ \Psi_\alpha(\po)}{ \Psi_\tau(\po)}
 	\right)
 	+\Psi_{\tautaubar}	(\po)\left(
 	\frac{ \Psi_\alpha(\po)}{ \Psi_\tau(\po)}
 	\right)\overline{
 		\left(
 		\frac{ \Psi_\beta(\po)}{ \Psi_\tau(\po)}
 		\right)},
 	\\
 	\MB_{\ab}=\Psi_{\ab}(\po)-\Psi_{\alpha\tau}(\po){
 		\left(
 		\frac{ \Psi_\beta(\po)}{ \Psi_\tau(\po)}
 		\right)}
 	-\Psi_{\tau\beta}(\po)
 	\left(
 	\frac{ \Psi_\alpha(\po)}{ \Psi_\tau(\po)}
 	\right)
 	+\Psi_{\tau\tau}(\po)	\left(
 	\frac{ \Psi_\alpha(\po)}{ \Psi_\tau(\po)}
 	\right){
 		\left(
 		\frac{ \Psi_\beta(\po)}{ \Psi_\tau(\po)}
 		\right)}.
 \end{align}
 Then the second order Taylor expansion of $\Psi|_{{T^\EC_{\Psi,\po}}}$ at $\po$ is
 \begin{align}
 	\MA_{{\alpha\betabar}} z^\alpha\overline{z^\beta}+\Ree(\MB_{\ab}z^\alpha z^\beta).
 \end{align}
 Because $q$ is pluriharmonic, we have $\Psi_{\ijbar}=\Phi_{\ijbar}$. Then using  $\Phi_{\alpha}(\po)=0$, we have
 \begin{align}
 	\MA_{\abbar}=\Phi_{\alpha\betabar} +\Phi_{\alpha\taubar}\overline{
 	\left(
 	\frac{q_\beta}{\Psi_\tau}
 	\right)}
 	+\Phi_{\tau\betabar}{
 		\left(
 		\frac{q_\alpha}{\Psi_\tau}
 		\right)}+	
 		\Phi_{\tau\taubar}
 		\frac{q_\alpha q_\betabar}{|\Psi_\tau|^2}.
 \end{align}We denote that $\MA_{\abbar}=\Phi_{\abbar}+\MU_{{\alpha\betabar}}$. Similarly, for $\MB$, we have
  \begin{align}
 	\MB_{\ab}=\Psi_{\alpha\beta} 
 	+\Psi_{\alpha\tau}{
 		\left(
 		\frac{q_\beta}{\Psi_\tau}
 		\right)}
 	+\Psi_{\tau\beta}{
 		\left(
 		\frac{q_\alpha}{\Psi_\tau}
 		\right)}+	
 	\Psi_{\tau\tau}
 	\frac{q_\alpha q_\beta}{(\Psi_\tau)^2};
 \end{align} we denote
 \begin{align}
 	B_{\ab}=\Phi_{\ab}+\MV_{\ab}.
 \end{align}
 So, 
 \begin{align}
 	&\text{the second order Taylor expansion of }\Psi|_{T^\EC_{\Psi,\po}} \text{ at $\po$}
 	\label{626A25}
 	\\
 	=&\text{the second order Taylor expansion of }\Phi|_{T^\EC_{\Phi,\po}} \text{ at $\po$}+\MU_{\abbar} z^\alpha\overline{z^\beta} +\Ree
 	\left((
 	\MV_{\ab}-q_{\ab})z^\alpha z^\alpha
 	\right).
 \end{align}
 Because $\Phi$ is \sqcconvex\ of modulus greater than $m$, we know the second order Taylor expansion of $\Psi|_{T^\EC_{\Psi,\po}}$ at $\po$ is convex of modulus greater than $m$. Then using Lemma 
 \ref{lemma:Algebra_Robustness_Half_perturbation}, we know when 
 \begin{align}
 	(\MU_{\abbar}) \leq \frac{m }{2} 
 	\label{626A27}
 \end{align}
 and 
 \begin{align}
 \left(
 \sum_\gamma\MV_{\alpha\gamma} \overline{\MV_{\gamma\beta}}
 \right)	\leq \frac{m^2}{4},
 	\label{626A28}
 \end{align}
 (\ref{626A25}) is strongly convex.  In the following, we show (\ref{626A27}) and (\ref{626A28}) are satisfied when $\rho$ is small enough;
we need to use the following linear algebra Lemma:
\begin{lemma}
	\label{lemma:matrix_size_control}
	Suppose $A$ is a $k\times k$ Hermitian matrix. Then $A<1$, providing 
	\begin{align}
		|A_{\ijbar}|<\frac{1}{k^2}
		\label{626Asmall}
	\end{align} for any $i, j\in \{1,\ ...\ ,k\}$. Suppose $B$ is a $k\times k$ complex-valued matrix. Then 
	$B{B}^\dagger<1$ providing $|B_{ij}|<\frac{1}{k^\frac{3}{2}}$ for any $i, j\in \{1,\ ...\ ,k\}$.
\end{lemma}
\begin{proof}
	[Proof of Lemma \ref{lemma:matrix_size_control}]
	Since $B{B}^\dagger$ is Hermitian and 
	\begin{align}
		(B{B}^\dagger)_{\ijbar}=\sum_p B_{ip} \overline{B_{pj}},
	\end{align} we only need to prove the estimate for $A$. Suppose the maximum eigenvalue of $A$ is $\sigma$. Then we can find a unit vector
	\begin{align}
		{\bold v}=(v_1,\ ...\ v_k)
	\end{align}so that
	\begin{align}
		{\bold v}A{\bold v}^\dagger=\sigma.
	\end{align} Here  unit vector means ${\bold v}{\bold v}^\dagger=1$.
It follows that
\begin{align}
	\sigma\leq \sum_{i,j=1}^k|v_i||A_{\ijbar}| |v_j|\leq  \sum_{i,j=1}^k|A_{\ijbar}|\leq k^2 \max|A_{\ijbar}|.
\end{align} Therefore (\ref{626Asmall}) implies $A<1$.
\end{proof} 

With the lemma above, we need to show $\MU_{\abbar}$ and $\MV_{\ab}$ are both small.
For $\MU_{\abbar}$, we have
\begin{align}
	|\MU_{\abbar}|\leq \rho C_2 (\frac{8}{m}+64\frac{\rho}{m^2}),
	\label{626A34}
\end{align}where $C_2$ depends on the $C^2$ norm of $\Phi$. 
In (\ref{626A34}), we used the estimate (\ref{0703A19}). We already required that $\rho\leq m$, so from (\ref{626A34}), we get
\begin{align}
	|\MU_{\abbar}|\leq 100 C_2 \frac{\rho}{m}.
	\label{626A35}
\end{align} Therefore, to satisfy (\ref{626A27}), we require that
\begin{align}
	\rho\leq \frac{m^2}{ 200 C_2}.
\end{align}
For $\MV_{\ab}$, similarly, we have
\begin{align}
	|\MV_{\ab}|\leq 100 \tct \frac{\rho}{m}.
	\label{626A37}
\end{align} 
Here the constant $\tct$ depends on the $C^2$ norm of $\Phi-q$.

To prove (\ref{626A37}), we need to show the $C^2$ derivatives of $q$ can be controlled by $\rho$ and the thickness of $\MR$. Suppose  $B_r(z)\subset \MR$, where $r$ is the thickness of $\MR$. Then we can estimate $D^2q(z)$ since $Dq$ is harmonic in $B_r(z)$. The following estimate is standard:
\begin{align}
	|D^2q (z)|\leq \frac{C_n}{r} \max_{B_r} |\nabla q|,
\end{align} where $C_n$ depends on the dimension $n$. Since $q$ is quadratic, we know $D^2 q$ is constant. Then it follows that 
\begin{align}
	|D^2q|\leq \frac{C_n \rho}{r}. 
	\end{align}

Since $ \rho\leq m$ and $m$ depends on the $C^2$ norm of $\Phi$, we have $\tct$ depends on the $C^2$ norm of $\Phi$ and the thickness of $\MR$. 

Summing up, (\ref{626A35}) and (\ref{626A37}) can both be satisfied by choosing 
	\begin{align}
		\rho=\frac{m^2}{C},
		\label{626A39}
	\end{align}  where $C$ is a constant depending on the $C^2$ norm of $\Phi$ and the thickness of $\MR$. We choose $C$ large enough so that $\frac{m}{C}\leq \frac{1}{8}$ because $\rho$ depends on the $C^2$ norm of $\Phi$. Then (\ref{626A39}) also implies $\rho\leq\frac{m}{8}$. The lemma is proved.
 \end{proof} 
\begin{lemma}
	\label{lemma:RobustnessPositive_implies_ModulusPositive}
	Suppose that $\Omega$ is a bounded domain with smooth boundaries and that the robustness of \qcconvexity\ of $\Phi\in C^{\infty}(\Omegabar)$ is greater than $\rho$ at $\po\in \Omegabar$.  Then the modulus of \qcconvexity\ of $\Phi$ is greater than $\frac{\rho}{C}$ at $\po$, for a constant $C$ depending on the diameter of $\Omega.$
\end{lemma}
\begin{proof}
	[Proof of Lemma \ref{lemma:RobustnessPositive_implies_ModulusPositive}]
	We assume the Hermitian metric on $\Omega$ is $dz^i\otimes \overline{dz^i}$.
	
	First, we estimate $|\nabla\Phi(\po)|$. The assumption that the robustness of \qcconvexity\ of $\Phi$ is greater than $\rho$ at $\po$  implies that for any linear function $l$ with $|\nabla l|\leq \rho$,  
	\begin{align}
		\nabla(\Phi-l)\neq 0\ \ \ \ \  \ \ \ \text{ at }\po.
	\end{align}So we have $|\nabla\Phi(\po)|>\rho$.
	
	For the convexity estimate, we choose a set of coordinates  $(z^\alpha, \tau)$ so that
	$\po=0$, $\Phi_{\alpha}(\po)=0$ and the metric is still $dz^i\otimes \overline{dz^i}$.  Then $\{\tau=0\}$ is the complex tangent plane of  $\{\Phi=\Phi(\po)\}$ at $\po$.  Let
	\begin{align}
		q_{\bold v} =\Ree 
		\left[
		\sum_{\alpha=1}^{n-1}v_\alpha(z^\alpha)^2\frac{1}{4 \text{ Diameter}(\Omega)}
		\right],
	\end{align}
	for any ${\bold v}=(v_\alpha)\in \EC^{n-1}$, with $|v_\alpha|\leq \rho$, for  $\alpha=1,\ ...\ , n-1$. Then, 
	\begin{align}
		|\nabla q_{\bold v}|\leq \rho \ \ \ \ \  \ \ \  \text{ in } \Omega,
	\end{align} and, therefore,
	$q_{\bold v}\in \spaceQ_\rho$. So, $\Phi-q_{\bold v}$ is \sqcconvex, because of  the condition that the robustness of \qcconvexity\ of $\Phi$ is greater than $\rho$.
	It follows that $\nabla(\Phi-q_{\bold v})\neq0$ in $\Omega$, so level sets of $\Phi-q_{\bold v}$ are smooth surfaces. Since $\nabla q_{\bold v}(\po)=0$,  $\{\tau=0\}$ is also the complex tangent plane of 
	\begin{align}
		\left\{
		\Phi-q_{\bold v}
		=
		(\Phi-q_{\bold v})(\po)
		\right\}
	\end{align} at $\po$.
	Therefore, the second order Taylor expansion of 
	$\Phi|_{\{\tau=0\}}-q_{\bold v}$ at $0$ is strongly convex. This says the degree of convexity of  the second order Taylor expansion of 
	$\Phi|_{\{\tau=0\}}$ at $0$ is greater than $\frac{\rho}{4 \text{ Diameter}(\Omega)}$.
	Then it follows that the modulus of \qcconvexity\ of $\Phi$ is greater than
	\begin{align}
		\rho\cdot \min
		\left\{
		1, \frac{1}{4 \text{ Diameter}(\Omega)}
		\right\}.
	\end{align}
\end{proof}
\begin{lemma}
	\label{lemma_Metric_and_Q_convert_to_Degree}
	Let $g_{\ijbar} \dzijbar$ be a constant coefficient Hermitian metric on $\EC^n$. Let $G=(g_{\ijbar})$. Suppose that $A $ is a hermitian matrix and that $B$ is a symmetric matrix, satisfying
	\begin{align}
		A&>\delta G, \label{510conditionA1}\\
		B{\Abarinverse}\ \Bbar \Ainverse&<1-\delta,\label{510conditionA2}
	\end{align}for a positive constant $\delta$, with $\delta<1$. Then the modulus of convexity of
	\begin{align}
		P(z)=A_{\ijbar} z^i \zjbar+\Ree(B_{ij} z^iz^j)
	\end{align} is greater than $\frac{\delta^2}{2}$.
\end{lemma}
\begin{proof}
	[Proof of Lemma \ref{lemma_Metric_and_Q_convert_to_Degree}]
	From (\ref{510conditionA2}), we easily get 
	\begin{align}
		B{\Abarinverse}\ \Bbar \Ainverse<(1-\frac{\delta}{2})^2
	\end{align}
	and, consequently,
		\begin{align}
		B{{\left[\overline A(1-\frac{\delta}{2})\right]}^{-1}}\ \Bbar \left[A(1-\frac{\delta}{2})\right]^{-1}<1.   			\label{0510A5}
	\end{align}
	According to (\ref{510conditionA1}), 
	\begin{align}
		A(1-\frac{\delta}{2})< A-\frac{\delta^2}{2}G;
	\end{align} then, (\ref{0510A5}) implies 
	\begin{align}
			B{{\left[\overline {A}-\frac{\delta^2}{2}\overline{G}\right]}^{-1}}\ \Bbar \left[A-\frac{\delta^2}{2}G\right]^{-1}<1.  \label{0510A7}
	\end{align}
	From (\ref{0510A5}) to (\ref{0510A7}), we used the following linear algebra lemma:
	\begin{lemma}
		Suppose $K$ and $H$ are two positive definite Hermitian matrices, with $K>H$. Then for a symmetric matrix $B$ we have
		the maximum eigenvalue of $B\overline{K^{-1}}\ \Bbar K^{-1}$ is smaller or equal to the maximum eigenvalue of $B\overline{H^{-1}}\ \Bbar H^{-1}$. 
	\end{lemma}
	The lemma above can be easily proved by simultaneously diagonalize $K$ and $H$; we omit the proof here.
	
	 Then using (\ref{0510A7}) and $A-\frac{\delta^2}{2}G>0$, we have 
		\begin{align}
		\left(A-\frac{\delta^2}{2}G\right)_{\ijbar} z^i \zjbar+\Ree(B_{ij} z^iz^j)=P-\frac{\delta^2}{2}g_{\ijbar} z^i \zjbar
	\end{align} is a strictly convex function on $\EC^n$. Therefore, the modulus of convexity of $P$ is greater than $\frac{\delta^2}{2}$
\end{proof}

\begin{remark}
	The proof of Lemma \ref{lemma_Metric_and_Q_convert_to_Degree} can be visualized in a simplified case of $n=1$. Suppose $a\in\ER^{+}$ and $b\in \ER$. If $a>\delta$ and $\frac{b^2}{a^2}<1-\delta$, then the modulus (or degree) of convexity of 
	$P(z)=a|z|^2+\Ree(bz^2)$ is greater than $\frac{\delta^2}{2}$. In Figure \ref{fig:Visualizing_Convexity_Estimate} the shadowed area is the set
	\begin{align}
	 F=	\{a>\delta\}\cap \left\{\frac{b^2}{a^2}<1-\delta\right\}
	\end{align}
	 \begin{figure}[h]
		\centering  
		\includegraphics[height=3.5cm]	{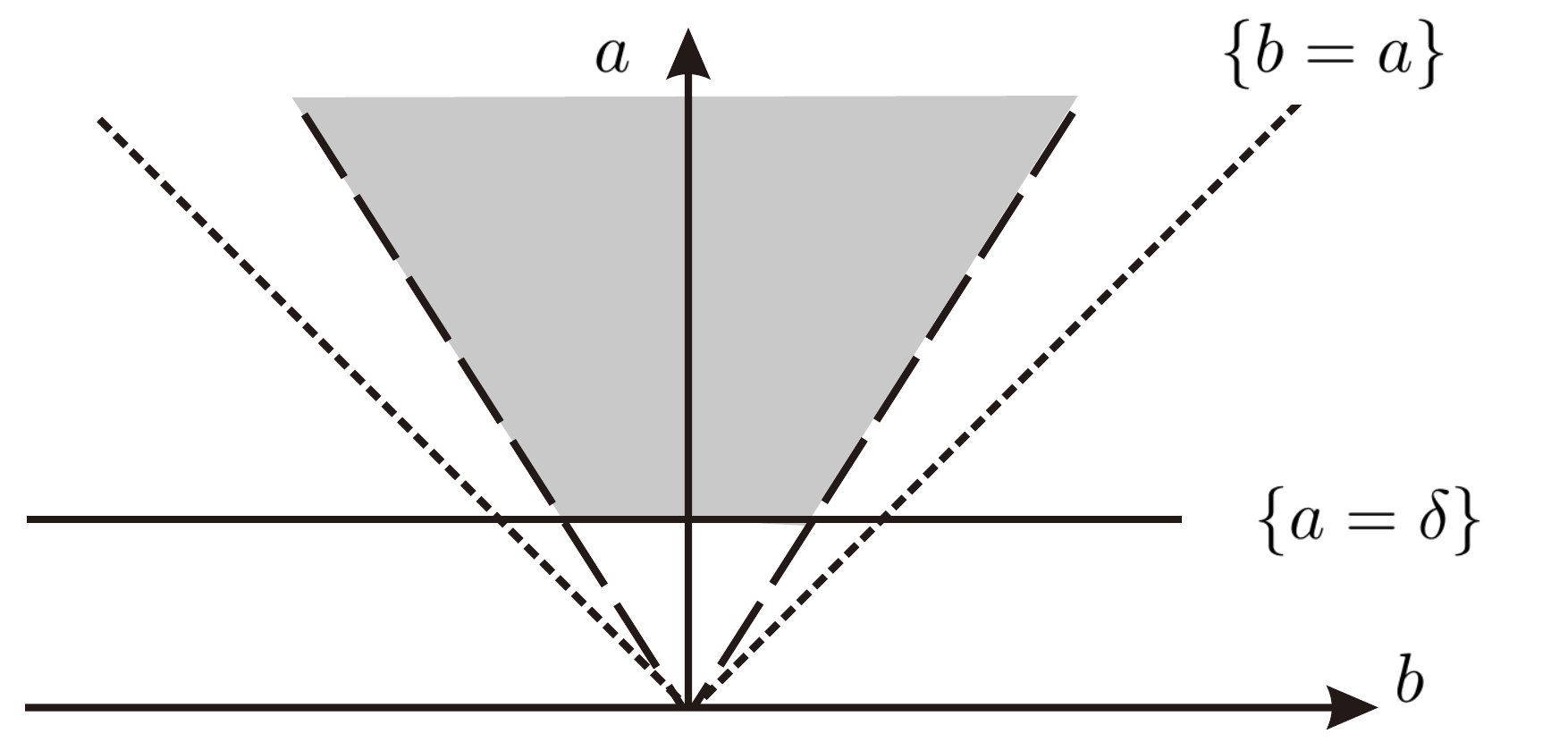}
		\caption{Visualizing the Convexity Estimate}
		\label{fig:Visualizing_Convexity_Estimate}  
	\end{figure}
	We can easily see for a small enough $c>0$, the cone $C=\{|b|<a-c\}$ contains $F$, which says the modulus of convexity of $P$ is greater than $c$.  Through computation, we find we can choose $c=\frac{\delta^2}{2}$.
\end{remark}

Lemma \ref{lemma_Metric_and_Q_convert_to_Degree} proves that for a function 
	\begin{align}
	P(z)=A_{\ijbar} z^i \zjbar+\Ree(B_{ij} z^iz^j)
\end{align} 
the positivity of $A$ and $B\overline{A^{-1}}\ \overline{B}A^{-1}<1$ implies a convexity estimate for $P$; the following lemma proves an estimate in the opposite direction;
\begin{lemma}
	\label{lemma:Robustness_positive_implies_K<1}
	Suppose the modulus (or, equivalently, the degree) of convexity of $P(z)$ is greater than $\delta>0$. Then 
	\begin{align}
		B\overline{A^{-1}}\ \overline{B}A^{-1}<1-\frac{\delta}{C},
	\end{align}where $C$ depends on the $C^2$ norm of $P$.
\end{lemma}
\begin{proof}
	[Proof of Lemma \ref{lemma:Robustness_positive_implies_K<1}]
	Using Autonne-Takagi factorization (Corollary 4.4.4(c) of \cite{HornMatrix}, or Lemma A.3 of \cite{Hu22Nov}), we diagonalize $A, B$ simultaneously so that 
	\begin{align}
		A=I,\ B=\diag(\lambda_1,\ ...\ , \lambda_{n-1}),
	\end{align}
	with $\lambda_\mu\geq0.$ Suppose the Hermitian metric on $\EC^{n-1}$ is $G_\abbar dz^\alpha\otimes \overline{dz^\beta}$. 
The modulus of convexity of $P$ is greater than $\delta$ implies $P-\delta G_{\abbar} z^\alpha\overline{z^\beta}$ is strongly convex. Suppose that
\begin{align}
	AG^{-1}\leq C_2.
\end{align}Then $C_2$ depends on $|D^2P|_{G}$. This implies
\begin{align}
	G\geq \frac{1}{C_2}A,
\end{align} and, with the current coordinate, we have
\begin{align}
	G\geq \frac{1}{C_2}I.
\end{align} So, $P-\frac{\delta}{C_2}|z|^2$ is strongly convex. Therefore, 
\begin{align}
	\sigma_\mu<1-\frac{\delta}{C_2},
\end{align}and it follows that
\begin{align}
	B\overline{A^{-1}}\ \overline{B}A^{-1} <
	\left(
	1-\frac{\delta}{C_2}
	\right)^2.
\end{align}
$P-G_{\abbar} z^\alpha\overline{z^\beta}$ is strongly convex implies $A>\delta G$; therefore,  $C_2\geq \delta$. So we can choose $C=2C_2.$

\end{proof}
We also need the following lemma which shows the modulus or robustness of \qcconvexity\ is weakly preserved when taking limit:
\begin{lemma}
	\label{lemma:Modulus_Estimate_with_Limit}
	 Let $\Omega$ be a domain with smooth boundary, and let $\po$ be a point in $\overline{\Omega}$. Let $\{ {\Phi^\eqe}\}_{\eqe\in[0,1)}$ be a family of smooth functions defined on $\overline{\Omega}$, with
	 \begin{align}
	 	{\Phi^\eqe}\rightarrow\Phi^0, \ D{\Phi^\eqe}\rightarrow D\Phi^0,\  D^2{\Phi^\eqe}\rightarrow D^2\Phi^0,\  \text{ at }\po.
	 \label{625_converge_3}
	 \end{align} Suppose that the robustness of \qcconvexity\ of $\Phi^\eqe$ at $\po$ is greater than $\rho$, for any $\eqe\in (0,1)$. 	 Then for any $0\leq \trho < \rho$, the robustness of \qcconvexity\ of $\Phi^0$ is greater than $\trho$ at $\po$.
\end{lemma}
\begin{proof}
	[Proof of Lemma \ref{lemma:Modulus_Estimate_with_Limit}] 
	We denote the metric on $\EC^n$ by $\MG_\ijbar dz^i\otimes \overline{dz^j}$. In the following, all the computations are performed at the point $\po$.
	
	Because of the \qcconvexity\  assumption for ${\Phi^\eqe}$, for any $q\in \spaceQ_{\trho}$ and $\eqe>0$, ${\Phi^\eqe}-q$ is \sqcconvex\ with the robustness greater than $\rho-\trho$. This implies 
	\begin{align}
		|\nabla({\Phi^\eqe}-q)|_{\MG} >\rho-\trho,
	\end{align}for any $\eqe\in(0,1)$; then, because of (\ref{625_converge_3}),  we have 
		\begin{align}
		|\nabla({\Phi^0}-q)|_{\MG} \geq\rho-\trho>0.
	\end{align} Therefore, we can choose a set of coordinates $(z^\alpha, \tau)$, so that 
	\begin{align}
		\partial_\alpha(\Phi^0-q)=0.
	\end{align} This implies $\partial_\tau(\Phi^0-q)\neq 0$, and when $\eqe$ is small enough, $\partial_\tau(\Phi^\eqe-q)\neq 0$. Let $\Psi^\eqe={\Phi^\eqe}-q$; for each $\Psi^\eqe$, with $\eqe $ small enough, we define
	\begin{align}
		\MA^{\eqe}_{\abbar}=\Psi^\eqe_{\abbar}(\po)-\Psi^\eqe_{\alpha\taubar}(\po)\overline{
		\left(
		\frac{ \Psi^\eqe_\beta(\po)}{ \Psi^\eqe_\tau(\po)}
		\right)}
		                                        -\Psi^\eqe_{\tau\betabar}(\po)
		                                        	\left(
		                                        \frac{ \Psi^\eqe_\alpha(\po)}{ \Psi^\eqe_\tau(\po)}
		                                        \right)
		                                        +\Psi^\eqe_{\tautaubar}	(\po)\left(
		                                        \frac{ \Psi^\eqe_\alpha(\po)}{ \Psi^\eqe_\tau(\po)}
		                                        \right)\overline{
		                                        	\left(
		                                        	\frac{ \Psi^\eqe_\beta(\po)}{ \Psi^\eqe_\tau(\po)}
		                                        	\right)},
		                                        	\\
		                                        		\MB^{\eqe}_{\ab}=\Psi^\eqe_{\ab}(\po)-\Psi^\eqe_{\alpha\tau(\po)}{
		                                        		\left(
		                                        		\frac{ \Psi^\eqe_\beta(\po)}{ \Psi^\eqe_\tau(\po)}
		                                        		\right)}
		                                        	-\Psi^\eqe_{\tau\beta}(\po)
		                                        	\left(
		                                        	\frac{ \Psi^\eqe_\alpha(\po)}{ \Psi^\eqe_\tau(\po)}
		                                        	\right)
		                                        	+\Psi^\eqe_{\tau\tau}(\po)	\left(
		                                        	\frac{ \Psi^\eqe_\alpha(\po)}{ \Psi^\eqe_\tau(\po)}
		                                        	\right){
		                                        		\left(
		                                        		\frac{ \Psi^\eqe_\beta(\po)}{ \Psi^\eqe_\tau(\po)}
		                                        		\right)}.
	\end{align} Because the robustness of \qcconvexity\ of ${\Phi^\eqe}-q$ is greater than $\rho-\trho$, we can use Lemma \ref{lemma:RobustnessPositive_implies_ModulusPositive} and get
	\begin{align}
	(	\MA^{\eqe}_{\abbar} )\geq \frac{\rho-\trho}{C}
	\left[
	\delta_{\ab} +	\left(
	\frac{ \Psi^\eqe_\alpha}{ \Psi^\eqe_\tau}
	\right)\overline{
		\left(
		\frac{ \Psi^\eqe_\beta}{ \Psi^\eqe_\tau}
		\right)}\ 
	\right], \ \ \ \ \  \ \ \ \text{ for any small enough $\eqe>0$,} 
	\end{align} for a constant $C>0$ independent of $\eqe$; using Lemma \ref{lemma:Robustness_positive_implies_K<1}, we also get
	\begin{align}
		 \MB^\eqe\overline{(\MA^{\eqe} )^{-1}}\ \overline{\MB^\eqe}(\MA^{\eqe} )^{-1}\leq 1-\frac{\rho-\trho}{C_K}, \ \ \ \ \  \ \ \ \text{ for any small enough $\eqe>0$,} 
	\end{align}for a constant $C_K$, also independent of $\eqe$. Since $\Psi_\tau^\eqe\neq 0$, for $\eqe$ small enough, condition (\ref{625_converge_3}) implies
	\begin{align}
		\MA^{\eqe} \rightarrow \MA^0,\ \ \ \ \ 	\MB^{\eqe} \rightarrow \MB^0.
	\end{align}Therefore,  we have
		\begin{align}
		(	\MA^{0}_{\abbar} )\geq \frac{\rho-\trho}{C}
		\left[
		\delta_{\ab} 
		\right], \\
		\MB^0\overline{(\MA^{0} )^{-1}}\ \overline{\MB^0}(\MA^{0} )^{-1}\leq 1-\frac{\rho-\trho}{C_K};
	\end{align}then, using Lemma \ref{lemma_Metric_and_Q_convert_to_Degree}, we get $\Psi^0$ is \sqcconvex.
	
	The argument above is valid for any $q\in\spaceQ_\trho$, so the robustness of \qcconvexity\ of $\Phi$ is greater than $\trho$.
\end{proof}

The following Lemmas show the relation between the \qcconvexity\ of a function and the \cconvexity \ of its level sets.
\begin{lemma}
	\label{lemma:Boundary_Convex_implies_Modulus}
	Let $\Omega$ be a \scconvex \ domain with a smooth boundary and with the modulus of \cconvexity\ greater than $\mu$. Suppose that $F\in C^{\infty}(\Omegabar)$ satisfies
	\begin{align}
		F=0\ \ \ \ \  \ \ \ \text{ on }\Gamma,\\
		F_{{\bold n}}>\sigma\ \ \ \ \  \ \ \  \text{ on } \Gamma,
	\end{align}where $\Gamma$ is an open set of $\partial \Omega$, and ${\bold n}$ is the exterior unit normal vector of $\partial \Omega$. Then the modulus of \qcconvexity\ of $F$ is greater than $		\min\{1, \sqrt{2}\mu\} \sigma$ 
on $\Gamma$. 
\end{lemma}
\begin{proof}
	[Proof of Lemma \ref{lemma:Boundary_Convex_implies_Modulus}] 
	
	The proof is straightforward, we only need to use the implicit function theorem.

	Let  $\po$ be a point on $\Gamma$.  We choose coordinates $(z^\alpha, \tau)$, with $\tau=t+\iu s,$ so that $\po=0$ and $\{t=0\}$  is the tangential plane of $\partial \Omega$ at $\po$. Suppose that around $\po$ 
	\begin{align}
		\Omega=
		\left\{
		t>\rho(z^\alpha, s)
		\right\}.
	\end{align}
	Then we have $\rho(0)=0$, $\nabla \rho=0$ and $\partial_t=-\sqrt{2}{{\bold n}}$, at $\po$. $F=0$ on $\Gamma$ implies
	\begin{align}
		F(z^\alpha, \rho(z^\alpha, s)+\iu s)=0.
	\end{align} The implicit function theorem gives
	\begin{align}
		\begin{split}
			&\text{the second order Taylor expansion of } F|_{\{\tau=0\}}  \text{ at $\po$}\\
		=&\sqrt{2}|F_{\bold n}|\cdot\text{the second order Taylor expansion of } \rho|_{\{s=0\}}  \text{ at $0$}.
		\end{split}
	\end{align}
	Therefore, the modulus of \cconvexity\ of $\Omega$ is greater than $\mu$ implies the modulus of convexity of the second order Taylor expansion of $ F|_{\{\tau=0\}}  \text{ at $\po$}$ is greater than $\sqrt{2}|F_{\bold n}| \cdot \mu.$ Then the result follows. 
\end{proof}
\begin{lemma}
	\label{lemma:Modulus_implies_levelsets_Convex}
	Let $\Phi$ be a smooth function in $\MR=\Omega_1\backslash\overline{\Omega_0}$, where $\Omega_0$ and $\Omega_1$ are \scconvex\ domains with smooth boundaries and $\overline{\Omega_0}\subset\Omega_1$. Suppose that $\Phi$ satisfies
	\begin{align}
		\Phi=0\ \ \ \ \  \ \ \ \text{ on }\partial \Omega_0,\\
			\Phi=1\ \ \ \ \  \ \ \ \text{ on }\partial \Omega_1,
	\end{align}
	and the modulus of \qcconvexity\ of $\Phi$ is greater than $\mu>0$. Then
	\begin{align}
		\Omega_t=
		\left\{
		\Phi<t
		\right\} \cup\overline{\Omega_0}
	\end{align}are all \scconvex\ with the modulus of \cconvexity\ greater than 
	\begin{align}
		\frac{\mu}{\max_{\overline{\MR}}|\nabla\Phi|}.
	\end{align}
\end{lemma}
\begin{proof}
	[Proof of Lemma \ref{lemma:Modulus_implies_levelsets_Convex}] 
	First,  the condition that $\Phi$ is \sqcconvex\ implies that $\nabla \Phi\neq 0$ anywhere in $\MR$; therefore, 
	\begin{align}
		0< \Phi <1 \ \ \ \ \  \ \ \  \text{ in }\MR, 
	\end{align} and all the level sets of $\Phi$ are smooth. 
	
	Let $\po$ be a point in $\MR$. We choose coordinate $(z^\alpha, \tau)$, with $\tau=t+\iu s$, so that the metric is $dz^i\otimes \overline{dz^i}$ and that at $\po$
	\begin{align}
		\Phi_\alpha=0, \ \ \ \Phi_s=0,\ \ \ \Phi_t=-|\nabla\Phi|.
	\end{align} We also assume that $\po=0.$
	Then, around $\po$, the level set $\{\Phi=\Phi(\po)\}$  is the graph of a function $\rho$:
	\begin{align}
		\{\Phi=\Phi(\po)\}=
		\left\{
		(z^\alpha, \rho(z^\alpha, s)+\iu s)
		\right\},
	\end{align}with $\rho(0)=0$. 
	Similar to the proof of Lemma \ref{lemma:Boundary_Convex_implies_Modulus}, we have
		\begin{align}
		\begin{split}
			&\text{the second order Taylor expansion of } \Phi|_{\{\tau=0\}}  \text{ at $\po$}\\
			=&\Phi(\po)+|\nabla \Phi(\po)|\cdot\text{the second order Taylor expansion of } \rho|_{\{s=0\}}  \text{ at $0$}.
		\end{split}
	\end{align}
	Therefore, the modulus of \qcconvexity \ of $\Phi$ is  than $\mu$ implies the modulus of convexity of the second order Taylor expansion of $\rho|_{\{s=0\}} $ at $0$ is greater than $\frac{\mu}{|\nabla \Phi(\po)|}$. Then we let $\po$ be an arbitrary point in $\overline{\MR}$ and get the result.
\end{proof}
\section{Construction of a Subsolution and Gradient Estimates}
\label{sec:B_boundary_Barrier}
In this appendix, we will construct a subsolution to Problem \ref{prob:Perturbed_Problem_HessianQuotient} (Lemma \ref{lemma:sub_solution_Construction}). The existence of a subsolution is the key to the solvability of Problem \ref{prob:Perturbed_Problem_HessianQuotient} and Problem \ref{prob:Main_Problem_HCMA_ring}. In addition, we will show for the solution ${\Phi^\eqe}$ to Problem \ref{prob:Perturbed_Problem_HessianQuotient} $|\nabla\Phi^\eqe|_\MG$ has a positive lower bound on $\partial\MR$ (Lemma \ref{lemma:GradientNorm_Lower_Boundary}). 

\begin{lemma}
	\label{lemma:sub_solution_Construction}
	Suppose $\MR=\Omega_1\backslash\overline{\Omega_0}$, with $\Omega_0$ and $\Omega_1$ satisfying the requirements of Problem \ref{prob:Perturbed_Problem_HessianQuotient}.  Then for $\scs>0$ small enough, we can find $\Psi\in C^{\infty}(\overline\MR)$ satisfying
	\begin{align}
		\iu\ddbar\Psi
		\geq \scs\bbform,\ \ \ \ \  \ \ \ &\text{ in }\MR,
		\label{627B1}
		\\
		\Psi=0, \ \ \ \ \  \ \ \ &\text{ on }\partial\Omega_0,\\
			\Psi=1, \ \ \ \ \  \ \ \ &\text{ on }\partial\Omega_1,
	\end{align}
	and 
	\begin{align}
		\Psi_{\bold n}\geq \scs\ \ \ \ \ \  \ \ \ \text{ on }\partial\Omega_0,
		\label{628B4}
	\end{align}here ${\bold n}$ is the exterior unit normal vector on $\partial \Omega_0$.
\end{lemma}
\begin{proof}
	[Proof of Lemma \ref{lemma:sub_solution_Construction}]
	Let $V_{\Omega_0}$ be the pluricomplex Green's function of $\Omega_0$ with pole at infinity. Then
	\begin{align}
		V_{\Omega_0}(z)=\sup \left\{
		u(z)\ | \ u\in\spaceL, u\leq 0 \text{ in }\Omega_0
		\right\},
	\end{align} where $\spaceL$ is the family of all plurisubharmonic functions on $\EC^n$ with 
	\begin{align}
		u(z)-\log|z|\leq O(1) \text{ as } |z|\rightarrow \infty.
	\end{align}Using the method of \cite{Bedford_Taylor_Dirichlet_Problem}, we know $V_{\Omega_0}$ is continuous in $\EC^n$. Actually, using the theory of Lempert, we can even show $V_{\Omega_0}$ is smooth in $\Omega_0^c$ (see Theorem 5.1 of \cite{Lempert85_Duke_Symmetries} and the first remark following the proof of Lemma 5.3); however, we only need to use the continuity.
	We also need to show
	\begin{align}
		V_{\Omega_0}>0\text{ \ \ \ \ \  \ \ \ in } (\overline{\Omega_0})^c.
		\label{626star}
	\end{align}This depends on the fact that $\Omega_0$ is polynomially convex, which is proved in the following lemma:
	\begin{lemma}\label{lemma:C_Convexity_Implies_PolyConvex}
		Suppose $E$ is a bounded \cconvex\ domain with a smooth boundary. Then $E$ is polynomially convex.
	\end{lemma}
	\begin{proof}
		[Proof of Lemma \ref{lemma:C_Convexity_Implies_PolyConvex}]
		We only need to show that any complex hyperplane in $E^c$ can be continuously moved to infinity; then, we can apply Prop 2.1.9 of \cite{Book_Cconvexity} (also see Remark 2.1.10).
		
		Assume $0\in E$, then let
		\begin{align}
			r_0=\sup
			\left\{
			r\ |\ B_r^c\bigcap E\neq \emptyset
			\right\},
		\end{align}where $B_r=\{|z|\leq r\}$. Since $E$ is bounded, $r_0\leq  \text{diameter}(E)$; since $\partial E$ is smooth, $\partial B_{r_0}$ should contact with $\partial E$. Let $\po\in B_{r_0}\cap E^c$.
		
		With the preparation above, we can take three steps to move a complex hyperplane $\plane\subset E^c$ to  infinity.
		
		{\bf Step 1.} We continuously move $\plane$ towards $0$ until it contacts with $\partial E$ at a point $\na$. Here $\na$ may not be unique.
		
		{\bf Step 2.} Let $\gamma=
		\left\{
		\gamma(t)\ | \ t\in [0,1]
		\right\}$ be a smooth curve on $\partial E$ connecting $\po$ with $\na$. Then 
		\begin{align}
			\left\{
			T^\EC_{\partial E, \gamma(t)}
			\right\}_{t\in[0,1]}
		\end{align} is a family of complex hyperplanes connecting 
		$
		T^\EC_{\partial E, \na}$ 
		and  
			$T^\EC_{\partial E, \po}$. Here $T^\EC_{\partial E, {\bf z}}$ denotes the complex tangent plane of $\partial E$ at a point ${\bf z}\in\partial E$.
		
		{\bf Step 3.} The complex tangent plane of $\partial E$ at $\po$ can be moved to infinity since it's also the complex tangent plane of $\partial B_{r_0}$ (we have $B_{r_0}\supset \overline{E}$).
	\end{proof}
	Denote 
	\begin{align}
		\Phi_{\Omega_0}(z)=\sup \left\{
		|p(z)|^{\frac{1}{\deg(p)}} \ \big|\ p\in\MP_{\Omega_0}
		\right\}
	\end{align} where $\MP_{\Omega_0}$ is the space of complex polynomials $p$ on $\EC^n$ with $\deg(p)\geq 1$ and $|p|\leq 1$ on $\Omega_0$. This function is called Siciak's extremal function.  Since $\Omega_0$ is polynomially convex, for any $z\in \left(
	\overline{\Omega_0}
	\right)^c$ we can find $p\in \MP_{\Omega_0}$ so that $|p(z)|>1$; then, it follows that $\Phi_{\Omega_0}(z)>1$.  
	
	According to the theory of Zajarjuta-Siciak (\cite{Siciak_Extremal_Functions}, Theorem 5.1.7 of \cite{Book_Klimek}), 
	\begin{align}
		V_{\Omega_0}=\log \Phi_{\Omega_0};
	\end{align}
	we have $V_{\Omega_0}(z)>0$, for $z\notin \overline{\Omega_0}$.
	
	In the following, applying the method of \cite{GuanAnnuals}, we use $V_{\Omega_0}$ and the distance functions to $\partial \Omega_0$ and $\partial\Omega_1$ to construct a strictly plurisubharmonic function in $\MR$.
	
	Let
	\begin{align}
		F(z)&=c\left(
		V_{\Omega_0}(z)-c+c^2|z|^2
		\right),\\
		f_1(z)&=\frac{1}{c}(e^{-\frac{1}{c}\dist(z, \partial \Omega_1)}-1)+1,\\
		f_0(z) &=e^{-\frac{1}{c^2}}\left(
		e^{\frac{1}{c}\dist(z, \partial\Omega_0)}-1
		\right),
	\end{align}
	where $c$ is a small constant to be determined. Since $V_{\Omega_0}$ is plurisubharmonic, $F$ is strongly plurisubharmonic. Let $\Omega_1^r$ and $\Omega_0^r$ be small neighborhoods of $\partial\Omega_0$ and $\partial\Omega_1$:
	\begin{align}
		\Omega_0^r=
		\left\{
		z\in\MR\ |\ \dist(z, \Omega_0)<r, \ z\in\MR
		\right\},\\
		\Omega_1^r=
		\left\{
		z\in\MR\ |\ \dist(z, \Omega_1^c)<r, \ z\in\MR
		\right\}.
	\end{align}Using Theorem 3.18 of \cite{CLaurentSCV}, we know when $c$ and $r$ are small enough $f_1$ and $f_0$ are  strongly plurisubharmonic in $\Omega_1^r$ and $\Omega_0^r$, respectively.  
	
	As illustrated by Figure \ref{fig:subsolution_construction}, we choose $c$ small enough so that
	\begin{align}
		F<1=f_1\ \ \ \ \  \ \ \ &\text{ on }\partial\Omega_1,
		\label{0701B17}\\
		F>0>f_1\ \ \ \ \  \ \ \ &\text{ on }\partial\Omega_1^r\backslash \partial\Omega_1,
		\label{0701B18}\\
		F>f_0\ \ \ \ \  \ \ \ &\text{ on }\partial\Omega_0^r\backslash \partial\Omega_0,
		\label{0701B19}\\
		F<0=f_0\ \ \ \ \  \ \ \ &\text{ on }\partial\Omega_0.
		\label{0701B20}
	\end{align}
	\end{proof}
	\begin{figure}[h]
	\centering  
	\includegraphics[height=3.5cm]	{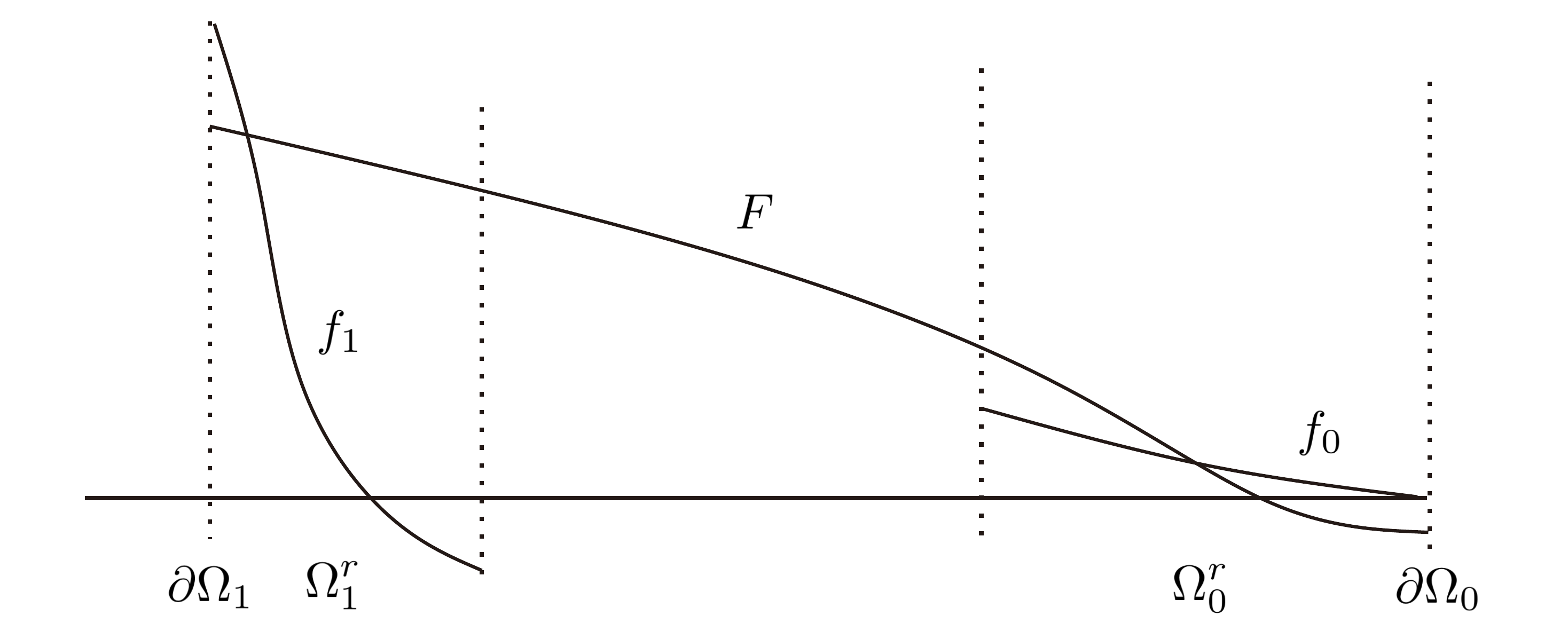}
	\caption{Construction of Subsolution}
	\label{fig:subsolution_construction}  
	\end{figure}
	(\ref{0701B17}) can be satisfied since $V_{\Omega_0}$ is bounded on a compact set and $f_1=1$ on $\partial\Omega_1$; (\ref{0701B18}) can be satisfied since $V_{\Omega_0}>0$ on $\partial\Omega_1^r\backslash \partial\Omega_1$ and $e^{-\frac{1}{c} \dist(z,\partial\Omega_1)}-1$ becomes negative on $\partial \Omega_1^r\cap \MR$ when $c$ is small enough; (\ref{0701B19}) can be satisfied because $V_{\Omega_0}>0$ on $\partial\Omega_0^r\backslash \partial\Omega_0$ and $f_0\leq e^{-\frac{1}{c}}$ providing $c$ small enough; (\ref{0701B20}) can be satisfied since $V_{\Omega_0}$ and $f_0$ both equal to $0$ on $\partial\Omega_0$.
	
	As we mentioned before, according to the theory of Lempert, $V_{\Omega_0}$ is smooth in $\Omega_0^c$ and, therefore, $F$ is smooth in $\Omega_0^c$; however, if we don't want to depend on the theory of Lempert and pretend to be unaware of the smoothness of $F$, then we need to mollify $F$. Note that $V_{\Omega_0}$ is a plurisubharmonic function on $\EC^n$ with $V_{\Omega_0}=0 $ in $\Omega_0$, so $F$ is strictly plurisubharmonic on $\EC^n$ in weak sense. Let $\eta=\eta(|y|)$ be a cut-off function on $\EC^n$, with support in $B_1$ and $\int_{\EC^n} \eta=1$; let
	\begin{align}
		F_\rho(x)=\int_{\EC^n} F(x-y) \eta(\frac{|y|}{\rho}) \frac{1}{\rho^{2n}}dy.
	\end{align}
	Then $F_\rho\rightarrow F$ in $C^0$ norm as $\rho\rightarrow 0$. When $\rho$ is small enough, condition (\ref{0701B17})-(\ref{0701B20}) are still valid after replacing $F$ by $F_\rho$.
	
	Let $h$ be the function constructed in Lemma 3.1 of \cite{GuanAnnuals}, with $\delta$ small enough; then, we let
	\begin{align}
	\Psi=
	\left\{
	\begin{array}{ccc}
		\frac{F_\rho+f_1}{2}+h(\frac{F_\rho-f_1}{2})& \text{ in }\Omega_1^r  , \\
		\frac{F_\rho+f_0}{2}+h(\frac{F_\rho-f_0}{2})& \text{ in }\Omega_0^r,   \\
		F_\rho &    \text{ in }\MR\backslash{(\Omega_1^r\cup \Omega_0^r)}.
	\end{array}
	\right .
	\end{align} According to the computation of Lemma 3.2 of \cite{GuanAnnuals} $\Psi$ is smooth and strictly plurisubharmonic. Therefore for $\scs$ small enough, (\ref{627B1}) is valid.
	We also note that $\Psi=f_0$ in a small neighborhood of $\partial\Omega_0$, proving $\delta$ is small enough so $\partial_n\Psi>0$ on $\partial\Omega_0$; then we can choose $\sigma$ small so that $(\ref{628B4})$ is satisfied.
	
Because of (\ref{627B1}), we can choose $\eqe_0$ small enough so that
\begin{align}
	\scs^n\bbform^n\geq \eqe_0\left(
	\iu\ddbar\Psi
	\right)^{n-1}\wedge\bbform
\end{align}
Then $\Psi$ is a subsolution to Problem \ref{prob:Perturbed_Problem_HessianQuotient} when $\eqe\leq \eqe_0$.

Suppose $\Phi^\eqe$ is a solution to Problem \ref{prob:Perturbed_Problem_HessianQuotient}. Then $|\nabla\Phi^\eqe|_\MG$ has a positive lower bound on $\partial \Omega_0$ since $\Phi^\eqe\geq \Psi$ in $\MR$ and $\Phi^\eqe=\Psi$ on $\partial \Omega_0$. 
The estimate of $|\nabla\Phi^\eqe|_\MG$ on $\partial\Omega_1$ is very simple. Let $H$ be the harmonic function in $\MR$ with 
$H=0$ on $\partial \Omega_0$ and $H=1$ on $\partial \Omega_1$. Then the maximum principle implies
\begin{align}
	H\geq \Phi^\eqe \ \ \ \ \  \ \ \ \text{ in }{\MR}
\end{align} since $\Phi^\eqe$ is subharmonic in $\MR$; the
 Hopf maximum principle implies $H_{{\bold n}}>0$ on $\partial\Omega_1$, where ${\bold n}$ is the exterior unit normal vector on $\partial\Omega_1$. Therefore, $\Phi_{\bold n}^\eqe>0$ on $\partial \Omega_1$.

Summing up, we have the following positive lower bound estimate for $|\nabla\Phi^\eqe|_\MG$:
\begin{lemma}
	\label{lemma:GradientNorm_Lower_Boundary}
	Suppose $\Omega_0$ and $\Omega_1$ are domains satisfying conditions of Problem \ref{prob:Perturbed_Problem_HessianQuotient} . Then there is a positive constant $C_{\partial\MR}^g$ so that when $\eqe$ is small enough, the solution $\Phi^\eqe$ to Problem \ref{prob:Perturbed_Problem_HessianQuotient} satisfies
	\begin{align}
		|\nabla{\Phi^\eqe}|_\MG\geq C_{\partial\MR}^g\ \ \ \ \  \ \ \ \text{ on }\partial\MR.
	\end{align}
\end{lemma}
Now, for a fixed \cconvex\ ring, we have already constructed the subsolution and derived the desired gradient estimate; however, when considering a family of domains $\{\MR^t\}$, we need to show the constants $ C_{\partial\MR^t}^g$ can be chosen so that they have a uniform positive lower bound and that the $C^3$ norms of $\Psi^t$, the subsolutions in $\MR^t$, have a uniform upper bound. So, we need the following lemma:
\begin{lemma}	\label{lemma:GradientNorm_Lower_Boundary_DomainFamily}
	Let $\{\Omega_0^t\}_{t=0}^1$ and $\{\Omega_1^t\}_{t=0}^1$  be two families of \scconvex\ domains with smooth boundaries and satisfying $\overline{\Omega_0^t}\subset\Omega_1^t$ and condition 3 of Lemma \ref{lem:Deform_a_strongly_C-convex_domain_to_Ball}. Then we can construct a family of smooth subsolutions $F^t$ to Problem \ref{prob:Perturbed_Problem_HessianQuotient} with $\Omega_0=\Omega_0^t$ and $\Omega_1=\Omega_1^t$; more precisely, the following conditions are satisfied:
	\begin{align}
		\iu \ddbar F^t\geq \sigma \bbform \ \ \ \ \  \ \ \ &\text{ in } \Omega_1^t\backslash\overline{\Omega_0^t},\\
		\partial_{{\bold n}} F^t\geq \sigma \ \ \ \ \  \ \ \ &\text{ on }\partial \Omega_0^t,\\
				|F^t|_{C^3}\leq \frac{1}{\sigma},
	\end{align}where $\sigma$ is a positive constant. Let  ${\Phi_\eqe^t}$ be the solution to Problem \ref{prob:Perturbed_Problem_HessianQuotient} with $\Omega_0=\Omega_0^t$ and $\Omega_1=\Omega_1^t$; we have
	\begin{align}
		|\nabla {\Phi^t_\eqe}|\geq C^g\ \ \ \ \  \ \ \ \text{ on }\partial \Omega_0^t\cup\partial\Omega^t_1,
	\end{align}when $\eqe$ is small enough, where $C^g$ is a positive constant.
\end{lemma}	  
	  
	  \begin{proof}
	  	Using Lemma \ref{lemma:sub_solution_Construction}, for each $s\in [0,1]$ we can find $\Psi^s$ and ${\ese_s}$ so that 
	  	\begin{align}
	  		\iu\ddbar \Psi^s \geq {\ese_s}\bbform\ \ \ \ \  \ \ \ \text{ in }\Omega_1^s\backslash\overline{\Omega_0^s},	\label{0701B29}
\end{align}
and
\begin{align}
	0<\Psi^s<1, \ \ \ \ \  \ \ \ \text{ in }\Omega_1^s\backslash\overline{\Omega_0^s}.
\end{align}
This implies $\Psi^s$  are subsolutions to the following Problem \ref{prob:MA_Family_subsolution} when $t=s$; therefore, using the method of \cite{GuanAnnuals},	the following problem is solvable when $t=s$:
	  	\begin{problem}
	  		\label{prob:MA_Family_subsolution} 
	  		Find $M^t$ satisfying
	  		\begin{align}
	  			\left[\iu\ddbar M^t\right]^n=\frac{({\ese_s})^n}{2} \bbform^n\ \ \ \ \  \ \ \ &\text{ in }
	  			\Omega_1^t\backslash\overline{\Omega_0^t},\\
	  				M^t=1\ \ \ \ \  \ \ \ &\text{ on } \partial \Omega_1^t,\\
	  				M^t=0\ \ \ \ \  \ \ \ &\text{ on } \partial \Omega_0^t.
	  		\end{align}
	  	\end{problem}
	  	Then we can apply the implicit function theorem and solve the Dirichlet problem above, for $t,$ with $|t-s|<\delta_s$, where $\delta_s$ is a small number. We note that $\delta_s$ may depend on ${\ese_s}$. We can choose $\delta_s$ and ${\ese_s}$ small enough so that
	  	\begin{align}
	  			\partial_{\bold n} M^t\geq \frac{\ese_s}{2} \ \ \ \ \  \ \ \ \text{ on }\partial\Omega_0^s,
	  	\end{align}and $|M^t|_{C^3}\leq \frac{2}{ {\ese_s}}$.

Then we use the compactness of $[0,1]$: We have 
\begin{align}
	\bigcup_{s\in[0,1]} \left\{
	|t-s|<\delta_s
	\right\}\supset [0,1],
\end{align}so we can find finite $s_k, $ for $k\in\{1,\ ...\ ,N\}$, such that 
\begin{align}
	\bigcup_{k=1}^n \left\{
	|t-s_k|<\delta_{s_k}
	\right\}\supset [0,1].
\end{align}Therefore, the solution $M^t$ to Problem \ref{prob:MA_Family_subsolution}, with $\ese_s={\ese_{s_k}}$ and $|t-s_k|<\delta_{s_k}$, are subsolutions  of Problem \ref{prob:Perturbed_Problem_HessianQuotient} satisfying the conditions of Lemma \ref{lemma:GradientNorm_Lower_Boundary_DomainFamily}, with
\begin{align}
	\sigma=\min
	\left\{
	\frac{{\ese_{s_k}}}{2}
	\right\}_{k=1}^N.
\end{align} The estimate for $|\nabla{\Phi^t_\eqe}|$ follows easily.
	  \end{proof}
\section{$C^{1,1}$ Estimate Independent of $\eqe$}\label{sec:Appendix_C11Estimate}
In this section we provide a  $C^{1,1}$ estimate for solutions of Problem \ref{prob:Perturbed_Problem_HessianQuotient} which is independent of $\eqe$, providing $\eqe$ small enough. All the ideas are taken from \cite{Guan98}. 
\begin{proposition}
	\label{prop:Appendix_C2independent_epsilon}
	Suppose that $\Phi$ is a smooth solution to Problem \ref{prob:Perturbed_Problem_HessianQuotient} with $\eqe$ small enough. Then $|\nabla\Phi|_\MG$ and $|D^2\Phi|_\MG$ have upper bounds which depend on the positive lower bound of $|\Phi_n|$ on $\partial\MR$ and the plurisubharmonicity and  the $C^3$ norm of $\Psi$ where $\Psi$ is the subsolution constructed in Section \ref{sec:B_boundary_Barrier}; in particular, the estimates are independent of $\eqe$, providing $\eqe$ small enough.
\end{proposition}

\begin{proof}
	In the following, we only need to do unitary coordinate transformations, so we let $\MG_\ijbar=\delta_{ij}$.
	Let $U$ be the solution to the Laplace equation 
	\begin{align}
		U_{i\ibar}=0,
	\end{align}
	with the same boundary values as those of $\Phi$.
	Then, obviously, we have 
	\begin{align}
		\Phi\leq U\ \ \ \ \ \ \text{ in }\MR \label{eq:E_U_upperbd}
	\end{align}since $\Phi$ is plurisubharmonic and, therefore, subharmonic. 
	Let $\Psi$ be the strictly plurisubharmonic function constructed in Section \ref{sec:B_boundary_Barrier}; suppose that 
	\begin{align}
		\sqrt{-1}\partial\overline{\partial}\Psi\geq \eqe_0\bbform.  \label{427E3}
	\end{align}
	Then $\Psi\leq \Phi$, providing $\epsilon$ small enough. So, we can control the gradient of $\Phi$ on the boundary since $\Psi=\Phi=U$ on $\partial \MR$ and $\Psi\leq \Phi\leq U$ in $\Omega$.
	
	The interior estimate for gradients can be reduced to the boundary estimate. Let $\MG=(\MG_{\ijbar})$ and $\MH=(\Phi_{\ijbar})$.  Then the equation (\ref{eq:SigmaQuotient_Problem_Perturbation}) becomes
	\begin{align}
		\MH^{-1}=\frac{1}{\eqe}.
	\end{align} Let $X$ be a constant vector field on $\EC^n$. We apply $\partial_X$ to the equation above and get
	\begin{align}
		\MH^{-1} \MH_X \MH^{-1}=0;  \label{423E5}
	\end{align}
	using  operator $L$ (introduced in Section \ref{sec:AuxiliaryFunctions}), this is equivalent to
	\begin{align}
		L^{\ijbar}(\Phi_X)_{\ijbar}=0.  \label{427E6}
	\end{align}Since $L$ is an elliptic operator, an estimate on the $C^1$ norm of $\Phi$ in $\overline{\MR}$ follows from the usual maximum principle.
	
	We apply $\partial_X$ to (\ref{423E5}) and get
	\begin{align}
		\MH^{-1} \MH_{XX} \MH^{-1}
		=2\MH^{-1} \MH_{X}\MH^{-1} \MH_{X} \MH^{-1},\label{423E7}
	\end{align} 
	which says
	\begin{align}
		L^{\ijbar}(\Phi_{XX})_{\ijbar}=2(\partial_X\Phi)_{i\qbar} \Phi^{p\qbar} (\partial_X\Phi)_{p\jbar} L^{\ijbar}\geq0.
	\end{align}
	Since $\Phi$ is a plurisubharmonic function, we have
	\begin{align}
		\Phi_{XX}+\Phi_{JXJX}\geq0.
	\end{align}So the estimate of second order derivatives of $\Phi$ can  be reduced to the boundary.
	
	Suppose that $0$ is a boundary point and that around $0$ the boundary of $\MR$ is locally the graph of a function: for a small $r>0$
	\begin{align}
		\partial\MR\cap Q_r=\left\{(z_\alpha, z_n= t+\iu s)| t=\rho(z_\alpha, s)\right\}.
	\end{align}
	Here $Q_r$ is a cube 
	\begin{align}
		Q_r=\left\{|z_\alpha|^2+|s|^2\leq r^2, |t|\leq r\right\},
	\end{align}and $\rho(0)=0, \nabla \rho(0)=0.$ We recall that Greek indices run from $1$ to $n-1$.   In the following, we will estimate second order derivatives of $\Phi$ at $0$.
	
	Let $Y$ be a constant vector field in $\EC^n$ parallel to $\{t=0\}$; let
	\begin{align}
		\MY(z_\alpha, z_n)=Y+\partial_Y\rho(z_\alpha, s)\partial_t.
	\end{align}
	Then the boundary condition $\Phi|_{\partial\MR}=\text{constant}$ implies
	\begin{align}
		\partial_\MY\Phi=\partial_{\MY\MY}\Phi=0 \ \ \ \ \ \ \text{ on }\partial\MR.
	\end{align}
	Therefore, 
	\begin{align}
		\Phi_{YY}+\rho_{YY}\Phi_t=0\ \ \ \ \ \ \text{ at }0;
		\label{427E13}
	\end{align}
	thus, $\Phi_{\alpha\betabar},\ \Phi_{\alpha\beta},\ \Phi_{\alpha s},\ \Phi_{ss}$ can all be controlled by $|\Phi|_{C^1}$ and $|\rho|_{C^2}$, and it remains to estimate $\Phi_{Yt}$ and $\Phi_{tt}$.
	
	The equation (\ref{eq:SigmaQuotient_Problem_Perturbation}) implies the following relation
	\begin{align}
		\Phi_{\nnbar}\cdot \det(\Phi_{\alpha\betabar})=\eqe 	\Phi_{\nnbar}\cdot \sigma_{n-2} (\Phi_{\alpha\betabar})+\text{Terms Consisting of }\Phi_{\alpha\betabar }\text{ and }\Phi_{\alpha \nbar}. 
	\end{align}
	Because of (\ref{427E13}), $\det(\Phi_{\alpha\betabar})$ has a lower bound, depending on the \cconvexity\ of the boundary and the lower bound of $|\Phi_t|$, and $\sigma_{n-2}(\Phi_{\alpha\betabar})$ has an upper bound, depending on the $C^2$ regularity of the boundary and the gradient estimate of $\Phi$. Therefore, when $\eqe$ is small enough (\ref{427E13}) reduces the estimate for $\Phi_{\nnbar}$ to the estimates for $\Phi_{\alpha\nbar}$ and $\Phi_{\alpha\betabar}$. Since $4\Phi_{\nnbar}=\Phi_{tt}+\Phi_{ss}$  and $\Phi_{ss}$ has been estimated, it only remains to estimate $\Phi_{tY}$.
	
	The estimate of $\Phi_{tY}$ depends on the construction of a barrier function in $Q_r\cap \MG$. Let 
	\begin{align}
		h=-\Phi_s^2+A(s^2+\sum_\alpha |z_\alpha|^2)-B\Psi.\label{427E11}
	\end{align}
	We will show, for properly chosen $A$ and $B$, 
	\begin{align}
		L^{\ijbar}h_{\ijbar}\leq L^{\ijbar}(\Phi_\MY)_{\ijbar}  \ \ \ \ \ \ \text{ in }Q_r, \label{427E16}
	\end{align}
	and
	\begin{align}
		h\geq \Phi_\MY  \ \ \ \ \ \ \text{ on }\partial Q_r. \label{427E17}
	\end{align}
	
	First, we compute $L^{\ijbar} (\Phi_{\MY})_{\ijbar}$ and bound it from below. Directly using the definition of $\MY$, we have
	\begin{align}
		L^{\ijbar}(\Phi_{\MY})_{\ijbar}&=L^{\ijbar} (\Phi_Y+\rho_Y \Phi_t)_{\ijbar}\\
		&=L^{\ijbar}\Phi_{Y\ijbar} +L^{\ijbar}\Phi_{t\ijbar}\rho_Y+L^{\ijbar} \rho_{i Y}\Phi_{t\jbar}+L^{\ijbar} \rho_{\jbar Y} \Phi_{ti}+L^{\ijbar} \rho_{Y\ijbar} \Phi_t.
	\end{align} In above, using (\ref{427E6}), we know $L^{\ijbar}\Phi_{Y\ijbar}=L^{\ijbar}\Phi_{t\ijbar}=0$. Since $L$ is an operator with bounded coefficients, we have
	\begin{align}
		L^{\ijbar}\rho_{Y\ijbar}\Phi_{t}  \leq C(|\rho|_{C^3}, |\Phi|_{C^1}).
	\end{align} 
	$\Phi_{t\jbar}$ is relatively harder to control; here, we couple it with $\Phi_{s\jbar}$ and get
	\begin{align}
		L^{\ijbar}\Phi_{t\jbar}\rho_{iY}
		=2L^{\ijbar} \Phi_{n \jbar} \rho_{iY}
		+
		L^{\ijbar}\iu\Phi_{s\jbar}\rho_{iY}.
	\end{align} In above,
	$L^{\ijbar}\Phi_{n\jbar}$ can be well controlled.
	Recall that $L^{\ijbar}=\eqe^2\Phi^{i\kbar} \MG_{l\kbar} \Phi^{l\jbar}$; so, 
	\begin{align}
		L^{\ijbar}\Phi_{n\jbar}=\eqe^2 \Phi^{i\qbar} \MG_{n\qbar}. 
	\end{align} In this section $\MG_{\ijbar}=\delta_{ij}$, so $ \Phi^{i\qbar} \MG_{k\qbar}= \Phi^{i \kbar}$ is a non-negative matrix, with $\Phi^{i\ibar}=\frac{1}{\eqe}$. Therefore, 
	\begin{align}
		|L^{\ijbar} \Phi_{n\jbar}\rho_{iY}|\leq 
		\eqe |\rho|_{C^2}.
	\end{align}
	For $\Phi_{s\jbar}\rho_{iY}L^{\ijbar}$, we use Cauchy inequality and get
	\begin{align}
		L^{\ijbar} \Phi_{s\jbar}\rho_{iY}\geq 
		-\frac{1}{2}L^{\ijbar}\Phi_{s\jbar}\Phi_{si}
		-\frac{1}{2}L^{\ijbar}\rho_{Y\jbar}\rho_{Yi}.
	\end{align}
	So
	\begin{align}
		L^{\ijbar} \Phi_{\MY\ijbar}\geq
		-C(|\rho|_{C^3}, |\Phi|_{C^1})-L^{\ijbar} (\Phi_s^2)_{\ijbar}. \label{427E25}
	\end{align}
	The computation for $h$ is straightforward:
	\begin{align}
		L^{\ijbar} h_{\ijbar}\leq -L^{\ijbar} (\Phi_s^2)_{\ijbar}
		+A C(\MG)-B L^{\ijbar}\Psi_{\ijbar}. \label{427E26}
	\end{align}
	Using (\ref{427E3}), we know $L^{\ijbar}\Psi_{\ijbar}\geq C(\MG) \eqe_0$.  So, combining (\ref{427E25}) and (\ref{427E26}) we know when $\frac{B}{A}$ is large enough  (\ref{427E16}) is valid.
	To make (\ref{427E17}) valid, we only need to make $A$ and $B$ both large enough, since $(\Phi_s)^2$ is a second order small quantity around $0$. 
	
	Summing up, we find a uniform estimate on $|\Phi|_{C^2}$ as  $\eqe$ goes to zero.
\end{proof}
\section{Deformation}
\label{sec:Appendix_How_to_Deform}
In this section, we explain how to deform a \scconvex\  domain with a smooth boundary to a Euclidean ball. We need to prove the following Lemma:
\begin{lemma}
	[Shrinking to a Standard Ball]
	\label{lem:Deform_a_strongly_C-convex_domain_to_Ball}
	Suppose that $E_1$ is a \scconvex\  domain with a smooth boundary. Then we can find a family of \scconvex\  domains $E_t$ with smooth boundaries and with $t\in[0,1]$, so that they satisfy the following conditions:
	\begin{itemize}
		\item[ 1.] $E_0$ is a Euclidean ball of small radius contained in $E_1$.
		\item[ 2.]  $\overline{E_s}\subset E_t$, for any $s<t$.
		\item[ 3.]  $E_t$ smoothly deforms from $E_1$ to $E_0$. 
	\end{itemize}
\end{lemma}
\begin{remark}
	\label{remark:smoothly_deform}
	For the third requirement, we mean the following: Let $\po$ be a point on $\partial \Omega_s$, for $s\in[0,1]$. Then around $\po$, $\partial \Omega_t$ is the graph of a function $\rho_t$, for $t$ close enough to $s$, and all the derivatives of $\rho_t$ are continuous functions of $t$.
\end{remark}
\begin{proof}
	[Proof of Lemma \ref{lem:Deform_a_strongly_C-convex_domain_to_Ball}] {\bf The following idea is provided to the author by L\'asz\'o Lempert through a private communication.}
	
	Given $E_1$, we can construct the pluricomplex Green's function $G$ of  $E_1$ with a logarithmic singularity at $0\in E_1$ \cite{Lempert84Bulgar}.  Then $G$ is smooth with $\nabla G\neq 0$ everywhere in $E_1\backslash \{0\}$, and all the level sets of $G$ are smooth and \scconvex, according to Lemma 5.3 of \cite{Lempert85_Duke_Symmetries}. 
	
	Let $\mu: {E_1^{\ast}}\rightarrow E_1 $ be the blowing up of $E_1$ at $0$, where
	\begin{align}
		{E_1^{\ast}}=
		\left\{
		(z_1, \ ...\ , z_n, \zeta)\ |\ (z_1, \ ...\ , z_n)\in E_1, \ \zeta=[\zeta_1,\ ...\ , \zeta_n]\in \EC P^{n-1},\ \text{with } z_i\zeta_j=\zeta_i z_j
		\right\}
	\end{align} and 
	\begin{align}
		\mu 	(z_1, \ ...\ , z_n, \zeta)=	(z_1, \ ...\ , z_n).
	\end{align}
	Denote $G^\ast=\mu^\ast G$; then it was proved in \cite{LempertDirichlet} that $e^{2G^\ast}$ is a smooth function on ${E_1^{\ast}}$. Here we assume that
	\begin{align}
		G(z)-\log(|z|)=O(1) \text{\ \  as } z\rightarrow 0.
	\end{align} Moreover, we have that
	\begin{align}
		H(z, \zeta)\triangleq \frac{1}{|z|^2} e^{2 G^\ast}
	\end{align} is smooth on ${E_1^{\ast}}$ and $H>0$. 
	Denote $\log(H)=2h$; then $h$ is smooth on ${E_1^{\ast}}$, and
	\begin{align}
		G^\ast=h+\log(|z|).
	\end{align}We will show
	\begin{align}
		\frac{1}{c} 
		\left\{
		G<\log c
		\right\}
	\end{align} converges to the indicatrix introduced in \cite{Lempert88_Annuals_Invariants}.
	Since 
	\begin{align}
		\left\{
		G<\log c
		\right\} =\mu \left\{
		G^\ast<\log c
		\right\},
	\end{align} we only need to study the level sets of $G^\ast$. We have
	\begin{align}
	\frac{1}{c}	\{G^\ast <\log(c)\}=\frac{1}{c}\left\{
		h(z, \zeta)+\log(|z|)<\log(c)
		\right\} =\left\{
		h(c z, \zeta)+\log(|z|)<0
		\right\}.
	\end{align} Here $\frac{1}{c}$ means rescaling in  $z$ direction. So $\frac{1}{c}\left\{
	G^\ast<\log c
	\right\}$ converges to 
	\begin{align}
		\left\{
		h(0, \zeta)+\log(|z|)<0
		\right\}=
		\left\{
		|z|^2\leq \frac{1}{H(0, \zeta)}
		\right\}.
	\end{align} Let $f_\zeta$ be an extremal disc so that  
	\begin{align}
		f_{\zeta}(0)=0 \text{ \ \ \ \ and \ \ \ \ } \pi f_{\zeta}'(0)=\zeta,
	\end{align}where $\pi$ is the projection from $\EC^n\backslash\{0\}$ to $\EC P^{n-1}$;
	then 
	\begin{align}
		\frac{1}{H(0, \zeta)}=\lim_{z\rightarrow 0}\frac{|z|^2}{e^{2 G^\ast}(z, \zeta)}=
		\lim_{\tau\rightarrow 0}\frac{|f_\zeta(\tau)|^2}{|\tau|^2}=|f_{\zeta}'(0)|^2.
	\end{align}
	So 
	$\mu \left\{
	|z|^2\leq \frac{1}{H(0, \zeta)} 
	\right\}$ is the indicatrix. 
	
	By slightly generalizing the argument on Page 47 of \cite{Lempert88_Annuals_Invariants},  we know the indicatrix for a \scconvex\ domain is strongly convex.
	So when  $\theta$ is large enough, $\{G\leq-\theta\}$ is a strongly convex domain. We choose $B_{r_0}\subset \{G\leq-\theta\}$, then 
	$\{G\leq-\theta\}$ is connected with $B_{r_0}$ by
	the following family of strongly convex domains
	\begin{align}
		s\{G\leq-\theta\}+(1-s)B_{r_0}.
	\end{align}
	Then  the family 
	\begin{align}
		E^t=\left\{
		\begin{array}{cc}
			\{G\leq -2\theta (1-t)\}& \text{ for }t\in [\frac{1}{2}, 1] ,   \\
			2 t \{G\leq -\theta\}+(1-2t) B_{r_0}& \text{ for }t\in [0,\frac{1}{2}] 
		\end{array}
		\right.
	\end{align}satisfies the condition of Lemma \ref{lem:Deform_a_strongly_C-convex_domain_to_Ball}.
\end{proof}
\section{Simplified Proof in Dimension $2$}
\label{sec:Formal_Computation_2dim}
In this appendix, we provide an estimate for the robustness of \qcconvexity\  in the case of $n=2$, under an unrealistic assumption that solutions to Problem \ref{prob:Main_Problem_HCMA_ring} are smooth. But the computations and arguments in this case are much simpler and straightforward, and they demonstrates the main idea of this paper. That is the purpose of this appendix.
    We will prove the following apriori estimate:
    \begin{proposition}
    	[Apriori Estimate in $2$-dimension with the Smoothness Assumption]
    	\label{prop:2d_apriori_estimate_with_smooth_Assumption}
    	Suppose $\Phi$ is a smooth solution to Problem \ref{prob:Main_Problem_HCMA_ring} in the case of $n=2$, and $\Phi$ is \sqcconvex. Then the robustness of the \qcconvexity\ of $\Phi$ is greater than $\epsilon$, where $\epsilon$ is a constant depending on the geometry of $\MR$.
    \end{proposition}

In Section \ref{sec:A_Foliation}, we explain the foliation structure associated to a \sqcconvex\ solution of Problem \ref{prob:Main_Problem_HCMA_ring}, which is not available when we consider the perturbed problem, Problem \ref{prob:Perturbed_Problem_HessianQuotient}. 
In the proof of the above Proposition, we need to first estimate the gradient; this is  the content of Section \ref{sec:A_Lower Bound Estimate for Norms of Gradients}.
We then introduce the auxiliary function $\cQ_\Phi$ in Section \ref{sec:A_ConstructQ}, which is a simplification of the quantity $\MK$ we constructed in general dimensions.  In Section \ref{sec:A_Equation_for_a_b} and  \ref{sec:A_subharmonic_Q_along_leaf}, we show $\frac{1}{1-\cQ_\Phi}$ is subharmonic along each leaf, which is essential for the convexity estimate. In Section \ref{sec:A_Convexity_Estimate} we proved Proposition \ref{prop:2d_apriori_estimate_with_smooth_Assumption} using a continuity estimate.

In this appendix, we let $\MG=\MG_{\ijbar}dz^i\otimes dz^\jbar$ be a constant coefficient Hermitian metric on $\EC^2$ and 
\begin{align}
	\spaceQ_\epsilon=\left\{q| \text{quadratic pluriharmonic polynomials, } |\nabla q|_\MG\leq \epsilon\right\}.
\end{align}
\subsection{Foliation Structure}
\label{sec:A_Foliation}
Given a solution $\Phi$ to the homogenous complex Monge-Amp\`ere equation in $\EC^n$ with the  rank of $(\Phi_{\ijbar})$ being $n-1$, there is a foliation structure, formed by integrating kernels of $(\Phi_{\ijbar})$. This has been discussed in many works including \cite{Lempert81_La_metrique} \cite{DonaldsonHolomorphicDiscs} \cite{CFH}. In our case (n=2), according to the apriori assumption, given any interior point $\po$, we can choose a coordinate $(\tau, z)$ on $\EC^2$ so that $\Phi_{\zzbar}(\po)\neq 0$. Then we let 
\begin{align}
	Z=\partial _\tau-\frac{\Phitauzbar}{\Phizzbar}\partial_z,
\end{align}
and it satisfies
\begin{align}
	[Z, \overline Z]=0.
\end{align} According to the complex Frobienius Theorem, we can find a holomorphic map $\varphi(\zeta)=(\zeta, h(\zeta)),$ from a disc in $\EC$ to $\EC^2$ such that 
\begin{align}
	\varphi(0)=\po, \text{ and }
	\varphi_\ast(\partial_\zeta)=\partial_\tau-\frac{\Phitauzbar}{\Phizzbar}\partial_z.
	\label{0721C4}
\end{align}At a boundary point the map is from a half disc to $\EC^2.$ The image of $\varphi$ is called a local leaf (or, for simplicity, a leaf), and since through any point we can find a leaf, these leaves form a foliation in $\MR$. In the following, we use $\zeta$ to denote the complex coordinate on a leaf. 

With $\Phizzbar\neq 0$, the homogenous complex Monge-Amp\`ere  equation for $\Phi$ is equivalent to 
\begin{align}
	\Phitautaubar-\frac{\Phitauzbar\Phiztaubar}{\Phizzbar}=0
	\label{equation:A_HCMA_inversed_2d}
\end{align}and also
\begin{align}
	\Phi_{\zeta\zetabar}=0. \label{equation:A_Phizetazetabar0}
\end{align} This says $\Phi$ is a harmonic function when restricted to a leaf.

We apply $\partial_z$ to (\ref{equation:A_HCMA_inversed_2d}) and get
\begin{align}
	\Phi_{z\tautaubar}
	-\frac{\Phi_{z\tau\zbar}\Phiztaubar}{\Phizzbar}
	-\frac{\Phitauzbar\Phi_{zz\taubar}}{\Phizzbar}
	+\frac{\Phitauzbar\Phiztaubar\Phi_{z\zzbar}}{(\Phizzbar)^2}=0.
	\label{equation:A_Phiz_harmonic_leaf_full}
\end{align}
This is equivalent to 
\begin{align}
	\Phi_{z\zeta\zetabar}=0,
	\label{equation:A_Phizzetazetabar=0}
\end{align} which  says $\Phi_z$ is harmonic when restricted to a leaf.

\begin{remark} Let $G$ be the pluricomplex Green's function of a \scconvex\ domain $\Omega$ with a smooth boundary; then, the rank of $\iu\ddbar G$ is $n-1$ and the kernels of $\iu\ddbar G$ form an extremal disc (see  \cite{Lempert81_La_metrique} and \cite{Lempert84Bulgar}). For any point ${\bf z}$ in $\Omega\backslash\{ 0\}$, where $0$ is the singular point of $G$, we can find a holomorphic map $\varphi$ from the unit disc $D$ in $\EC$ to $\Omega$, with
	\begin{align}
	\varphi_\ast\partial_\zeta\in \text { kernel of }\iu\ddbar G,\ \ \ \ \ \text{ in }\Omega\backslash\{0\},
\end{align}
and
\begin{align}
	  \varphi(D)\ni 0, {\bf z}\ \ \ \ \  \text{and}\ \ \ \ \ \varphi(\partial D)\subset \partial\Omega.
	\end{align}
	However, in our situation, we would not have an ``extremal cylinder". An entire leaf may be an infinite strap. So we only consider local leaves. This is the main reason why Problem \ref{prob:Main_Problem_HCMA_ring} does not have a smooth solution in general. For geodesic problem in the space of K\"ahler potentials, we discussed a similar singular phenomenon in \cite{HuIMRN}.
\end{remark}

\subsection{Positive Lower Bound Estimate for Norms of Gradients}
\label{sec:A_Lower Bound Estimate for Norms of Gradients}
In this section, we show the norms of the gradients of $\Phi$ have a positive lower bound. We consider the following quantity:
\begin{align}
	S=\frac{\Phi_\zeta\Phi_\zetabar	}{\MG(\partial_\zeta, \partial_\zetabar)},
\end{align}
where $\zeta$ is the complex coordinate on a leaf and $\MG$ is the Hermitian metric on $\EC$. We know that $S$ is independent of the choice of coordinate on the leaf. Actually, it's the square of the norm of the gradient of the restriction of $\Phi$ to a leaf; as a consequence,
\begin{align}
	|\nabla \Phi|_\MG^2> S.
\end{align}

We will show $S$ has a positive lower bound in $\MR$. First, we estimate $S$ on the boundary; then, we show $\log(S)$ is a superharmonic function on a leaf. The desired estimate then follows from the maximum principle.

When estimating $S$ at a boundary point $\po$, we choose a set of coordinates $(\tau, z)$ on $\EC^2$, so that $ \MG_{\ijbar}=\delta_{ij}$ and $\Phi_z(\po)=0$. Then, at $\po$, 
\begin{align}
	S=\frac{\left|(\partial_\tau-\frac{\Phitauzbar}{\Phizzbar}\partial_z)\Phi\right|^2}
	{\left|\partial_\tau-\frac{\Phitauzbar}{\Phizzbar}\partial_z\right|^2_\MG}
	=\frac{|\Phi_\tau|^2}{1+\left|\frac{\Phitauzbar}{\Phizzbar}\right|^2}.
\end{align}
Therefore, for a constant $C$ depending on the $C^2$ norm of $\Phi$ and the lower bound of $\Phizzbar$, 
\begin{align}
	C|\Phi_{\bold n}|^2\leq S\leq |\Phi_{\bold n}|^2 ,
\end{align}
where ${\bold n}$ is the unit normal vector to $\partial_\MR$ at $\po$. According to Lemma \ref{lemma:Boundary_Convex_implies_Modulus} the lower bound of $\Phizzbar$ depends on the \cconvexity\  of the boundary and the gradient of $\Phi$ on the boundary. The lower bound estimate for $|\Phi_{\bold n}|$ can be find in Appendix \ref{sec:B_boundary_Barrier}. Thus the lower bound for
$S$ on the boundary is known.

To get the interior estimate for $S$, we need to show $\log(S)$ is superharmonic on each leaf. We have
\begin{align}
	(\log(S))_\zetazetabar=\left[\log(\Phi_\zeta\Phi_\zetabar)
	-\log \MG(\partial_\zeta, \partial_{\zetabar})\right]_{\zetazetabar}.
\end{align} 
In the following, we compute $\left[\log(\Phi_\zeta\Phi_\zetabar)\right]_{\zetazetabar}$ and $-\left[\log \MG(\partial_\zeta, \partial_{\zetabar})\right]_{\zetazetabar}$ seperately and show that they are both non-positive.

Using $\Phi_\zetazetabar=0$, which is (\ref{equation:A_Phizetazetabar0}), and $\Phi_{\zeta\zetazetabar}=0$, which follows from the previous one, we have
\begin{align}
	\left[\log(\Phi_\zeta\Phi_\zetabar)\right]_{\zetazetabar}
	=\left[\frac{\Phi_{\zeta\zeta}\Phi_\zetabar+\Phi_\zeta\Phi_{\zetazetabar}}{\Phi_\zeta \Phi_\zetabar}\right]_{\zetabar}
	=\frac{|\Phi_{\zeta\zeta}|^2}{|\Phi_\zeta|^2}-\frac{|\Phi_{\zeta\zeta}   \Phi_\zetabar|^2}{|\Phi_\zeta|^4}=0.
\end{align}
For the computation of $\left[\log \MG(\partial_\zeta, \partial_{\zetabar})\right]_{\zetazetabar}$, we choose a coordinate $(\tau, z)$, so that $g_{\ijbar}=\delta_{ij}$. Then 
\begin{align}
	\MG(\partial_\zeta, \partial_\zetabar)=1+\left|\frac{\Phi_{\tau\zbar}}{\Phi_{\zzbar}}\right|^2.
\end{align}
As a locally defined function $\frac{\Phi_{\tau\zbar}}{\Phi_{\zzbar}}$ is holomorphic on each leaf; actually it is $-h'$, where $h$ is the $\tau$-component of $\varphi$ in (\ref{0721C4}). 
This can be verified by applying $\partial_\zetabar$ to $\frac{\Phi_{\tau\zbar}}{\Phi_{\zzbar}}$. In the following, we denote $\frac{\Phi_{\tau\zbar}}{\Phi_{\zzbar}}$ by $f$, and then 
\begin{align}
	\MG(\partial_\zeta, \partial_\zetabar)=1+|f|^2;
\end{align}
we have
\begin{align}
	[\log (1+|f|^2)]_{\zeta\zetabar}=\left(\frac{f_\zeta \overline f}{1+|f|^2}\right)_\zetabar=\frac{\overline f_\zetabar f_\zeta}{(1+|f|^2)^2}\geq 0.
\end{align}
To sum up, we have
\begin{align}
	(\log S)_\zetazetabar\leq 0, 
\end{align} and combining with the boundary estimate, we have a positive lower bound for $S$.
Therefore, $ |\nabla \Phi|_\MG^2$ is also bounded from below by a positive constant since it's greater or equal to $S$. We denote
\begin{align}
	|\nabla \Phi|_\MG \geq C_{\MR}^g.
\end{align}
\begin{remark}A positive lower bound estimate for $|\nabla\Phi|_{\MG}$ is crucial in the study of level sets of $\Phi$; without it, we don't even know the level sets of $\Phi$ are submanifolds.   	It's interesting to note the following results. In \cite{MaOuZhang_CPAM}, Proposition 4.1, they proved that for the solution $u$ of the $p$-Laplace equation  $|\nabla u|$ strictly increases in the direction of $\nabla u$. In \cite{Lempert81_La_metrique}, for an extremal map $f: U\rightarrow D$, where $U$ is the unit disc in $\EC$ and $D$ is a strongly convex domain with a smooth boundary, it was proved that $f'\neq 0$ anywhere; as a consequence, for the pluricomplex Green's function $\Phi$, $\nabla\Phi\neq0$.
\end{remark}
\subsection{The Construction of $\cQ$ and Its Relation to $\EC$-Convexity}
\label{sec:A_ConstructQ}
Under the apriori assumption that $\Phi$ is \sqcconvex, the gradients of $\Phi$ never vanish; then, for any point $\po$,  we can choose a  coordinate $(\tau, z)$ so that $\Phi_\tau\neq0$, $\Phi_{z}=0$ and $\po=(0,0)$. That's to say, the plane $\{\tau=0\}$ is the complex tangent plane of the level set of $\Phi$ at the point $\po$. In such a coordinate chart, we define
\begin{align}
&	\cbphi=\Phizz-2\frac{\Phi_z}{\Phi_\tau}\Phitauz+\left(\frac{\Phi_z}{\Phi_\tau}\right)^2\Phitautau,
\label{equation:A_define_cbPhi}
\\
&	\caphi=\Phizzbar-\frac{\Phi_z}{\Phi_\tau}\Phitauzbar
							-\frac{\Phi_\zbar}{\Phi_\taubar}\Phiztaubar
	+\left|\frac{\Phi_z}{\Phi_\tau}\right|^2\Phitautaubar.
	\label{equation:A_define_caPhi}
\end{align} Because of the apriori assumption that $\Phi$ is \sqcconvex, when this neighborhood is small enough, $\caphi\neq 0$ and we can define
\begin{align}
	\cQphi=\frac{|\cbphi|^2}{\caphi^2}.
\end{align} Actually,  $\cQphi$ is invariant under a complex linear change of coordinate, even $\caphi$ and $\cbphi$ depend on the choice of coordinates. $\Phi$ is \sqcconvex\  is equivalent to that $\ca_\Phi>0$ and $\cQphi<1$. In Section \ref{sec:A_subharmonic_Q_along_leaf}, we will show $\frac{1}{1-\cQphi}$ is subharmonic along each leaf.

We also need to perturb $\Phi$ by pluriharmonic quadratic polynomials, which does not change the equation but will alter $\cQphi$. This gives us the estimate of the robustness of \qcconvexity\  of $\Phi$. For a quadratic polynomial $q$, assuming that $\Phi-q$ is \sqcconvex\  and $\partial_\tau(\Phi-q)\neq0$, we replace $\Phi$ by $\Phi-q$ in (\ref{equation:A_define_cbPhi}) and (\ref{equation:A_define_caPhi}) and let
\begin{align}
&	\cb_{\Phi-q}=(\Phi-q)_{zz}-2\frac{(\Phi-q)_z}{(\Phi-q)_\tau}(\Phi-q)_{\tau z}+\left(\frac{(\Phi-q)_z}{(\Phi-q)_\tau}\right)^2(\Phi-q)_{\tau \tau},\\
&	\ca_{\Phi-q}=(\Phi-q)_{\zzbar}-\frac{(\Phi-q)_z}{(\Phi-q)_\tau}(\Phi-q)_{\tau\zbar}
	-\frac{(\Phi-q)_\zbar}{(\Phi-q)_\taubar}(\Phi-q)_{z\taubar}
	+\left|\frac{(\Phi-q)_z}{(\Phi-q)_\tau}\right|^2(\Phi-q)_{\tautaubar}.
\end{align}
Then we let
\begin{align}
	\cQ_{\Phi-q}=\frac{|\cb_{\Phi-q}|^2}{\ca_{\Phi-q}^2}.
\end{align}
If $q$ is a pluriharmonic function, $\Phi-q$ is still a solution to the homogenous complex Monge-Amp\`ere equation, and the subharmonicity of $\cQ_\Phi$ implies the subharmonicity of $\cQ_{\Phi-q}$.

\subsection{Equations for $\ca$ and $\cb$ with a Good Coordinate Choice}
\label{sec:A_Equation_for_a_b}
At any point $\po$, because of the assumption that $\Phi$ is \sqcconvex, we have
\begin{align}
	\ker \sqrt{-1}\partial\overline{\partial}\Phi \oplus \ker \partial\Phi=\EC^2,
\end{align}
so we can choose a  coordinate $(\tau, z)$ so that $\Phi_z=0$ and $\Phi_{\tautaubar}=0$. Then the homogenous complex Monge-Amp\`ere equation implies $\Phiztaubar=0$.
We note that with this coordinate the coefficients of $\MG$ may be very large.

Then, the computations at $\po$ become simple:
  (\ref{equation:A_HCMA_inversed_2d}) becomes $\Phi_{\tau\taubar}=0$, and (\ref{equation:A_Phiz_harmonic_leaf_full}) becomes $\Phi_{z\tautaubar}=0.$
For any function $f$ defined around $\po$, we have $f_\zeta=f_\tau$, and \begin{align}
	f_{\zeta\zetabar}
	=\left(f_{\tau}-\frac{\Phitauzbar}{\Phizzbar}f_{z}\right)_\taubar
	=f_{\tautaubar}-\frac{\Phi_{\tau\zbar\taubar}}{\Phizzbar}f_{z}=f_{\tautaubar}.
\end{align}
In the last equality, we used $\Phi_{z\tautaubar}=\Phi_{\zbar\tautaubar}=0$.

In the following, we discuss the equations satisfied by $\ca_\Phi$ and $\cb_\Phi$ at $\po$ with the coordinate chosen above. We will drop the subindex $\Phi$ and just denote them by $\ca$ and $\cb$. We compute $\ca_{\zeta\zetabar}$, $\cb_{\zeta\zetabar}$, $\ca_{\zeta}$  and $\cb_{\zetabar}$ and show that they satisfy the following equations:
\begin{align}
	\cb_{\zeta\zetabar}&=2\frac{\cb_\zetabar \ca_\zeta}{\ca},
	\label{equation:A_cb_Phi}\\
	\ca_{\zeta\zetabar}&=\frac{\cb_\zetabar {\overline\cb}_\zeta}{\ca}
	   +\frac{\ca_\zeta\ca_\zetabar}{\ca}.
	\label{equation:A_ca_Phi}
\end{align}
Using these equations, we show, in Section \ref{sec:A_subharmonic_Q_along_leaf}, $\frac{1}{1-\cQ_\Phi}$ is subharmonic along any leaves.

To simplify the computation, let $b=\Phi_{zz}$ and $a=\Phi_{\zzbar}$. By applying $\partial_\zzbar$ and $\partial_{zz}$ to (\ref{equation:A_HCMA_inversed_2d}), we have 
\begin{align}
	b_{\zeta\zetabar}&=2\frac{b_\zetabar a_\zeta}{a},
	\label{equation:A_b_Phi}\\
	a_{\zeta\zetabar}&=\frac{b_\zetabar {\overline b}_\zeta}{a}+\frac{a_\zetabar a_\zeta}{a}.
		\label{equation:A_a_Phi}
\end{align}
The computation is the same as that of Section 2 of \cite{Hu22Nov} and Appendix A of \cite{HuC2Perturb}. In the following, we will show, at the point $\po$, $\cb_\zetabar=b_\zetabar$,  $\ca_\zeta=a_\zeta$, $\cb_{\zeta\zetabar}=b_{\zeta\zetabar}$, $\ca_\zetazetabar=a_\zetazetabar$. Thus, (\ref{equation:A_cb_Phi}) and (\ref{equation:A_ca_Phi}) follow from (\ref{equation:A_b_Phi}) and (\ref{equation:A_a_Phi}).

Subtracting $a=\Phi_{\zzbar}$ from (\ref{equation:A_define_caPhi}), we get
\begin{align}
	\ca_\zeta-a_\zeta=\left(-\frac{\Phi_z}{\Phi_\tau}\Phitauzbar
								-\frac{\Phi_\zbar}{\Phi_\taubar}\Phiztaubar
								+\left|\frac{\Phi_z}{\Phi_\tau}\right|^2\Phitautaubar\right)_\zeta=0, \ \ \ \ \ \ \text{ at } p.\label{414C19}
\end{align}
It's zero because $\Phi_{z\taubar}$, $\Phi_{\tau \zbar}$ and $\Phi_z$ all vanish at the point $\po$.
Similarly, we subtract $b=\Phi_{zz}$ from (\ref{equation:A_define_cbPhi}) and get 
\begin{align}
	\cb_\zetabar-b_\zetabar=\left(-2\frac{\Phi_z}{\Phi_\tau}\Phitauz+\left(\frac{\Phi_z}{\Phi_\tau}\right)^2\Phitautau\right)_\zetabar.
	\label{414_C20}
\end{align}
On the right-hand side of (\ref{414_C20}), the only non-zero terms are produced by $\partial_\zetabar$ acting on $\Phi_z$ because $\Phi_z$ vanishes at $\po$. However,
$\partial_{\zetabar}\Phi_z=\Phi_{z\taubar}=0$, at $\po$, so $\cb_\zetabar=b_\zetabar$. Applying $\partial_\zetabar$ to 
(\ref{414C19}),
we get
\begin{align} 	\ca_{\zeta\zetabar}-a_{\zeta\zetabar}=
	\left(-\frac{\Phi_z}{\Phi_\tau}\Phitauzbar
	-\frac{\Phi_\zbar}{\Phi_\taubar}\Phiztaubar
	+\left|\frac{\Phi_z}{\Phi_\tau}\right|^2\Phitautaubar\right)_\zetazetabar.
\end{align}
On the right-hand side of the equation above, $\Phi_z$,  $\Phiztaubar$ and $\Phitautaubar$ all vanish at $\po$, so
\begin{align}
		\left(\left|\frac{\Phi_z}{\Phi_\tau}\right|^2\Phitautaubar\right)_\zetazetabar
		=0,
\end{align}
since the item in the bracket is a third order small quantity around $\po$, and
\begin{align}
	\ca_{\zeta\zetabar}-a_{\zeta\zetabar}=
	-\frac{\Phi_{z\tau}}{\Phi_\tau}\Phi_{\tau\zbar\taubar}
		-\frac{\Phi_{z\taubar}}{\Phi_\tau}\Phi_{\tau\zbar\tau}
	-\frac{\Phi_{\zbar\tau}}{\Phi_\taubar}\Phi_{z\taubar\taubar}
	-\frac{\Phi_{\zbar\taubar}}{\Phi_\taubar}\Phi_{z\taubar\tau}.
	\label{414C23}
\end{align}
In (\ref{414C23}), $\Phi_{z\taubar}$ and $\Phi_{z\tau\taubar}$ all vanish at $\po$, so $	\ca_{\zeta\zetabar}=a_{\zeta\zetabar}$.
Applying $\partial_\zeta$ to (\ref{414_C20}) gives
\begin{align}
	\cb_{\zeta\zetabar}-b_{\zeta\zetabar}=\left(-2\frac{\Phi_z}{\Phi_\tau}\Phitauz+\left(\frac{\Phi_z}{\Phi_\tau}\right)^2\Phitautau\right)_{\zeta\zetabar}.
\end{align}
It vanishes at $\po$ since $\Phi_z$, $\Phiztaubar$ and $\Phi_{z\tautaubar}$ all vanish at $\po$, similar to the previous argument.

Therefore, (\ref{equation:A_cb_Phi}) and (\ref{equation:A_ca_Phi}) are valid at the point $\po$. However, we point out that they depend on the choice of coordinates, but in the next section, using (\ref{equation:A_cb_Phi}) and (\ref{equation:A_ca_Phi}), we will show $\frac{1}{1-\cQ_\Phi}$ is subharmonic along a leaf, which is a fact independent of the choice of coordinates.
\subsection{Subharmonicity of $\frac{1}{1-\cQ_\Phi}$ Along Leaves} 
\label{sec:A_subharmonic_Q_along_leaf}

In this section, we will show $\frac{1}{1-\cQ_\Phi}$ is subharmonic along a leaf when $\cQ_\Phi<1$. We will drop the subindex $\Phi$ and denote $\cQ_\Phi$ by $\cQ$ since we don't consider  perturbations of $\Phi$ in this section. We will continue to use the coordinate introduced in Section \ref{sec:A_Equation_for_a_b}, where $\partial_\zeta=\partial_\tau$ and $\partial_{\zetazetabar}=\partial_{\tautaubar}$ at $\po$.

We directly apply $\partial_{\tautaubar}$ to $\frac{1}{1-\cQ}$ and get
\begin{align}
	\left(\frac{1}{1-\cQ}\right)_{\tau\taubar}=\left[\frac{1}{(1-\cQ)^2}\cQ_\tau\right]_{\taubar}=\frac{1}{(1-\cQ)^3}\left[2\cQ_\tau\cQ_\taubar +\cQ_{\tautaubar}(1-\cQ)\right].
	\label{420C27}
\end{align}
To simplify the computation, we introduce a quantity:
\begin{align}
	\beta=\cb_\tau-2\frac{\cb  \ca_\tau}{\ca}.
\end{align} With $\beta$, we have
\begin{align}
	\left(\frac{\cb}{\ca^2}\right)_\tau=\frac{\beta}{\ca^2} \label{equation:A_introducing_beta}
\end{align}
and
\begin{align}
	\beta_\taubar=\underbrace{\cb_\tautaubar-2\frac{\cb_\taubar \ca_\tau}{\ca}}_{=0,\ \text{ by (\ref{equation:A_cb_Phi})}}\underbrace{-2\frac{\ca_\tautaubar \cb}{\ca}
	+2\frac{ \ca_\tau \ca_\taubar \cb}{\ca^2}=-2\frac{ \overline{\cb}_\tau \cb_\taubar \cb}{\ca^2}}_{\text{by (\ref{equation:A_ca_Phi})}}.
\label{414c29}
\end{align}
In addition, we observe that (\ref{equation:A_cb_Phi}) is equivalent to 
\begin{align}
	\left(\frac{\cb_\taubar}{\ca^2}\right)_\tau=0.
	\label{equation:A_cb_simplifed}
\end{align}
Then
\begin{align}
	\cQ_\tau=\left(\frac{\cb}{\ca^2}\cdot \overline\cb\right)_\tau
	=\left(\frac{\cb}{\ca^2}\right)_\tau \overline\cb+\frac{\cb}{\ca^2}\cdot \overline\cb_\tau=\frac{\beta}{\ca^2}\overline\cb+\frac{\cb}{\ca^2}\cdot \overline\cb_\tau, \label{414c31}
\end{align}
and
\begin{align}
	\cQ_\tau\cQ_\taubar=\left(\frac{\beta}{\ca}, \frac{\overline{\cb}_\tau}{\ca}\right) 
\left(
\begin{array}{cc}
	\cQ&\frac{\overline{\cb}^2}{\ca^2}\\
	\ &\ \\
	\frac{\cb^2}{\ca^2}&\cQ
\end{array}
\right)	
\left(
\begin{array}{c}
	\frac{\overline\beta}{\ca}\\  \\  \frac{\cb_\taubar}{\ca}
\end{array}
\right).\label{414C32}
\end{align}
For $\cQ_{\tautaubar}$, we also make it into a quadratic form of $\frac{\beta}{\ca}$ and $\frac{\overline\cb_\tau}{\ca}$.
Applying $\partial_\taubar$ to (\ref{414c31}), we have
\begin{align}
	\cQ_{\tautaubar}=
	\left( \beta\cdot \frac{\overline \cb}{\ca^2}+\cb \cdot \frac{\overline \cb_\tau}{\ca^2}
	\right)_\taubar= 
	\beta_\taubar\cdot \frac{\overline \cb}{\ca^2}
	+\beta\cdot\left( \frac{\overline \cb}{\ca^2}\right)_\taubar
	+\cb_\taubar \cdot \frac{\overline \cb_\tau}{\ca^2}
	+\cb \cdot \left(\frac{\overline \cb_\tau}{\ca^2}\right)_\taubar.
	\label{414C33}
\end{align}
Using the conjugate of (\ref{equation:A_cb_simplifed}), the last term of (\ref{414C33}) is zero; for  $\left( \frac{\overline \cb}{\ca^2}\right)_\taubar$\ , we use the conjugate of (\ref{equation:A_introducing_beta}); for $\beta_\taubar$, we use the equation (\ref{414c29}). Then we have
\begin{align}
	\cQ_{\tautaubar}=
	\left(\frac{\beta}{\ca}, \frac{\overline{\cb}_\tau}{\ca}\right) 
	\left(
	\begin{array}{cc}
		1&\ \\
		\ &1-2\cQ
	\end{array}
	\right)	
	\left(
	\begin{array}{c}
		\frac{\overline\beta}{\ca}\\ \\  \frac{\cb_\taubar}{\ca}
	\end{array}
	\right).\label{414C34}
\end{align}
Combining (\ref{414C32}) (\ref{414C34}) with (\ref{420C27}), we know
\begin{align}
	\left(\frac{1}{1-\cQ}
	\right)_\tautaubar=\frac{1}{(1-\cQ)^3} 
	\left(\frac{\beta}{\ca}, \frac{\overline{\cb}_\tau}{\ca}\right) 
M
	\left(
		\frac{\overline\beta}{\ca},\frac{\cb_\taubar}{\ca}
	\right)',
\end{align}
where $M$ is a $2\times 2$ matrix:
\begin{align}
	M=
		\left(
	\begin{array}{cc}
		1+\cQ&2\frac{\overline \cb^2}{\ca^2} \\
		\ \\
		2\frac{ \cb^2}{\ca^2}   &1-\cQ+2\cQ^2
	\end{array}
	\right)	
\end{align}
Direct computation gives 
\begin{align}
	\det(M)&=(1-\cQ)^2(1+2\cQ),\\
	\tr(M)&=2+2\cQ^2.
\end{align} By the definition of $\cQ$, it's always nonnegative, so $1+2\cQ>0$.
Therefore, $	\left(\frac{1}{1-\cQ}
\right)_\tautaubar\geq 0$, at $\po$, providing $\cQ<1$. 
We can choose $\po$ to be any point in $\MR$, so $\cQ$, as a globally defined function in $\MR$, satisfies
\begin{align}
		\left(\frac{1}{1-\cQ}
	\right)_{\zetazetabar}\geq 0,
\end{align}
along a leaf, providing $\cQ<1$.
Then, 
\begin{align}
	\frac{1}{1-\cQ}\leq \max_{\partial_\MR}	\frac{1}{1-\cQ} \ \ \ \ \ \ \text{ in }\MR\text{, \ providing }\cQ<1 \text{ in }\overline{\MR}.
	\label{420230720}
\end{align}This is because the maximum of $\frac{1}{1-\cQ}$ cannot be achieved at any interior point in $\MR$. Suppose $\frac{1}{1-\cQ}$ achieves its maximum at $\po\in\MR$.
Let $\ML$ be a local leaf passing $\po$. Then  the restriction of $\frac{1}{1-\cQ}$ on $\ML$ also achieves a local maximum at $\po$, but this contradicts with the strong maximum principle. It follows that
\begin{align}
	\cQ\leq \max_{\partial_\MR}\cQ \ \ \ \ \ \ \text{ in }\MR\text{, \ providing }\cQ<1 \text{ in }\overline{\MR}.
	\label{420C41}
\end{align}

\subsection{Convexity Estimate}
\label{sec:A_Convexity_Estimate}
In this section we will show for a small constant $\epsilon_0$, $\Phi-q$ is \sqcconvex \ for any $q\in \spaceQ_{\epsilon_0}$. The constant $\epsilon_0$ will depend on the geometry of $\MR$. In this section, we only do unitary change of coordinate, so we assume the metric is $\sum_i dz^i\otimes \overline{dz^i}$.

Firstly, we need to choose $\epsilon_0$ small enough, so that it's smaller than the lower bound estimate for $|\nabla\Phi|_\MG$ in Section \ref{sec:A_Lower Bound Estimate for Norms of Gradients}: we choose
\begin{align}
	\epsilon_0< C_{\MR}^g. \label{eq:A_epsilon_small_than_Gradient}
\end{align}
This implies $\nabla(\Phi-q)$ is non-zero for $q\in \spaceQ_{\epsilon_0}$ and the level sets of $\Phi$ are smooth. Then we  make $\epsilon_0$ smaller so that 
\begin{align}
	\Phi-q \text{ is \sqcconvex},\text{ on }\partial\MR \text{ for any }q\in \spaceQ_{\epsilon_0}.
	\label{eq:A_Q<1Boundary}
\end{align}
Then using a continuity argument, we 
 will prove
$\Phi-q$ is actually \sqcconvex\   in $\MR$ for any $q\in\spaceQ_{\epsilon_0}$.
 Let 
\begin{align}
	I_{\epsilon_0}=\left\{
	\epsilon\in [0,\epsilon_0] |
	 \text{ for any }q\in \spaceQ_\epsilon, \Phi-q
	  \text{ is \sqcconvex}
	    \right\}.
	    \label{420C56}
\end{align}
We will show $I_{\epsilon_0}$ is non-empty, open and right-closed; thus, $I_{\epsilon_0}=[0,\epsilon_0]$.

It's obvious that $I_{\epsilon_0}$ is non-empty since it contains $0$; this is because of the apriori assumption. The openness can also be proved easily:  if $\Phi-q_1$ is \sqcconvex, then we have 
$\ca_{\Phi-q_1}>0$; so, $\cQ_{\Phi-q_1-q_2}$ is close to $\cQ_{\Phi-q_1}$ for $q_2\in \spaceQ_{\epsilon_2}$ with small enough $\epsilon_2$. 

The right-closeness is more complicated, and we need to use the subharmonicity of $\frac{1}{1-\cQ_{\Phi-q}}$. Here, we notice that to derive (\ref{420C41}) we only need that $\Phi$ is a solution to the homogenous complex Monge-Amp\`ere equation and that $\Phi$ is \sqcconvex. For a pluriharmonic function $q$, $\Phi-q$ is still the solution to the homogenous complex Monge-Amp\`ere equation; so, if $\Phi-q$ is also \sqcconvex, we have 
\begin{align}
	\cQ_{\Phi-q}\leq \max_{\partial_\MR}\cQ_{\Phi-q} \ \ \ \ \ \ \text{ in }\MR.
	\label{420C57}
\end{align}
Suppose $[0,\epsilon_1)\subset I_{\epsilon_0}$ and $\epsilon_1\notin I_{\epsilon_0}$. Then we will derive a contradiction. According to the definition of $I_{\epsilon_0}$ (\ref{420C56}), there is a $k\in \spaceQ_{\epsilon_1}$ so that $\Phi-k$ is not \sqcconvex. This means we can find a point $\po\in\MR$ so that if we choose a  coordinate $(\tau, z)$ around $\po$ such that $(\Phi-k)_z(\po)=0$, then either
\begin{align}
	(\Phi-k)_{\zzbar}(\po)=\Phi_{\zzbar}(\po)=0
	\label{Possibility_metric_Dege}
\end{align} or 
\begin{align}
	\cQ_{\Phi-k}(\po)=1.
	\label{Possibility_Q}
\end{align} In the following we will argue none of these can happen.

Let $h$ be an element of $\spaceQ_{\frac{\epsilon_1}{2}}$ . For $\theta\in (0,1]$ we consider the following polynomials
\begin{align}
	k^\theta=(1-\theta)k-\theta h\ \ \ \ \ \text{and}\ \ \ \ \ 
	k_\theta=(1-\theta)k+\theta h.
\end{align}They all belong to the interior of $\spaceQ_{\epsilon_1}$. So
\begin{align}
	\Phi^\theta=\Phi-k^\theta\ \ \ \ \ \text{and}\ \ \ \ \ 
    \Phi_\theta=\Phi-k_\theta
\end{align}are both \sqcconvex.
Then using a continuity argument, we have
\begin{align}
	\cQ_{\Phi^\theta}, \cQ_{\Phi_\theta} \leq \max_{q\in \spaceQ_{\epsilon_0}}\max_{\partial \MR}\cQ_{\Phi-q}<1\ \ \ \ \ \ \text{for }\theta\in(0,1].
	\label{421C62}
\end{align}
Because of (\ref{eq:A_epsilon_small_than_Gradient}), we know the gradient of $\Phi-k$ does not vanish anywhere. We denote $(\Phi-k)_\tau(\po)=\nu$; then $\nu>0$. The gradients of $\Phi^\theta$ and $\Phi_\theta$ are both close to the gradient of $\Phi-k$, providing $\theta$ very small. We choose $h$ so that $\nabla h(\po)=0$; then, we have
\begin{align}
&	\partial_\tau\Phi^\theta=	\partial_\tau\Phi_\theta  =(\Phi-k)_\tau+\theta(k+h)_{\tau}=\nu+\theta k_\tau,\\
	&	\partial_z\Phi^\theta=	\partial_z\Phi_\theta  =(\Phi-k)_z+\theta(k+h)_z=\theta k_z,
\end{align} at the point $\po$.
 Let 
 \begin{align}
 	\slopev=\frac{\theta k_z(\po)}{\nu+\theta k_\tau(\po)};
 \end{align}then $\partial_z-\slopev\partial_\tau$ is a complex tangential vector to the level sets of $\Phi-k$, with $\slopev=O(\theta)$.  We compute $\cQ$ in the complex tangential direction and get
\begin{align}
	\cQ_{\Phi^\theta}(\po)=\left|\frac{K_\theta+\theta(h_{zz}+2 \slopev h_{z\tau}+(\slopev)^2h_{\tau\tau})}
	{\Phi_{\zzbar}-\slopev \Phi_{\tau\zbar}-\overline{\slopev}\Phi_{z\taubar}+|\slopev|^2\Phi_{\tautaubar}}\right|^2,\\
		\cQ_{\Phi^\theta}(\po)=\left|\frac{K_\theta-\theta(h_{zz}+2 \slopev h_{z\tau}+(\slopev)^2h_{\tau\tau})}
	{\Phi_{\zzbar}-\slopev \Phi_{\tau\zbar}-\overline{\slopev}\Phi_{z\taubar}+|\slopev|^2\Phi_{\tautaubar}}\right|^2,
\end{align}
where we denote
\begin{align}
	(\Phi-(1-\theta)k)_{zz}-2(\Phi-(1-\theta)k)_{z\tau} \slopev +(\Phi-(1-\theta)k)_{\tau\tau}(\slopev)^2=K_\theta.
\end{align} Here $k$ and $h$ don't appear in the denominator because they are both pluriharmonic. Equations above imply
\begin{align}
	\lim_{\theta\rightarrow 0}\cQ_{\Phi^\theta}=\lim_{\theta\rightarrow 0}\cQ_{\Phi^\theta}=\cQ_{\Phi-k};
\end{align}therefore, (\ref{Possibility_Q}) is impossible.

(\ref{421C62}) implies 
\begin{align}
	\sqrt{	\cQ_{\Phi^\theta}(\po)}+	\sqrt{	\cQ_{\Phi_\theta}(\po)}\leq 2, \ \ \ \ \ \ \text{ for }\theta\in(0,1].
\end{align} The triangle inequality implies
\begin{align}
	\left|\frac{\theta(h_{zz}+2 \slopev h_{z\tau}+(\slopev)^2h_{\tau\tau})}
	{\Phi_{\zzbar}-\slopev \Phi_{\tau\zbar}-\overline{\slopev}\Phi_{z\taubar}+|\slopev|^2\Phi_{\tautaubar}}\right|\leq 1 \text{\ \ \ \ \ at $\po$}.
	\label{0511D70}
\end{align} If (\ref{Possibility_metric_Dege}) is true, then the subharmonicity of $\Phi$  implies \begin{align}
	\Phi_{z\taubar}=\Phi_{\tau\zbar}=0,\ \ \ \ \  \ \ \ \text{ at }\po.
\end{align}
So we get
\begin{align}
	\frac{\theta|h_{zz}(\po)|+O(\theta^2)}{O(\theta^2)}\leq 2
\end{align} when $\theta$ is small enough. This is a contradiction since we can choose $h$ so that $h_{zz}(\po)\neq 0$
 Here we used $\Phi_{\zzbar}=\Phi_{z\taubar}=\Phi_{\tau\zbar}=0$ at $\po$  and $\slopev=O(\theta)$. 

 So we have proved that $\Phi-k$ is \sqcconvex, and as a consequence,  $I_{\epsilon_0}$ is right-closed. Therefore, $I_{\epsilon_0}=[0,\epsilon_0]$. This says the robustness of \qcconvexity \ of $\Phi$ is greater than $\epsilon_0$.

Jingchen Hu\\
jingchenhoo@gmail.com
\end{document}